\documentclass[10pt]{amsart}
\usepackage[USenglish]{babel}
\usepackage{amsmath, amsfonts, amssymb, amsthm, amscd}
\usepackage{bbm, color, csquotes, dsfont, enumitem, epsf, graphicx, indentfirst, latexsym, mathtools, rotating, scalerel, tikz, tikz-cd, verbatim}
\usepackage[final, breaklinks=true, colorlinks, citecolor=green, linkcolor=blue, pdfusetitle]{hyperref}

\usetikzlibrary{decorations.pathmorphing}

\usepackage[T1]{fontenc}

\usepackage[abbrev]{amsrefs}

\usepackage{lmodern}

\newcommand{\bm}[1]{\boldsymbol{#1}}

\newtheorem{theorem}{Theorem}[section]

\newtheorem{proposition}[theorem]{Proposition}

\newtheorem*{mainresult}{Main Result}

\newcommand{\thisaxiomname}{}
\newtheorem*{genericaxiom*}{\thisaxiomname}
	\newenvironment{axiom*}[1]{\renewcommand{\thisaxiomname}{#1}%
	\begin{genericaxiom*}}{\end{genericaxiom*}}

\theoremstyle{definition}
\newtheorem{definition}[theorem]{Definition}

\newtheorem{problem}{Problem}

\theoremstyle{remark}
\newtheorem{remark}[theorem]{Remark}

\numberwithin{equation}{section}

\newcommand{\R}{\mathbb{R}}
\newcommand{\Z}{{\mathbb Z}}

\newcommand{\C}{{\mathbb C}}

\newcommand{\Q}{{\mathbb Q}}

\newcommand{\ep}{\varepsilon}

\newcommand{\cB}{{\mathcal B}}

\newcommand{\cJ}{{\mathcal J}}

\newcommand{\cM}{{\mathcal M}}

\newcommand{\cO}{{\mathcal O}}
\newcommand{\cP}{{\mathcal P}}

\newcommand{\cR}{{\mathcal R}}
\newcommand{\cS}{{\mathcal S}}

\newcommand{\cU}{{\mathcal U}}
\newcommand{\cV}{{\mathcal V}}
\newcommand{\cW}{{\mathcal W}}
\newcommand{\cX}{{\mathcal X}}
\newcommand{\cY}{{\mathcal Y}}
\newcommand{\cZ}{{\mathcal Z}}

\newcommand{\CM}{\overline{\mathcal{M}}}

\newcommand{\ww}{\omega}
\newcommand{\CP}{\mathbb{C}P}

\newcommand{\delbarj}{\overline{\partial}_J}

\newcommand{\delbarjinv}{\smash{\overline{\partial}_J}\vphantom{\partial}^{-1}}

\newcommand{\hb}{
	\mathchoice
	{
		B_{\frac{1}{2}}
	}
	{
		B_{1/2}
	}
	{}{}
}

\newcommand{\abs}[1]{\lvert#1\rvert}

\DeclareMathOperator{\ind}{ind}
\DeclareMathOperator{\dR}{dR}

\newcommand{\red}[1]{\textcolor{red}{#1}}

\DeclareMathOperator*{\GW}{GW}
\DeclareMathOperator{\PD}{PD}

\newcommand{\ssc}{\text{sc}}
\newcommand{\dmspace}{\overline{\mathcal{M}}}
\newcommand{\dmlog}{\smash{\overline{\mathcal{M}}}\vphantom{\mathcal{M}}^{\text{log}}}
\newcommand{\dmexp}{\smash{\overline{\mathcal{M}}}\vphantom{\mathcal{M}}^{\text{exp}}}

\newcommand{\av}{\text{av}}

\newcommand{\supp}{\operatorname{supp}}
\newcommand{\id}{\operatorname{id}}
\newcommand{\pt}{\operatorname{pt}}
\newcommand{\sign}{\text{sgn}}
\newcommand{\Aut}{\text{Aut}}

\title[The Gromov--Witten axioms via polyfold theory]{The Gromov--Witten axioms for symplectic manifolds via polyfold theory}
\date{\today}
\thanks{Research partially supported by Project C5 of SFB/TRR 191 ``Symplectic Structures in Geometry, Algebra and Dynamics,'' funded by the DFG}

\author{Wolfgang Schmaltz}
\address{Mathematics Institute, Justus-Liebig University, D-35392 Gie{\ss}en, Germany}
\email{\href{mailto:wolfgang.schmaltz@math.uni-giessen.de}{wolfgang.schmaltz@math.uni-giessen.de}}
\urladdr{\url{https://sites.google.com/view/wolfgang-schmaltz/home}}

\begin{document}


\begin{abstract}
	Polyfold theory, as developed by Hofer, Wysocki, and Zehnder, is a relatively new approach to resolving transversality issues that arise in the study of $J$-holomorphic curves in symplectic geometry.
	This approach has recently led to a well-defined Gromov--Witten invariant for $J$-holomorphic curves of arbitrary genus, and for all closed symplectic manifolds.
	
	The Gromov--Witten axioms, as originally described by Kontsevich and Manin, give algebraic relationships between the Gromov--Witten invariants.	
	In this paper, we prove the Gromov--Witten axioms for the polyfold Gromov--Witten invariants.
\end{abstract}

\maketitle

\tableofcontents


\section{Introduction}

\subsection{History}

In 1985 Gromov published the paper ``Pseudo holomorphic curves in symplectic manifolds,'' laying the foundations for the modern study of pseudo holomorphic curves (also know as $J$-holomorphic curves) in symplectic topology \cite{G}.
In this paper, Gromov proved a compactness result for the moduli space of $J$-holomorphic curves in a fixed homology class.  This paper contained antecedents to the modern notion of the Gromov--Witten invariants in the proofs of the nonsqueezing theorem and the uniqueness of symplectic structures on $\CP^2$.

Around 1988, inspired by Floer's study of gauge theory on three manifolds, Witten introduced the topological sigma model \cites{floer1988instanton, witten1988topological}.
The invariants of this model are the ``$k$-point correlation functions,'' another precursor to the modern notion of the Gromov--Witten invariants.
Witten also observed some of the relationships between these invariants and possible degenerations of Riemann surfaces \cite{witten1990two}.
Further precursors to the notion of the Gromov--Witten invariants can also be seen in McDuff's classification of symplectic ruled surfaces \cite{mcduff1991symplectic}.

In 1993 Ruan gave a modern definition of the genus zero Gromov--Witten invariants for semipositive symplectic manifolds \cites{ruan1996topological, ruan1994symplectic}.  At the end of 1993, Ruan and Tian established the associativity of the quantum product for semipositive symplectic manifolds, giving a mathematical basis to the composition law of Witten's topological sigma model \cite{rt1995quatumcohomology}.

In 1994 Kontsevich and Manin stated the Gromov--Witten axioms, given as a list of formal relations between the Gromov--Witten invariants \cite{KM}.  
At the time it was not possible for Kontsevich and Manin to give a proof of the relations they listed; the definition of the Gromov--Witten invariant (complete with homology classes from a Deligne--Mumford space) would require in addition new ideas involving ``stable maps'' \cite{Kstable}.
Hence they used to term ``axiom'' with the presumed meaning ``to take for assumption without proof''/``to use as a premise for further reasoning.''  And indeed, from these starting assumptions they were able to establish foundational results in enumerative geometry, answers to esoteric questions such as:
\begin{quote}
	\textit{(Kontsevich's recursion formula).}  Let $d\geq 1$.  How many degree $d$ rational curves in $\CP^2$ pass through $3d - 1$ points in general position?
\end{quote}
Moreover, in this paper they outlined some of the formal consequences of the axioms by demonstrating how to combine the invariants into a Gromov--Witten potential, and interpret the axioms as differential equations which the potential satisfies.

To varying extents, this work has predated the construction of a well-defined Gromov--Witten invariant in symplectic geometry for $J$-holomorphic curves of arbitrary genus, and for all closed symplectic manifolds.
Efforts to construct a well-defined Gromov--Witten invariant constitute an ever growing list of publications, including but not limited to the following: \cites{LT, FO, fukaya2012technical, Si, cieliebak2007symplectic, mcduff2012smooth, mcduff2018fundamental, MWtopology, ionel2013natural, pardon2016algebraic}. 
A discussion of some of the difficulties inherent in these approaches can be found in \cite{ffgw2016polyfoldsfirstandsecondlook}.
Similarly, there have been several efforts to prove the Gromov--Witten axioms \cites{FO, MSbook, castellano2016genus}.

Over the past two decades, Hofer, Wysocki, and Zehnder have developed a new approach to resolving transversality issues that arise in the study of $J$-holomorphic curves in symplectic geometry called \textit{polyfold theory} \cites{HWZ1, HWZ2, HWZ3, HWZGW, HWZsc, HWZdm, HWZint, HWZbook}.  This approach has been successful in constructing a well-defined Gromov--Witten invariant \cite{HWZGW}.

\subsection{The polyfold Gromov--Witten invariants}

Let $(Q,\ww)$ be a closed symplectic manifold and let $J$ be a compatible almost complex structure.
Let $2n:= \dim_{\R} Q$.
For a fixed homology class $A\in H_2(Q,\Z)$, and for fixed integers $g\geq 0$, $k\geq 0$, we consider the following set:
	\[
	\cM_{A,g,k}(J) := 
	\left\{
	\begin{array}{c}
	u: (\Sigma_g,j) \to Q \\
	\{z_1,\ldots,z_k\}\in \Sigma_g
	\end{array}
	\biggm|
	\begin{array}{c}
	\tfrac{1}{2} (du+J\circ du\circ j)=0 \\
	u_*[\Sigma_g] = A
	\end{array}
	\right\}
	\biggm/
	\begin{array}{l}
	u \sim u\circ \phi,\\
	\phi\in \Aut
	\end{array}
	\]
consisting of smooth maps $u:(\Sigma_g,j)\to Q$ which satisfy the Cauchy--Riemann equation modulo reparametrization; here $(\Sigma_g,j)$ is a genus $g$ Riemann surface and $\Aut$ is the automorphism group of the Riemann surface $(\Sigma_g,j)$ which preserves the ordering of the marked points.
We will refer to an equivalence class of a solution to the Cauchy--Riemann equation as a \textbf{$J$-holomorphic curve}.

Gromov's compactness theorem states that given a sequence of $J$-holomorphic curves there exists a subsequence which ``weakly converges'' to a ``cusp-curve'' \cite{G}.
This was later refined in \cite{Kstable} into the ``stable map compactification.''
Consequently, the set $\cM_{A,g,k}(J)$ can be compactified by adding nodal curves yielding a \textit{compact} topological space
	\[
	\CM_{A,g,k}(J) := \cM_{A,g,k}(J) \sqcup \{\text{nodal curves}\}.
	\]	
We call this space the \textbf{unperturbed Gromov--Witten moduli space} of genus $g$, $k$ marked stable curves which represent the class $A$.

In a set of small but often studied cases where the symplectic manifold $(Q,\ww)$ is ``semipositive'' it is possible to give this compact topological space the additional structure of a ``pseudocycle,'' which is suitable for defining an invariant.
This is achieved via a perturbation of the almost complex structure $J$.
The space of compatible almost complex structures $\mathcal{J}(Q,\ww)$ is nonempty and contractible, from which it can be shown that the invariant does not depend on the choice of $J$.
However in general symplectic manifolds no $J\in \mathcal{J}(Q,\ww)$ can give sufficient transversality to yield a well-defined invariant.  For a textbook treatment of this material, we refer to \cite{MSbook}.

Polyfold theory, developed by Hofer, Wysocki, and Zehnder, is a relatively new approach to resolving transversality issues that arise in attempts to solve moduli space problems in symplectic geometry.
The polyfold theoretic approach to solving a moduli space problem is to recast the problem into familiar terms from differential geometry.
To do this, we may construct a ``Gromov--Witten polyfold'' $\cZ_{A,g,k}$---a massive, infinite-dimensional ambient space, designed to contain the entire unperturbed Gromov--Witten moduli space $\CM_{A,g,k}(J)$ as a compact subset.  
We may furthermore construct a ``strong polyfold bundle'' $\cW_{A,g,k}$ over $\cZ_{A,g,k}$; the Cauchy--Riemann operator then defines a ``scale smooth Fredholm section'' of this bundle, $\delbarj :\cZ_{A,g,k} \to \cW_{A,g,k}$, such that $\delbarjinv(0) = \CM_{A,g,k}(J)$.
We can construct ``abstract perturbations'' $p$ of this section such that $\delbarj +p$ is transverse to the zero section and such that $(\delbarj+p)^{-1}(0)$ is a compact set.  
In this way, we may take a scale smooth Fredholm section and ``regularize'' the unperturbed Gromov--Witten moduli space yielding a \textbf{perturbed Gromov--Witten moduli space} $\cS_{A,g,k}(p):= \text{`` }(\delbarj+p)^{-1}(0) \text{ ''}$ which has the structure of a compact oriented ``weighted branched orbifold.''

	\[
	\begin{array}{c}
	\CM_{A,g,k}(J) = \delbarj^{-1}(0) \\
	\text{\small{compact topological space}} \\
	\mbox{}
	\end{array}
	\begin{tikzcd}
	\mbox{} \arrow[rrr, "\text{\small{``polyfold regularization''}}", squiggly] & \mbox{} & \mbox{} & \mbox{}
	\end{tikzcd}
	\begin{array}{c}
	\cS_{A,g,k}(p):=(\delbarj+p)^{-1}(0) \\
	\text{\small{compact ``weighted}}\\
	\text{\small{branched orbifold''}}
	\end{array}
	\]

This approach has been successful in giving a well-defined Gromov--Witten invariant for curves of arbitrary genus, and for all closed symplectic manifolds.
Suppose that $2g+k\geq 3$, and consider the following diagram of smooth maps between the perturbed Gromov--Witten moduli space $\cS_{A,g,k}(p)$, the $k$-fold product manifold $Q^k$, and the Deligne--Mumford orbifold $\dmlog_{g,k}$:  
	\[
	\begin{tikzcd}[column sep=12ex]
	\cS_{A,g,k}(p) \arrow{r}{ev_1\times\cdots\times ev_k} \arrow{d}{\pi} & Q^k\\
	\dmlog_{g,k}
	\end{tikzcd}
	\]
Here $ev_i$ is evaluation at the $i$th-marked point, and $\pi$ is the projection map to the Deligne--Mumford space which forgets the stable map solution and stabilizes the resulting nodal Riemann surface by contracting unstable components.

Consider homology classes $\alpha_1,\ldots, \alpha_k \in H_* (Q;\Q)$ and $\beta\in H_* (\dmlog_{g,k};\Q)$.
We can represent the Poincar\'e duals of the $\alpha_i$ and $\beta$ by closed differential forms in the de Rahm cohomology groups, $\PD(\alpha_i)\in H^*_{\dR} (Q)$ and $\PD(\beta)\in H^*_{\dR}(\dmlog_{g,k})$.
By pulling back via the evaluation and projection maps, we obtain a closed $\ssc$-smooth differential form 
	\[
	ev_1^* \PD (\alpha_1) \wedge \cdots \wedge ev_k^* \PD(\alpha_k) \wedge\pi^* \PD (\beta) \in H^*_{\dR} (\cZ_{A,g,k}).
	\]
	
\begin{theorem}[{\cite[Thm.~1.12]{HWZGW}}]
	The \textbf{polyfold Gromov--Witten invariant} is the homomorphism
		\[
		\GW_{A,g,k} : H_* (Q;\Q)^{\otimes k} \otimes H_* (\dmlog_{g,k}; \Q) \to \Q
		\]
	defined via the ``branched integration'' of \cite{HWZint}:
		\[
		\GW_{A,g,k} (\alpha_1,\ldots,\alpha_k;\beta) : = \int_{\cS_{A,g,k}(p)} ev_1^* \PD (\alpha_1) \wedge \cdots \wedge ev_k^* \PD(\alpha_k) \wedge\pi^* \PD (\beta).
		\]
	This invariant does not depend on the choice of perturbation.
\end{theorem}

For a survey of the core ideas of the polyfold theory, we refer to \cite{ffgw2016polyfoldsfirstandsecondlook}.
For a complete treatment of polyfold theory in the abstract, we refer to \cite{HWZbook}.
For the construction of the GW-polyfolds and the polyfold GW-invariants, we refer to \cite{HWZGW}.

\subsection{The Gromov--Witten axioms}

With a general polyfold Gromov--Witten invariant in place, a natural question is: To what extent does this newly defined invariant satisfy traditional results of Gromov--Witten theory for symplectic manifolds?  A natural place to begin is with verifying the Gromov--Witten axioms.

\begin{mainresult}
	The polyfold Gromov--Witten invariants satisfy the Gromov--Witten axioms.
\end{mainresult}
\begin{axiom*}{Effective axiom}
	If $\ww (A) <0 $ then $\GW_{A,g,k} = 0$.
\end{axiom*}
\begin{axiom*}{Grading axiom}
	If $\GW_{A,g,k} (\alpha_1,\ldots,\alpha_k; \beta) \neq 0$ then
	\[
	\sum_{i=1}^k (2n - \deg (\alpha_i)) + (6g-6+2k - \deg(\beta)) = 2c_1(A) + (2n - 6)(1-g) + 2k.
	\]
\end{axiom*}
\begin{axiom*}{Homology axiom}
	There exists a homology class
	\[
	\sigma_{A,g,k} \in H_{2c_1(A) + (2n-6)(1-g) + 2k} (Q^k\times \dmspace_{g,k};\Q)
	\]
	such that
	\[
	\GW_{A,g,k} (\alpha_1,\ldots,\alpha_k; \beta) = \langle p_1^* \PD(\alpha_1) \smallsmile \cdots \smallsmile p_k^*\PD (\alpha_k) \smallsmile p_0^*\PD(\beta), \sigma_{A,g,k} \rangle
	\]
	where $p_i: Q^k \times \dmspace_{g,k} \to Q$ denotes the projection onto the $i$th factor and the map $p_0:Q^k \times \dmspace_{g,k}\to\dmspace_{g,k}$ denotes the projection onto the last factor.
\end{axiom*}
\begin{axiom*}{Zero axiom}
	If $A=0,\ g=0$ then $\GW_{0,0,k} (\alpha_1,\ldots,\alpha_k;\beta) = 0$ whenever $\deg (\beta) >0$, and
	\[
	\GW_{0,0,k} (\alpha_1,\ldots,\alpha_k; [\pt]) 
	= \int_Q \PD(\alpha_1) \wedge \cdots \wedge \PD(\alpha_k).
	\]
\end{axiom*}
\begin{axiom*}{Symmetry axiom}
	Fix a permutation $\sigma: \{1,\ldots, k\}\to \{1,\ldots,k\}$.  Consider the permutation map $\sigma:\dmlog_{g,k}\to \dmlog_{g,k}, \ [\Sigma,j,M,D]  \mapsto [\Sigma,j,M^\sigma,D]$ where $M = \{z_1,\ldots,z_k\}$ and where $M^\sigma := \{z'_1,\ldots,z'_k\}$, $z'_i:= z_{\sigma(i)}$.
	Then
	\[
	\GW_{A,g,k} (\alpha_{\sigma(1)},\ldots,\alpha_{\sigma(k)}; \sigma_*\beta) = (-1)^{N(\sigma;\alpha_i)} \GW_{A,g,k} (\alpha_1,\ldots,\alpha_k; \beta)
	\]
	where $N(\sigma;\alpha_i):= \sharp 	\{	i<j \mid \sigma(i)> \sigma(j), \deg (\alpha_i)\deg(\alpha_j)\in 2\Z +1	\}	$.
\end{axiom*}

\begin{definition}[{\cite[Eq.~2.3]{KM}}]
	\label{def:basic-classes}
	We say that $(A,g,k)$ is a \textbf{basic class} if it is equal to one of the following: $(A,0,3)$, $(A,1,1)$, or $(A,g\geq 2,0)$.
\end{definition}
The point is, for such values of $g$ and $k$ we will have $\dmspace_{g,k-1} = \emptyset$ by definition.

\begin{axiom*}{Fundamental class axiom}
	Consider the fundamental classes $[Q]\in H_{2n}(Q;\Q)$ and $[\dmlog_{g,k}] \in H_{6g-6+2k}(\dmlog_{g,k};\Q)$.  Suppose that $A\neq 0$ and that $(A,g,k)$ is not basic.  Then 
	\[
	\GW_{A,g,k} (\alpha_1,\ldots,\alpha_{k-1},[Q]; [\dmlog_{g,k}]) = 0.
	\]
	
	Consider the canonical section $s_i :\dmlog_{g,k-1} \to \dmlog_{g,k}$ defined by doubling the $i$th-marked point.  Then
	\[
	\GW_{A,g,k} (\alpha_1,\ldots,\alpha_{k-1},[Q]; s_{i*}\beta) = \GW_{A,g,k-1} (\alpha_1,\ldots,\alpha_{k-1};\beta).
	\]
\end{axiom*}

\begin{axiom*}{Divisor axiom}
	Suppose $(A,g,k)$ is not basic.
	If $\deg (\alpha_k) = 2n-2$ then
	\[
	\GW_{A,g,k} (\alpha_1,\ldots,\alpha_k; \PD (ft_k^* \PD (\beta))) = (A\cdot \alpha_k ) \ \GW_{A,g,k-1} (\alpha_1,\ldots,\alpha_{k-1};\beta),
	\]
	where $A\cdot \alpha_k$ is given by the homological intersection product.
\end{axiom*}

Let $\{e_\nu \} \in H^*(Q;\Q)$ be a homogeneous basis and let $\{e^\mu \} \in H^*(Q;\Q)$ be the dual basis with respect to Poincar\'e duality, i.e., $\langle e_\nu \smallsmile e^\mu, [Q] \rangle = \delta_{\nu \mu}$.
It follows from the K\"unneth formula that $\{e_\nu \otimes e^\mu \}$ is a basis for $H^*(Q\times Q;\Q)$.
We correct the sign by redefining $e_\nu$ as $(-1)^{\deg e_\nu} e_\nu$.
We can write the Poincar\'e dual of the diagonal $\Delta \subset Q\times Q$ in this basis as 
$\PD([\Delta]) = \sum_\nu e_\nu \otimes e^\nu$ (see \cite[Lem.~11.22]{bott2013differential}).

\begin{axiom*}{Splitting axiom}
	Fix a partition $S_0 \sqcup S_1 =\{1,\ldots, k\}$.
	Let $k_0 := \sharp S_0,$ $k_1 := \sharp S_1$ and let $g_0,$ $g_1 \geq 0$ such that $g= g_0 + g_1$,
	and $k_i + g_i \geq 2$ for $i=0,1$.
	Consider the natural map 
	\[\phi_S : \dmspace_{k_0+1 , g_0}\times \dmspace_{k_1+1 , g_1} \to \dmspace_{g,k}\]
	which identifies the last marked point of a stable noded Riemann surface in $\dmspace_{k_0+1 , g_0}$ with the first marked point of a stable noded Riemann surface in $\dmspace_{k_1+1, g_1}$, and which maps the first $k_0$ marked points of $\dmspace_{g_0,k_0+1}$ to marked points indexed by $S_0$ and likewise maps the last $k_1$ marked points of $\dmspace_{g_1,k_1+1}$ to marked points indexed by $S_1$.
	Then
	\begin{align*}
	&\GW_{A,g,k} (\alpha_1, \ldots, \alpha_k; \phi_{S*} (\beta_0\otimes \beta_1) ) = 
	(-1)^{N(S;\alpha)} \sum_{A_0+A_1 = A}	\sum_\nu \\
	&\qquad 
	\GW_{A_0,g_0,k_0+1} (\{\alpha_i\}_{i\in S_0}, \PD (e_\nu) ; \beta_0)
	\cdot
	\GW_{A_1,g_1,k_1+1} (\PD (e^\nu), \{\alpha_j\}_{j\in S_1} ; \beta_1)
	\end{align*}
	where $N(S;\alpha)=\sharp \{j<i \mid i\in S_0, j\in S_1, \deg(\alpha_i)\deg(\alpha_j)\in 2\Z +1 \}$.	
\end{axiom*}

\begin{axiom*}{Genus reduction axiom\footnote{We note that the original statement \cite[Eq.~2.12]{KM} missed the additional factor of $2$.}}
	Consider the natural map 
	\[\psi: \dmspace_{g-1,k+2} \to \dmspace_{g,k}\]
	which identifies the last two marked points of a stable noded Riemann surface, increasing the arithmetic genus by one.
	Then
	\[
	2 \cdot \GW_{A,g,k} (\alpha_1, \ldots, \alpha_k; \psi_* \beta) = \sum_\nu \GW_{A,g-1,k+2} (\alpha_1,\ldots,\alpha_k, \PD (e_\nu) , \PD (e^\nu) ; \beta).
	\]
\end{axiom*}

\subsection{Strategy of the proof}

The Gromov--Witten axioms give relationships between the Gromov--Witten invariants.
These relationships are determined by the geometry of certain naturally defined maps defined between the unperturbed Gromov--Witten moduli spaces, namely:
\begin{itemize}
	\item permutation maps, \[\sigma : \CM_{A,g,k}(J) \to \CM_{A,g,k}(J),\]
	\item $k$th-marked point forgetting maps, \[ft_k : \CM_{A,g,k}(J) \to \CM_{A,g,k-1}(J),\]
	\item canonical sections, \[s_i : \CM_{A,g,k-1}(J) \hookrightarrow \CM_{A,g,k}(J).\]
\end{itemize}

Furthermore, using the map
	$ev_{k_0+1} \times ev_1 : \CM_{A_0,g_0,k_0+1}(J) \times \CM_{A_1,g_1,k_1+1}(J) \to Q\times Q$
we may consider the subset  $(ev_{k_0+1} \times ev_1 )^{-1}(\Delta)$ of the product unperturbed Gromov--Witten moduli space with a constraint imposed by the diagonal $\Delta \subset Q\times Q$.
We then additionally have:
\begin{itemize}
	\item inclusion maps, and maps $\phi$ which identify the marked points $z_{k_0+1}$ and $z_1'$,
\end{itemize}
	\[
	\begin{tikzcd}
	\CM_{A_0,g_0,k_0+1}(J) \times \CM_{A_1,g_1,k_1+1}(J)	&	\\
	(ev_{k_0+1} \times ev_1 )^{-1}(\Delta) \arrow[u, hook, "i"] \arrow[r, "\phi"] & \CM_{A_0+A_1,g_0+g_1,k_0+k_1}(J)
	\end{tikzcd}
	\]
Likewise, using the map
	$ev_{k+1} \times ev_{k+2} : \CM_{A,g-1,k+2}(J) \to Q\times Q$
we may consider the subset $(ev_{k+1} \times ev_{k+2} )^{-1}(\Delta)$ of the unperturbed Gromov--Witten moduli space with a constraint imposed by the diagonal $\Delta \subset Q\times Q$.
We then additionally have: 
\begin{itemize}
	\item inclusion maps, and maps $\psi$ which identify the marked points $z_{k+1}$ and $z_{k+2}$ (increasing the arithmetic genus by one),
\end{itemize}
	\[
	\begin{tikzcd}
	\CM_{A,g-1,k+2}(J)	&	\\
	(ev_{k+1} \times ev_{k+2} )^{-1}(\Delta) \arrow[u, hook, "i"] \arrow[r, "\psi"] & \CM_{A,g,k}(J)
	\end{tikzcd}
	\]
	
Intuitively, we should prove the Gromov--Witten axioms by interpreting the Gromov--Witten invariants as a finite count of curves and using the geometry of the above maps to directly compare such counts with respect to constraints imposed by the homology classes on $Q$ and $\dmlog_{g,k}$.

A substantial amount of work is required to make this intuition rigorous in the context of an abstract perturbation theory.
A deep understanding of the full machinery of polyfold theory, in addition to the geometry of the Gromov--Witten invariants is necessary to navigate substantial difficulties that we encounter.

\subsubsection*{The polyfold Gromov--Witten invariants as intersection numbers}

The branched integral is useful for giving a well-defined definition of the polyfold Gromov--Witten invariants and moreover showing that they are, in fact, invariants and do not depend on choices.  But they are not the best viewpoint for giving a proof of all of the axioms.

To prove the Gromov--Witten axioms, it is necessary to interpret the Gromov--Witten invariants as a finite count of curves via intersection theory.
By \cite[Cor.~1.7]{schmaltz2019steenrod}, the polyfold Gromov--Witten invariants may equivalently be defined as an intersection number evaluated on a basis of representing submanifolds $\cX \subset Q$ and {representing suborbifolds} $\cB \subset \cO$:
\[
\GW_{A,g,k} ([\cX_1],\ldots,[\cX_k];[\cB]) :=
\left(ev_1\times\cdots\times ev_k\times\pi\right)|_{\cS_{A,g,k}(p)} \cdot \left(\cX_1 \times\cdots\times \cX_k \times \cB\right).
\]
The invariant does not depend on the choice of abstract perturbation, nor on the choice of representing basis.
Thus, the traditional geometric interpretation of the Gromov--Witten invariants as a ``count of curves which at the $i$th-marked point passes through $\cX_i$ and such that the image under the projection $\pi$ lies in $\cB$'' is made literal.

\subsubsection*{Pulling back abstract perturbations}

In some cases there exist natural extensions of these maps from the Gromov--Witten moduli spaces to the modeling Gromov--Witten polyfolds.
However, these maps will not in general \textit{persist} after abstract perturbation, i.e., maps between Gromov--Witten polyfolds will not have well-defined restrictions to the perturbed Gromov--Witten moduli spaces.

In the semipositive situation, the set of compatible almost complex structures $\cJ (Q, \ww)$ gives a common space of perturbations; we can therefore choose a common regular $J$ for the source and target of a map and obtain a well-defined map between Gromov--Witten moduli spaces.

In contrast, abstract perturbations are constructed using bump functions and choices of vectors in a strong polyfold bundle, which in general we cannot assume will be preserved by an arbitrary map between Gromov--Witten polyfolds.
For example, consider the permutation map which lifts to a $\ssc$-diffeomorphism 
	$\sigma: \cZ_{A,g,k} \to \cZ_{A,g,k}.$
In general, the aforementioned bump functions and choices of vectors in a strong polyfold bundle will not exhibit symmetry with regards to the labelings of the marked points.
As a result, given a stable curve $x \in \cZ_{A,g,k}$ which satisfies a perturbed equation $(\delbarj +p )(x)=0$ we cannot expect that  $(\delbarj +p )(\sigma (x))=0$.
Therefore, naively there does not exist a well-defined permutation map between perturbed Gromov--Witten moduli spaces.

The natural approach for obtaining a well-defined map between perturbed moduli spaces is to pullback an abstract perturbation.
In \S~\ref{subsec:pulling-back-abstract-perturbations} we apply \cite[Thm.~1.7]{schmaltz2019naturality} to obtain well-defined restricted maps between perturbed Gromov--Witten moduli spaces for several of the maps we have considered.

\subsubsection*{Problems arise}

The $k$th-marked point forgetting map is, by far, the most difficult and complicated map to define between perturbed Gromov--Witten moduli spaces, as we immediately encounter numerous difficulties.

The construction of the smooth structure for the Deligne--Mumford orbifolds as described in \cites{HWZGW, HWZdm} requires a choice: that of a ``gluing profile,'' i.e., a smooth diffeomorphism $\varphi: (0,1]\to [0,\infty).$
The logarithmic gluing profile is given by 
	\[\varphi_{\log} (r) = -\frac{1}{2\pi} \log (r)\]
and produces the classical holomorphic Deligne--Mumford orbifolds $\dmlog_{g,k}$.  There is also an exponential gluing profile, given by 
	\[\varphi_{\exp} (r) = e^{1/r} - e\]
which produces Deligne--Mumford orbifolds $\dmexp_{g,k}$ which are only smooth orbifolds.

This use of nonstandard smooth structure has the following consequence: 
	\begin{quote}
		\textit{In general the map $ft_k: \dmexp_{g,k} \to \dmexp_{g,k-1}$ is continuous but not differentiable (see Problem~\ref{prob:ftk-continuous-not-c0}).}
	\end{quote}

Independent of the usage of a nonstandard gluing profile, there is no hope of defining a $k$th-marked point forgetting map on the Gromov--Witten polyfolds as they are defined:
\begin{quote}
	\textit{In general there does not exist a natural map $ft_k$ on the Gromov--Witten polyfolds (see Problem~\ref{prob:no-natural-forgetting-map}).}
\end{quote}
The reason is that the GW-stability condition \eqref{eq:gw-stability-condition} imposed on stable curves in the polyfold $\mathcal{Z}_{A,g,k}$ may not hold in $\mathcal{Z}_{A,g,k-1}$ once the $k$th point is removed.
A stable curve in $\cZ_{A,g,k}$ may contain a ``destabilizing ghost component,'' i.e., a component $C_k\simeq S^2$ with precisely $3$ special points, one of which is the $k$th-marked point, and such that $\int_{C_k} u^*\ww =0, \ u|_{C_k} \neq\text{const}$.
After removal of the $k$th-marked point from such a component, the GW-stability condition \eqref{eq:gw-stability-condition} is no longer satisfied and we cannot consider the resulting data as a stable curve in $\cZ_{A,g,k-1}$.


We might try to consider a subset of stable curves in $\cZ_{A,g,k}$ for which the GW-stability condition \eqref{eq:gw-stability-condition} will hold after forgetting the $k$th-marked point; thus we can attempt to restrict to a subset $\cZ^\text{const}_{A,g,k}\subset \cZ_{A,g,k}$ with a stronger stability condition, and such that the $k$th-marked point forgetting map is well-defined on $\mathcal{Z}^\text{const}_{A,g,k}$.
However, if we consider $\mathcal{Z}^\text{const}_{A,g,k}\subset \mathcal{Z}_{A,g,k}$ with the subspace topology, and $\mathcal{Z}_{A,g,k-1}$ with the usual polyfold topology, then:
\begin{quote}
	\textit{In general the well-defined restriction $ft_k:\cZ_{A,g,k}^\text{const} \to \mathcal{Z}_{A,g,k-1}$ is not continuous (see Problem~\ref{prob:restriction-is-not-continuous}).}
\end{quote}

There is a final problem. In general, the projection map must factor through the $k$th-marked point forgetting map; this is due to the need to forget the added stabilizing points (see Proposition~\ref{prop:projection-map-smooth}). Thus, in order to obtain a smooth projection map we must map to the logarithmic Deligne--Mumford orbifold. However:
\begin{quote}
	\textit{While the projection $\pi : \cZ_{A,g,k} \to \dmlog_{g,k}$ is $\ssc$-smooth, in general it is not a submersion (see Problem~\ref{prob:projection-not-submersion}).}
\end{quote}
This has important consequences if we wish to consider the Gromov--Witten invariant as an intersection number; the only way to get transversality of the projection map with a representing suborbifold $\cB \subset \dmlog_{g,k}$ is through perturbation of the suborbifold.
Fortunately, by \cite[Thm.~1.2]{schmaltz2019steenrod} there exist suitable representing suborbifolds for which perturbation is possible \cite[Prop.~3.9]{schmaltz2019steenrod}.

\subsubsection*{The universal curve polyfold}

In essence the central problem is that the Gromov--Witten polyfolds as constructed are not ``universal curves.''
Our proof of Gromov--Witten axioms rectifies this by constructing a universal curve polyfold $\cZ^\text{uc}_{A,g,k}$ over $\cZ_{A,g,k-1}$, on which we may consider a well-defined $k$th-marked point forgetting map
	\[
	ft_k : \cZ^\text{uc}_{A,g,k} \to \cZ_{A,g,k-1}.
	\]
The preimage of stable curve  in $\cZ_{A,g,k-1}$ via $ft_k$ consists of the underlying Riemann surface with nodes identified, thereby justifying our choice of nomenclature ``universal curve.''

Although this map is $\ssc^0$ due to the use of the exponential gluing profile, it is still possible to pullback $\ssc$-smooth abstract perturbations via this map.
Furthermore, applying \cite[Thm.~1.3]{schmaltz2019naturality} we may prove that the invariants associated to the universal curve polyfold coincide with the usual polyfold Gromov--Witten invariants.

\subsection{Organization of the paper}

In \S~\ref{sec:dm-orbifolds-gw-polyfolds} we review the construction of the DM-orbifolds and the GW-polyfolds.
In particular, we consider the local smooth structures given by the good uniformizing families of stable noded Riemann surfaces/stable maps. We also consider additional GW-type polyfolds needed to prove the splitting and genus reduction axioms.
In \S~\ref{sec:maps-between-orbifolds-polyfolds} we consider certain naturally defined maps on the DM-orbifolds/GW-polyfolds.
In \S~\ref{sec:the-polyfold-gw-invariants} we recall the definition of the polyfold GW-invariants, defined equivalently via the branched integral and by the intersection number. We show how to pullback abstract perturbations for many of the natural maps we are considering, yielding well-defined maps on the perturbed GW-moduli spaces.
In \S~\ref{sec:problems-arise} we discuss the many problems that arise in considering the $k$th-marked point forgetting map.
In \S~\ref{sec:universal-curve-polyfold} we construct the universal curve polyfold.
We prove that the invariants for this polyfold are equal to the usual GW-invariants.
Furthermore, we show that we can define a $k$th-marked point forgetting map on this polyfold, and that we may pullback abstract perturbations via this map.
In \S~\ref{sec:gw-axioms} we prove the GW-axioms.


\section{The Deligne--Mumford orbifolds and the Gromov--Witten polyfolds}
	\label{sec:dm-orbifolds-gw-polyfolds}

In this section, we give a precise description of the underlying sets and the local smooth structures of the DM-orbifolds and GW-polyfolds.
A full treatment of the smooth structure of an orbifold or a polyfold requires in addition transition information for how these local smooth structures fit together; this is accomplished via the language of ep-groupoids.

Our discussion of the DM-orbifolds is from the perspective of modern symplectic geometry (rather than algebraic geometry), for this point of view, we refer to \cites{HWZdm,robbinsalamon2006delignemumford}.
Our discussion of the GW-polyfolds is due to \cite{HWZGW}.

\subsection{The Deligne--Mumford orbifolds}

We begin by describing the underlying sets of the DM-spaces and discussing their natural topology.

\begin{definition}\label{def:underlying-dmspace}
	For fixed integers $g\geq 0$ and $k\geq 0$ which satisfy $2g+k\geq 3$, the \textbf{underlying set of the Deligne--Mumford space} $\dmspace_{g,k}$ is defined as the set of equivalence classes of stable noded Riemann surfaces of arithmetic genus $g$ and with $k$ marked points, i.e.,
	\[
	\dmspace_{g,k}:= \{(\Sigma,j,M,D) \mid \cdots, \text{ DM-stability condition} \} / \sim
	\]
	with data as follows:
	\begin{itemize}
		\item $(\Sigma,j)$ is a closed (possibly disconnected) Riemann surface.
		\item $M$ consists of $k$ ordered distinct marked points $z_1,\ldots,z_k\in\Sigma$.
		\item $D$ consists of finitely many unordered nodal pairs $\{x,y\}$ with $x,y\in\Sigma$ and $x\neq y$.  We require that two such pairs are disjoint.  We require both elements of the pair to be distinct from $M$.  We denote $\abs{D} := \cup_{\{x,y\}\in D} \{x,y\} \subset \Sigma$, and we let $\sharp D:=\tfrac{1}{2} \sharp \abs{D}$, i.e., the number of pairs.
		\item Viewing each connected component $C\subset \Sigma$ as a vertex, and each nodal pair $\{x,y\}$ as an edge via the incidence relation $\{C_x,C_y\}$ if $x\in C_x,\ y\in C_y$, we obtain a graph $T$.  We require that $T$ be connected.
		\item The arithmetic genus $g$ is given as:
		\[
		g=\sum_{C} g_C + \text{number of cycles of the graph $T$}
		\]
		where the sum is taken over the finitely many connected components $C\subset \Sigma$, and where $g_C$ is defined as the genus of the connected component $C$.
		\item For each connected component $C\subset \Sigma $ we require the following \textbf{DM-stability condition}:
			\begin{equation}
			\label{eq:dm-stability-condition}
			2 g_C+\sharp (M \cup \abs{D})_C \geq 3
			\end{equation}
		where $\sharp (M \cup \abs{D})_C := \sharp ((M\cup\abs{D}) \cap C)$, i.e., the number of marked and nodal points on the component $C$.
		\item The equivalence relation is given by $(\Sigma,j,M,D)\sim (\Sigma',j',M',D')$ if there exists a biholomorphism $\phi: (\Sigma,j)\to(\Sigma',j')$ such that $\phi(M)=M',\ \phi(\vert D\vert)=\vert D'\vert$, and which preserves the ordering of the marked points, and maps each pair of nodal points to a pair of nodal points. We call such a $\phi$ a \textbf{morphism between stable Riemann surface}, and write it as
			\[
			\phi: (\Sigma,j,M,D) \to (\Sigma',j',M',D').
			\]
	\end{itemize}
  We will refer to any point in $M \cup \abs{D}$ as a \textbf{special point}.
  We call a tuple $(\Sigma,j,M,D)$ which satisfies the DM-stability condition a \textbf{stable noded Riemann surface}.
\end{definition}

For a fixed $(\Sigma,j,M,D)$ the \textbf{isotropy group}
	\[
	\Aut(\Sigma,j,M,D) := 
	\left\{
	\begin{array}{c}
	\phi: (\Sigma,j) \to (\Sigma,j)
	\end{array}
	\biggm|
	\begin{array}{c}
	\phi \text{ biholomorphic}, \\
	\phi(M)=M \text{ preservers ordering},\\
	\phi(\vert D\vert)=\vert D\vert
	\end{array}
	\right\}
	\]
is finite if and only if the DM-stability condition holds.

\begin{proposition}[{\cite[Prop.~2.4]{HWZGW}}]
	Independent of the construction of a smooth structure, the set $\dmspace_{g,k}$ has a natural second-countable, paracompact, Hausdorff topology.  With this topology, $\dmspace_{g,k}$ is a compact topological space.
\end{proposition}


To describe the local smooth structure of a DM-orbifold, we must construct certain families of stable Riemann surfaces. These families are the so-called ``good uniformizing families.''

\subsubsection{Gluing profiles and the gluing construction for a nodal Riemann surface}
	\label{subsubsec:gluing-profiles-gluing-construction-noded-reimann-surface}

In order to describe the gluing construction for a nodal Riemann surface, we must first choose a ``gluing profile'' which we now define.

\begin{definition}[{\cite[Def.~2.1]{HWZGW}}]
	\label{def:gluing-profile}
	A \textbf{gluing profile} is a smooth diffeomorphism 
	\[
	\varphi:(0,1]\to [0, \infty ).
	\]
\end{definition}

A gluing profile allows us to make a conversion of the absolute value of a non-zero complex number into a positive real number.
This is used to define the gluing construction at the nodes of a nodal Riemann surface, replacing a neighborhood of a nodal region with a finite cylinder. We recall this construction now.

Consider a Riemann surface $(\Sigma,j)$ with a nodal pair $\{x,y\}$.
Associate to this pair a \textbf{gluing parameter} $a\in\hb =\{z\in\C \mid \abs{z}< 1/2\}$.
We use the gluing profile to construct a family of Riemann surfaces parametrized by $a$ in the following way:
\begin{enumerate}
	\item Choose small disk-like neighborhoods $D_x$ of $x$ and $D_y$ of $y$, and identifications (via biholomorphisms) $D_x\setminus\{x\}\simeq \R^+\times S^1$ and $D_y \setminus \{y\} \simeq \R^- \times S^1$.  Moreover, the punctures $x$ and $+\infty$ are identified, and likewise $y$ and $-\infty$.
	\item Write the gluing parameter $a\neq 0$ in polar coordinates as
	\[
	a = r_a e^{-2\pi i\theta_a}, \qquad r_a \in (0,\tfrac{1}{2}),\ \theta_a\in \R / \Z.
	\]
	We then use the gluing profile to define a \textbf{gluing length} given by
	\[
	R_a:= \varphi(r_a) \in (\varphi (\tfrac{1}{2}), \infty).
	\]
	\item Delete the points $(R_a,+\infty)\times S^1 \subset \R^+\times S^1$ and $(-\infty,-R_a)\times S^1\subset \R^-\times S^1$ from $\R^+\times S^1$ and $\R^-\times S^1$, and identify the remaining cylinders $[0,R_a]\times S^1$ and $[-R_a,0]\times S^1$ via the map
	\[
	L_{R_a} : [0,R_a]\times S^1 \to [-R_a,0]\times S^1, \qquad (s,t)\mapsto (s-R_a,t-\theta_a).
	\]	
	We replace $D_x \sqcup D_y$ with the finite cylinder
	\[
	Z_a:= [0,R_a] \times S^1 \simeq_{L_{R_a}} [-R_a,0] \times S^1.
	\]
	(For $a=0$ we may define $Z_0 := \R^+\times S^1\sqcup \R^-\times S^1$, identifiable with $D_x \setminus \{x\}\sqcup D_y \setminus \{y\}$.)
	\item This procedure yields a new Riemann surface defined by the quotient space
	\[
	\Sigma_a := \frac{\Sigma \setminus \left (( [R_a,\infty) \times S^1 \cup \{x\} ) \cup (\{y\}\cup (-\infty, R_a] \times S^1)\right) }{L_{R_a} : [0,R_a]\times S^1 \to [-R_a,0]\times S^1}
	\]
	and which carries a naturally induced complex structure.  
\end{enumerate}

Alternatively, we can write $\Sigma_a$ as the disjoint union
	\[
	\Sigma_a = (\Sigma \setminus (D_x\sqcup D_y) ) \sqcup Z_a
	\]
with complex structure
	\[
	j(a):= \begin{cases}
	j &\text{on } \Sigma \setminus (D_x\sqcup D_y) \\
	i &\text{on } Z_a
	\end{cases}
	\]
where $i$ is the standard complex structure on the finite cylinder $Z_a=[0,R_a]\times S^1$.

We can repeat this gluing construction for a Riemann surface with multiple nodes, hence we introduce the following definition.
\begin{definition}
	\label{def:family-Riemann-surfaces}
	Consider a Riemann surface $(\Sigma,j)$ with nodal pairs $\{x_a,y_a\} \in D$.  
	We may choose disjoint small disk-like neighborhoods $D_{x_a}$ and $D_{y_a}$ at every node.  Carrying out the above procedure at every node we obtain a \textbf{glued Riemann surface} defined by the quotient space
	\[
	\Sigma_a := \frac{\Sigma \setminus \sqcup_{\{x_a,y_a\}\in D} \left (( [R_a,\infty) \times S^1 \cup \{x_a\} ) \cup (\{y_a\}\cup (-\infty, R_a] \times S^1)\right) }{ L_{R_a} : [0,R_a]\times S^1 \to [-R_a,0]\times S^1 }
	\]
	and which carries a naturally induced complex structure.
\end{definition}

Alternatively, we can write $\Sigma_a$ as the disjoint union 
	\[
	\Sigma_a := \left(\Sigma \setminus \cup_{\{x_a,y_a\}\in D} (D_{x_a}\sqcup D_{y_a} )\right) \sqcup Z_a.
	\]
with complex structure
	\[
	j(a) := \begin{cases}
	j &\text{on } \Sigma \setminus \cup_{\{x_a,y_a\}\in D} (D_{x_a}\sqcup D_{y_a} )\\
	i &\text{on } Z_a \text{ for every pair } \{x_a,y_a\}\in D
	\end{cases}
	\]
where $i$ is the standard complex structure on the finite cylinder $Z_a$.

\subsubsection{Good uniformizing families of stable Riemann surfaces}

We now describe a certain family of variations of the complex structure of a stable Riemann surfaces.
This variation is designed to additionally account for the movement of the marked points.

\begin{definition}[{\cite[Def.~2.6]{HWZGW}}]
	Let $(\Sigma,j,M,D)$ be a stable noded Riemann surface. Choose (disjoint) small disk-like neighborhoods $D_z \subset \Sigma$ at the special points $z\in M\cup \abs{D}$; we may moreover assume that these disk-like neighborhoods are invariant under the natural action of the isotropy group $\Aut(\Sigma,j,M,D)$.
	Consider a smooth family of complex structures $\cJ(\Sigma)$ on $\Sigma$:
		\[
		V \to \cJ(\Sigma), \qquad	v \mapsto j(v)
		\]
	where $V$ is an open subset which contains $0$ of a complex vector space $E$ of dimension $\dim_{\R} E = 6g-6 + 2\sharp M - 2\sharp D$.
	We call such a family $v\mapsto j(v)$ a \textbf{good complex deformation} if it satisfies the following:
	\begin{itemize}
		\item $j(0)=j$.
		\item The family is constant on the disk-like neighborhoods, i.e., $j(v)=j$ on $D_z$ for every $z\in M\cup \abs{D}$.
		\item For every $v\in V$ the \textit{Kodaira--Spencer differential}
		\[
		[Dj(v)] : H^1 (\Sigma,j,M,D) \to H^1 (\Sigma,j(v),M,D)
		\]
		is a complex linear isomorphism (for the definition of the differential see \cite[pp.~21--22]{HWZGW}).
		\item There exists a natural action
		\[
		\Aut(\Sigma,j,M,D) \times V \to V,\qquad (\phi,v)\mapsto \phi\ast v
		\]
		such that $\phi: (\Sigma,j(v)) \to (\Sigma,j(\phi\ast v))$ is biholomorphic.
	\end{itemize}
\end{definition}

\begin{definition}[{\cite[Def.~2.12]{HWZGW}}]
	\label{def:good-uniformizing-family}
	Let $(\Sigma,j,M,D)$ be a stable noded Riemann surface and let $v\mapsto j(v)$ be a good complex deformation.
	We may define a \textbf{good uniformizing family} as the following family of stable noded Riemann surfaces
		\[
		(a,v) \mapsto (\Sigma_a, j(a,v), M_a, D_a), \quad (a,v)\in (\hb)^{\sharp D} \times V.
		\]
	with data as follows.
	\begin{itemize}
		\item As in Definition~\ref{def:family-Riemann-surfaces} the glued surface $\Sigma_a$ is given by
		\[
		\Sigma_a = \left(\Sigma \setminus \cup_{\{x_a,y_a\}\in D} (D_{x_a}\sqcup D_{y_a} )\right) \sqcup Z_a.
		\]
		\item The complex structure $j(v)$ on the unglued Riemann surfaces $\Sigma$ induces the following complex structure on $\Sigma_a$:
			\[
			j(a,v):= 
			\begin{cases}
			j(v) &\text{ on } \Sigma \setminus \cup_{\{x_a,y_a\}\in D} (D_{x_a}\sqcup D_{y_a} ),\\
			i &\text{ on } Z_a \text{ for every pair } \{x_a,y_a\}\in D.
			\end{cases}
			\]
		\item The set of marked points $z_1,\ldots,z_k \in M_a$ are given by the former marked points which by definition all lie in $\Sigma \setminus \cup_{\{x_a,y_a\}\in D} (D_{x_a}\sqcup D_{y_a})$.
		\item The set of nodal pairs $D_a$ is obtained from $D$ by deleting every nodal pair $\{x_a,y_a\}\in D$ for which $a\neq 0$.
	\end{itemize}
	This good uniformizing family is \textbf{centered} at $(\Sigma,j,M,D)$ in the sense that $(0,0) \mapsto (\Sigma,j,M,D)$.
	We call the parameters $(a,v)$ \textbf{local coordinates} centered at the stable noded Riemann surface $(\Sigma,j,M,D)$.
\end{definition}

The following proposition gives a description of the action of the isotropy group of a point on a neighborhood of the point.

\begin{proposition}
	\label{prop:natural-representation}
	Consider a stable noded Riemann surface $(\Sigma,j,M,D)$ and let $G:= \Aut(\Sigma,j,M,D)$.
	Consider the family of stable noded Riemann surfaces as constructed above, 
		\[(a,v) \mapsto (\Sigma_a,j(a,v),M_a,D_a),\qquad (a,v) \in (\hb)^{\sharp D} \times V.\]
	For every open neighborhood of $0$ in $\hb^{\sharp D} \times V$ there exists an open subneighborhood $U$ of $0$ and a group action
		\begin{equation}
		\label{eq:natural-representation}
		\begin{split}
		G \times U 		&\to U, \\
		(\phi,(a,v)) 	&\mapsto \phi\ast(a,v) 
		\end{split}
		\end{equation}
	such that the following holds.
	\begin{enumerate}
		\item For every $(\phi,(a,v))\in G\times U$ there exists a morphism 
			\[
			\phi_{(a,v)} : (\Sigma_a,j(a,v),M_a,D_a) \to (\Sigma_{\phi\ast (a,v)}, j(\phi\ast(a,v)), M_{\phi\ast (a,v)}, D_{\phi\ast (a,v)}).
			\]
		\item Given a morphism $\psi : (\Sigma_a,j(a,v),M_a,D_a) \to (\Sigma_{a'}, j(a',v'), M_{a'}, D_{a'})$ for parameters $(a,v), (a',v')\in U$, there exists a unique element $\phi \in G$ such that $\phi \ast (a,v) = (a',v')$ and moreover $\psi = \phi_{(a,v)}$.
	\end{enumerate}
	The neighborhood $U$ is called a \textbf{local uniformizer} centered at $(\Sigma,j,M,D)$.
\end{proposition}

\subsubsection{Parametrizing the movement of a marked point}
Situations will arise where we will want to parametrize the movement of the $k$th-marked point directly, e.g., when we consider the $k$th-marked point forgetting maps in \S~\ref{subsec:kth-marked-point-forgetting-map-log-dm}.

In the above description of a good uniformizing family, movement of the marked points is determined by the variation of the complex structure $j$ via the parameter $v$.  Consider a stable noded Riemann surface $(\Sigma,j,M,D)$ with a good complex deformation $v\mapsto j(v)$.  Wiggle the marked points slightly, and obtain a new stable noded Riemann surface $(\Sigma,j,M',D)$; there exists a parameter $v\in V$ such that there exists a biholomorphism $\phi : (\Sigma,j(v),M,D) \to (\Sigma,j,M',D)$.

We now describe a good uniformizing family in which the movement of the $k$th-marked point is parametrized. 
Consider a stable noded Riemann surface $(\Sigma,j,M,D)$ and suppose that the component $C_k$ which contains the $k$th-marked point $z_k$ remains stable after forgetting $z_k$.
We parametrize a neighborhood of $z_k$ by embedding a small disk via a holomorphic map
	\[
	\varphi:(B_\epsilon,i) \hookrightarrow (\Sigma,j), \quad \text{such that } \varphi(0)=z_k\text{ and }\varphi(B_\epsilon)\subset D_{z_k},
	\]
where $D_{z_k}$ is a small disk-like neighborhood of $z_k$.

\begin{definition}
	Let $(\Sigma,j,M,D)$ be a stable noded Riemann surface and choose disjoint smooth disk-like neighborhoods $D_z \subset \Sigma$ at the special points $z\in M\cup \abs{D}$ which we may moreover assume are $\Aut(\Sigma,j,M,D)$-invariant.
	Consider a smooth family of complex structures $\cJ(\Sigma)$ on $\Sigma$:
		\[
		V \to \cJ(\Sigma), \qquad	v \mapsto j(v)
		\]
	where $V$ is an open subset which contains $0$ of a complex vector space $E$ of dimension $\dim_{\R} E = 6g-6 + 2\sharp M - 2\sharp D-2$.
	We call such a family $v\mapsto j(v)$ an \textbf{alternative good complex deformation} if it satisfies the following:
	\begin{itemize}
		\item $j(0)=j$.
		\item The family is constant on the disk-like neighborhoods, i.e., $j(v)=j$ on $D_z$ for every $z\in M\cup \abs{D}$.
		\item For every $v\in V$ the Kodaira--Spencer differential
		\[
		[Dj(v)] : H^1 (\Sigma,j,M\setminus\{z_k\},D) \to H^1 (\Sigma,j(v),M\setminus\{z_k\},D)
		\]
		is a complex linear isomorphism.
		\item There exists a natural action
		\[
		\Aut(\Sigma,j,M,D) \times V \times B_\ep \to V\times B_\ep,\qquad (\phi,v,y)\mapsto (\phi\ast v, \phi \ast y)
		\]
		such that $\phi: (\Sigma,j(v)) \to (\Sigma,j(\phi \ast v))$ is biholomorphic and which furthermore satisfies $\phi(\varphi(y)) = \varphi(\phi \ast y)$.
	\end{itemize}
\end{definition}

\begin{definition}
	Let $(\Sigma,j,M,D)$ be a stable noded Riemann surface and let $v\mapsto j(v)$ be a good complex deformation.
	We may define an \textbf{alternative good uniformizing family} centered at $(\Sigma,j,M,D)$ as the following family of stable noded Riemann surfaces
	\[
	(a,v,y) \mapsto (\Sigma_a, j(a,v), M_{(a,y)}, D_a), \quad (a,v,y)\in (\hb)^{\sharp D} \times V \times B_\ep
	\]
	with data as follows.
	\begin{itemize}
		\item As in Definition~\ref{def:family-Riemann-surfaces} the glued surface $\Sigma_a$ is given by
		\[
		\Sigma_a = \left(\Sigma \setminus \cup_{\{x_a,y_a\}\in D} (D_{x_a}\sqcup D_{y_a} )\right) \sqcup Z_a.
		\]
		\item The complex structure $j(v)$ on the unglued Riemann surfaces $\Sigma$ induces the following complex structure on $\Sigma_a$:
		\[
		j(a,v):= 
		\begin{cases}
		j(v) &\text{ on } \Sigma \setminus \cup_{\{x_a,y_a\}\in D} (D_{x_a}\sqcup D_{y_a} ),\\
		i &\text{ on } Z_a \text{ for every pair } \{x_a,y_a\}\in D.
		\end{cases}
		\]
		\item The set of marked points $z'_1,\ldots,z'_{k-1} \in M_{(a,y)}$ are given by the former marked points which by construction all lie in $\Sigma \setminus \cup_{\{x_a,y_a\}\in D} (D_{x_a}\sqcup D_{y_a})$.  The $k$th-marked point is parametrized by the map $\varphi : B_\epsilon \hookrightarrow \Sigma$, i.e.,
		\[
		z'_k := \varphi(y) \in \Sigma \setminus \cup_{\{x_a,y_a\}\in D} (D_{x_a}\sqcup D_{y_a}).
		\]
		\item The set of nodal pairs $D_a$ is obtained from $D$ by deleting every nodal pair $\{x_a,y_a\}\in D$ for which $a\neq 0$.
	\end{itemize}
\end{definition}

There exists an analog of Proposition~\ref{prop:natural-representation} associated to an alternative good uniformizing family.

\subsubsection{Local smooth structures on the logarithmic and exponential Deligne--Mumford orbifolds}

An orbifold is locally homeomorphic to the quotient of an open subset of $\R^n$ by a finite group action.
In the current context, given a local uniformizer for a good uniformizing family of stable Riemann surfaces, the projection to an equivalence class is $\Aut$-invariant. Hence we obtain a well-defined map
	\[
	\frac{\{(\Sigma_a,j(a,v), M_a,D_a)\}_{(a,v)\in U}}{\Aut(\Sigma,j,M,D)} \to \dmspace_{g,k}.
	\]
This map is a local homeomorphism (see \cite[\S~2.1]{HWZGW}).

However, this is only a small part of the orbifold structure of the DM-spaces; in order to understand the full smooth orbifold structure it is necessary to construct an ep-groupoid structure on $\dmspace_{g,k}$. This is discussed in \cite[pp.~28--31]{HWZGW}.

Remember that the construction of a family of Riemann surfaces, and hence of the good uniformizing families required the choice of a gluing profile.
Ergo different choices of gluing profile yield different smooth orbifold structures for the same underlying topological space $\dmspace_{g,k}$.

We will be especially concerned with the following two gluing profiles: the \textbf{logarithmic gluing profile}
	\[
	\varphi_\text{log}: (0,1]\rightarrow [0,\infty), \quad r\mapsto  -\frac{1}{2\pi}\log(r).
	\]
and the \textbf{exponential gluing profile}
	\[
	\varphi_\text{exp}: (0,1]\rightarrow [0,\infty),\quad  r\mapsto e^{1/r}-e.
	\]

\begin{theorem}[{\cite[Thms.~2.14,~2.16]{HWZGW}}]
	Using the logarithmic gluing profile, we reproduce the classical Deligne--Mumford theory and obtain a complex orbifold we denote as $\dmlog_{g,k}$.  Conversely, using the exponential gluing profile we obtain a smooth oriented orbifold $\dmexp_{g,k}$. In both cases, the real dimension is equal to $6g-6 + 2k$.
\end{theorem}

\subsection{The Gromov--Witten polyfolds}
	\label{subsec:the-gw-polyfolds}

We now describe the underlying set of the GW-polyfolds.

\begin{definition}
	\label{def:gw-polyfold}
	For a fixed homology class $A\in H_2 (Q;\Z)$, and for fixed integers $g\geq 0$ and $k\geq 0$, \textbf{the underlying set of the Gromov--Witten polyfold} $\cZ_{A,g,k}$ is defined as the set of stable curves with homology class $A$, arithmetic genus $g$, and $k$ marked points
	\[
	\cZ_{A,g,k}:= \{ (\Sigma,j,M,D,u)  \mid  \cdots, \text{ GW-stability condition} \}/ \sim
	\]
	where $(\Sigma,j,M,D)$ is a noded Riemann surface except that we do not require the DM-stability condition \eqref{eq:dm-stability-condition}, with data as follows.
	\begin{itemize}
		\item $u:\Sigma \to Q$ is a continuous map such that $u_*[\Sigma]=A \in H_2 (Q;\Z)$.
		\item For each nodal pair $\{x,y\}\in D$ we have $u(x)=u(y)$.
		\item The map $u$ is of class $H^{3,\delta_0}$ at the nodal points in $\abs{D}$ and of class $H^3_\text{loc}$ near the other points in $\Sigma$ (see Definition~\ref{def:class-3delta} below).
		\item $\int_C u^*\ww \geq 0$ for each connected component $C\subset \Sigma$.
		\item For each connected component $C\subset \Sigma$ the following \textbf{GW-stability condition} holds.  We require at least one of the following:
		\begin{equation}
		2 g_C+\sharp (M\cup\abs{D})_C \geq 3 \quad \text{or} \quad \int_C u^* \ww > 0, \label{eq:gw-stability-condition}
		\end{equation}
		where $g_C$ is the genus of $C$ and $\sharp (M\cup\abs{D})_C$ is the number of marked and nodal points on the component $C$.
		\item The equivalence relation is given by $(\Sigma,j,M,D,u)\sim (\Sigma',j',M',D',u')$ if there exists a biholomorphism $\phi:(\Sigma,j)\to (\Sigma',j')$ such that $u'\circ\phi = u$, in addition to $\phi(M)=M',\ \phi(\vert D\vert)=\vert D'\vert$, and which preserves ordering and pairs.
	\end{itemize}
	We call a tuple $(\Sigma,j,M,D,u)$ which satisfies these requirements a \textbf{stable map}, and call an equivalence class $[\Sigma,j,M,D,u]$ a \textbf{stable curve}.
\end{definition}

\begin{definition}[{\cite[Def.~1.1]{HWZGW}}]
	\label{def:class-3delta}
	Let $u: \Sigma \to Q$ be a continuous map, and fix a point $z\in \Sigma$.
	We consider a local expression for $u$ as follows. Choose a small disk-like neighborhood $D_z\subset \Sigma$ of $z$ such that there exists a biholomorphism $ \sigma:[0,\infty)\times S^1\rightarrow D_z \setminus\{z\}$.
	Let $\varphi:U\rightarrow \R^{2n}$ be a smooth chart on a neighborhood $U \subset Q$ of $u(z)$ such that to $\varphi (u(z))=0$.
	The local expression 
		\[
		\tilde{u}: [s_0,\infty)\times S^1 \to \R^{2n}, \qquad (s,t) \mapsto \varphi\circ u\circ \sigma(s,t)
		\]
	is defined for $s_0$ large.
	
	Let $m\geq 3$ be an integer, and let $\delta >0$. We say that $u$ is of \textbf{class $H^{m,\delta}$ around the point $z\in \Sigma$} 
	if $e^{\delta s}  \tilde{u}$ belongs to the space $L^2([s_0,\infty)\times S^1,\R^{2n})$.
	We say that $u$ is of \textbf{class} $H^m_{\text{loc}}$ \textbf{around the point $z\in \Sigma$} if $u$ belongs to the space $H^m_{\text{loc}}(D_z)$.
	If $u$ is of class $H^{m,\delta}$ at a point $z\in \Sigma$ we will refer to that point as a \textbf{puncture}.
	
	These definitions do not depend on the choices involved of holomorphic polar coordinates on $\Sigma$ or smooth charts on $Q$.
\end{definition}

Some situations will require that the map
$u$ is of class $H^{3,\delta_0}$ at a fixed subset of the marked points, in addition to the nodal points.
Allowing a puncture at an $i$th-marked is a global condition on our polyfold, and so we may
add the following to the above conditions on the set $\cZ_{A,g,k}$:
\begin{itemize}
	\item We require that $u$ is of class $H^{3,\delta_0}$ at all marked points in a fixed subset of the index set $\{1,\ldots,k\}$.
\end{itemize}
By \cite[Cor.~1.6]{schmaltz2019naturality} the polyfold GW-invariants are independent of a fixed choice of puncture at the marked points.
As an aside, consider a point $z\in \Sigma$ with a small disk neighborhood $D_z\subset \Sigma$ as above.
Consider the spaces $H^{3,\delta_0} (D_z\setminus\{z\})$ and $H^3_\text{loc} (D_z)$;  as it turns out, neither of these spaces contains the other.
As an aside, one can show that there exists an inclusion map $H^{3,\delta_0} (D_z\setminus \{z\}) \hookrightarrow C^0 (D_z)$.

For a fixed $(\Sigma,j,M,D,u)$ the \textbf{isotropy group}
	\[
	\Aut(\Sigma,j,M,D,u) := 
	\left\{
	\begin{array}{c}
	\phi: (\Sigma,j) \to (\Sigma,j)
	\end{array}
	\biggm|
	\begin{array}{c}
	\phi \text{ biholomorphic}, 
	u \circ \phi = u, \\
	\phi(M)=M \text{ preservers ordering},\\
	\phi(\vert D\vert)=\vert D\vert
	\end{array}
	\right\}
	\]
is finite if and only if the GW-stability condition \eqref{eq:gw-stability-condition} holds.

\begin{theorem}[{\cite[Thm.~3.27]{HWZGW}}]
	The set of stable curves $\cZ_{A,g,k}$ has a natural second countable, paracompact, Hausdorff topology.
\end{theorem}

Before describing a fully general good uniformizing family of stable maps, we will consider special cases of such families which allow us to isolate:
\begin{enumerate}
	\item the gluing construction of stable maps in a region of a nodal Riemann surface,
	\item the transversal constraint construction in the case that a domain component does not satisfy the DM-stability condition \ref{eq:dm-stability-condition}.
\end{enumerate}


\subsubsection{Good uniformizing families centered at unnoded stable maps in the stable case}

The simplest description of a good uniformizing family centered at a stable map occurs when the stable map is without nodes and whose underlying Riemann surface is already stable.
Thus, consider an unnoded stable map $(\Sigma,j,M,\emptyset,u)$.
Let us assume moreover that $2g+k \geq 3$; it follows that $(\Sigma,j,M,\emptyset)$ is a stable Riemann surface, and we can let $v\mapsto j(v)$, $v\in V$ be a good complex deformation.

Consider the space of sections $H^3(\Sigma, u^*TQ)$; for a given Riemannian metric $g$ on $Q$.
There exists a sufficiently small neighborhood $U\subset H^3(\Sigma, u^*TQ)$ such that the associated exponential map defines a map $\exp_u \eta : \Sigma \to Q$.

In this case, a good uniformizing family centered at the stable map $(\Sigma,j,M,\emptyset,u)$ can be given as the following family of stable maps
	\[
	(v, \eta) \mapsto (\Sigma, j(v), M, \emptyset, \exp_u \eta) \quad \text{where } (v,\eta) \in V \times U \subset E\times H^3(\Sigma, u^*TQ).
	\]

\subsubsection{The gluing construction for a stable map}
	\label{subsubsec:the-gluing-construction-stable-map}

The description of a good uniformizing family is complicated by the fact that stable maps may be defined on nodal domains.
To examine this situation, let us consider a stable map with a single nodal pair, and such that that every domain component is stable.
In order to describe the good uniformizing family centered at such a stable map, we recall the gluing construction for a stable map as described in \cite[\S~2.4]{HWZGW}.

At the outset, fix a smooth cutoff function $\beta:{\mathbb R}\rightarrow [0,1] $ which satisfies the following:
\begin{itemize}
	\item $\beta (-s)+\beta (s)=1$ for all $s\in \R$,
	\item $\beta (s)=1$ for all $s\leq -1$,
	\item $\tfrac{d}{ds}\beta(s)<0$ for all $s\in (-1, 1)$.
\end{itemize}
Consider a pair of continuous maps
	\[
	h^+:\R^+\times S^1\to \R^{2n}, \qquad h^-:\R^-\times S^1\to \R^{2n};
	\]
with common asymptotic constant $c:= \lim_{s\to\infty} h^+(s,t) = \lim_{s\to-\infty} h^-(s,t)$.
For a given gluing parameter $a\in \hb$ we define the \textbf{glued map} $\oplus_a (h^+,h^-) : Z_a \to \R^{2n}$ by the interpolation
	\[
	\oplus_a (h^+,h^-) (s,t):=
	\begin{cases}
	\beta\left(	s - \tfrac{R_a}{2}	\right) \cdot h^+(s,t) &	\\
	\qquad+\left(1-\beta \left(	s-\tfrac{R_a}{2}	\right)\right) \cdot  h^-(s-R_a,t-\theta_a) &\text{if } a\neq 0,\\
	(h^+,h^-) & \text{if }a=0.
	\end{cases}
	\]

Now consider a stable map with a single nodal pair $(\Sigma,j,M,\{x_a,y_a\}, u)$, and such that every connected component $C\subset \Sigma$ is stable. Since $(\Sigma,j,M,\{x_a,y_a\})$ is assumed to be stable there exists a good complex deformation $v\mapsto j(v), v\in V$.

Consider the following space of sections $H^{3,\delta_0}_c(\Sigma, u^*TQ)$, consisting of sections $\eta : \Sigma \to u^*TQ$ such that:
$\eta$ is of class $H^{3,\delta_0}$ around the nodal points and of class $H^3_\text{loc}$ at the other points of $\Sigma$, and in addition $\eta$ has matching asymptotic values at the nodal pairs, i.e., $\eta(x_a) = \eta(y_a)$ for $\{x_a,y_a\}\in D$.
We now consider local expressions for the map $u$ and the section $\eta$ as follows.
In a neighborhood of the point $u(x_a)=u(y_a) \in Q$ choose a chart which identifies $u(x_a)=u(y_a)$ with $0\in \R^{2n}$. 
Furthermore, given a Riemannian metric $g$ on $Q$ we may assume that this chart is chosen such that this metric is identifiable with the Euclidean metric on $\R^{2n}$.
Choose small disk neighborhoods at $x_a$ and $y_a$ such that via biholomorphisms we may identify $D_{x_a}\setminus \{x_a\} \simeq \R^+\times S^1$ and $D_{y_a}\setminus \{y_a\} \simeq \R^-\times S^1$.
Localized to these coordinate neighborhoods, we may view the base map $u$ as a pair of maps 
	\[u^+:\R^+\times S^1\to \R^{2n}, \qquad u^-:\R^-\times S^1\to \R^{2n}\]
and likewise the section $\eta$ as maps
	\[\eta^+:\R^+\times S^1\to \R^{2n}, \qquad \eta^-:\R^-\times S^1\to \R^{2n}.\]

Given a gluing parameter $a\in \hb$ we define the \textbf{glued stable map} $\oplus_a \exp_u (\eta) :\Sigma_a \to Q$ as follows:
	\[
	\oplus_a \exp_u \eta :=
	\begin{cases}
	\exp_u \eta 							& \text{ on } \Sigma \setminus (D_{x_a}\sqcup D_{y_a}), \\
	\oplus_a (u^+ + \eta^+, u^- + \eta^- ) 	& \text{ on } Z_a.
	\end{cases}
	\]

At this point, we might hope to define a good uniformizing family of stable maps centered at $(\Sigma,j,M,\{x_a,y_a\},u)$ as follows:
	\begin{align*}
	&(a,v,\eta) \mapsto (\Sigma_a, j(a,v), M_a, \{x_a,y_a\}_a, \oplus_a \exp_u \eta), \\
	&\qquad \text{where } (a,v,\eta) \in \hb \times V \times U \subset \C \times E \times H^{3,\delta_0}_c (\Sigma,u^*TQ),
	\end{align*}
where $U$ is a suitably small neighborhood of the zero section.
The problem with this is that this map is not injective; to see this, fix a gluing parameter $a\neq 0$ and consider two sections $\eta$, $\eta'$ which differ only in a sufficiently small neighborhood of the nodal pair $\{x_a,y_a\}$; the glued stable maps will be identical, $\oplus_a \exp_u  \eta = \oplus_a \exp_u \eta'$.

\begin{theorem}[{\cite[Thm.~2.49]{HWZGW}}]
	\label{thm:ssc-retraction}
	There exists a ``$\ssc$-retraction'' (i.e., a $\ssc$-smooth map which satisfies $\pi \circ \pi = \pi$),
		\begin{align*}
			\pi: \hb \times V \times U 	& \to \hb \times V \times U \\
					(a,v,\eta)				& \mapsto (a,v,\pi_a (\eta)),
		\end{align*}
	such that the restriction of the above family to the subset $\cV := \pi (\hb \times V \times U)$ is injective.
\end{theorem}

We therefore define a good uniformizing family of stable maps centered at $(\Sigma,j,M,\{x_a,y_a\},u)$ by
	\[
	(a,v,\eta) \mapsto (\Sigma_a, j(a,v), M_a, \{x_a,y_a\}_a, \oplus_a \exp_u \eta) \quad \text{where } (a,v,\eta) \in \cV.
	\]

\begin{remark}
	The exponential gluing profile is necessary to prove the scale smoothness of the above map $\pi$; see \cite[Thms.~1.27,~1.28]{HWZsc}, the relevant calculations involving the exponential gluing profile appear in \cite[\S~2.3]{HWZsc}).
\end{remark}

\subsubsection{The transversal constraint construction}

The description of a good uniformizing family is again complicated by the possibility that for a given stable map the underlying Riemann surface may contain unstable components. For the sake of explaining the phenomena in a simple case, let us consider an unnoded stable map $(\Sigma,j,M, \emptyset, u)$ whose underlying Riemann surface $(\Sigma,j,M,\emptyset)$ is not stable.
Let us assume the automorphism group is the identity, $\Aut(\Sigma,j,M, \emptyset, u)= \{\id \}$.
Since this is a stable map, it necessarily follows from the GW-stability condition \eqref{eq:gw-stability-condition} that $\int_\Sigma u^*\ww >0$.
We note that $\Sigma$ consists of a single component of genus $0$ or $1$. If $g=0$ then $\sharp M <3$, and if $g=0$ then $M=\emptyset$.

Suppose that $\Sigma = S^2$, and suppose that there is a single marked point, $M=\{z_1\}$ (the cases where there are no marked points or two marked points are similar).
Recall that up to diffeomorphism, any complex structure on $S^2$ is the standard complex structure.
In order to stabilize $S^2$, choose a finite set of unordered points $S\subset S^2$ such that:
\begin{itemize}
	\item $S$ is disjoint from the marked point, in this case, $\{z_1\}\cap S = \emptyset$,
	\item $\sharp S \geq 2$, i.e.,  $(S^2,i,\{z_1\}\cup S,\emptyset)$ satisfies the DM-stability condition \eqref{eq:dm-stability-condition},
	\item the image $u(S)$ is disjoint from $u(z_1)$,
	\item at each point $z \in S$ the differential $du_z : T_z S^2 \to T_{u(z)} Q$ is injective, the pullback $u^*\ww$ is non-degenerate on $T_zS^2$, and the induced orientation of $u^*\ww$ on $T_z S^2$ agrees with the orientation determined by the standard complex structure $i$.
\end{itemize}
Such a stabilization always exists \cite[Lem.~3.2]{HWZGW}.
Let $l :=\sharp S -2$, and note that $\dim \dmlog_{0,1 + \sharp S} = 2l$.
For any three distinct points on the sphere there exists a unique M\"{o}bius transformation $\phi\in \text{PGL}(2,\C)$ which sends these points to $\{0,1,\infty\}\in S^2$.
Let $S_0:=\{y',y''\}$ denote the first two points of $S$; by fixing the three points $M\cup S_0$ and by parametrizing the remaining $\sharp S -2$ stabilizing points via maps $\phi_i : B_\ep \hookrightarrow S^2$, one can associate to a parameter $y\in B^{2l}_\ep$ the stabilization $S_y := S_0 \cup \{ \phi_i(y_i)\}_{1\leq i\leq l}$; we write $z_y\in S_y$ for one of the parametrized points.
Thus, we obtain a good uniformizing family centered at the stabilized sphere $(S^2,i,M\cup S,\emptyset)$:
	\[
	y \mapsto (S^2,i, \{z_0\} \cup S_0 \cup S_y,\emptyset), \qquad y \in B_\ep^{2l}.
	\]

Intuitively, each stabilizing point ``increases'' the dimension by $2$,
To get the correct dimension, we can place a codimension $2$ constraint at the stabilizing points on the space of sections.
This can be done as follows.

Using the requirement that at each point $z \in S$ the differential $du_z : T_z \Sigma \to T_{u(z)} Q$ is injective, choose a $2n-2$-dimensional complement $H_{u(z)}$ such that
	\[
	du_z (T_z\Sigma ) \oplus H_{u(z)} = T_{u(z)} Q.
	\]
We call such a complement a \textbf{linear constraint} associated with the point $z\in S$. 
We may also identify a neighborhood of zero in $H_{u(z)}$ with an embedded submanifold of $Q$ via the exponential map.
We will restrict our family of stable maps to sections which satisfy the linear constraint the following space $H_{u(z)}$ at associated (parametrized) point $z_y \in S_y$, i.e., we restrict to an open neighborhood $U$ of zero in the space 
	\[
	E_S := \{ (y,\eta) \in B^{2l}_\ep \times H^3 (\Sigma, u^*TQ)	\mid \eta(z_y) \in H_{u(z)} \text{ for } z_y \in S_y\}.
	\]

Thus using the above \textbf{transversal constraint construction}, we define a good uniformizing family of stable maps centered at $(S^2,i,\{z_0\},\emptyset,u)$ by
	\[
	(y, \eta) \mapsto (S^2,i,\{z_0\},\emptyset, \exp_u \eta), \qquad (y,\eta) \in U \subset E_S.
	\]
This construction is justified by the following assertion: the projection to the space of stable curves
	\[
	(y,\eta)  \mapsto [S^2,i, \{z_0\}, \emptyset, \exp_u \eta], \qquad (y,\eta)\in U \subset E_S
	\]
is a local homeomorphism (for the appropriate topology on the space of stable curves).

Indeed, it is locally surjective. Consider a nearby stable curve, and let $(S^2,i,\{w_0\}, \emptyset, v)$ be a stable map representative which is also near $(S^2,i,\{z_0\},\emptyset,u)$.
Observe that for any map $v$ sufficiently close to $u$ will have the property that it intersects the submanifold associated to each linear constraint $H_{u(z)}$ precisely once, and intersects transversally. 
Let $\phi:S^2\to S^2$ be the unique M\"{o}bius transformation which takes $\{z_0,y',y''\}$ to $\{w_0\}$ and the unique points of intersection of the map $v$ with the linear constraints associated to $y'$ and $y''$.
We then define $\eta$ uniquely by $\exp_u \eta = v\circ \phi$ and $y$ by the unique points of intersection of $v\circ \phi$ with the constraints $H_{u(z)}, z\in S$.
Moreover, since these parameters are uniquely determined the projection is injective.

\subsubsection{Good uniformizing families of stable maps in the general case}
	\label{subsubsec:families-stable-maps-in-general}

Hopefully the above prototypical cases provide some intuition for the gluing/transversal constraint constructions.
In general, both constructions will be needed in the construction of a general good uniformizing family---clearly, a stable map may contain nodal points as well as unstable domain components.
There is no issue with using both constructions at the same time, since these constructions are localized to disjoint regions of the underlying Riemann surface.

Consider a stable map $(\Sigma,j,M,D,u)$ and consider the associated automorphism group $\Aut(\Sigma,j,M,D,u)$.
We now recall the full definition of a stabilization \cite[Def.~3.1]{HWZGW} in the general case. A \textbf{stabilization} is a set of points $S\subset \Sigma$ which satisfy the following:
\begin{itemize}
	\item $S$ is disjoint from the special points, i.e., $S\cap (M\cup\abs{D}) = \emptyset$,
	\item given an automorphism $\phi \in \Aut(\Sigma,j,M,D,u)$ then $\phi(S) =S$,
	\item the tuple $(\Sigma,j,M\cup S,D)$ satisfies the DM-stability condition \eqref{eq:dm-stability-condition},
	\item if $u(z)=u(z')$ for $z,z'\in S$ then there exists an automorphism $\phi$ such that $\phi(z)=z'$,
	\item the image $u(S)$ is disjoint from $u(M)$ and from $u(D)$,
	\item at each point $z \in S$ the differential $du_z : T_z \Sigma \to T_{u(z)} Q$ is injective, the pullback $u^*\ww$ is non-degenerate on $T_z\Sigma$, and the induced orientation of $u^*\ww$ on $T_z \Sigma$ agrees with the orientation determined by the complex structure $j$.
\end{itemize}
Again, such a stabilization always exists \cite[Lem.~3.2]{HWZGW}.

The Riemann surface $(\Sigma,j,M\cup S,D)$ is now stable; let 
	\[
	(a,v)\mapsto (\Sigma_a, j(a,v), (M\cup S)_a, D_a), \quad (a,v) \in \hb^{\sharp D} \times V
	\]
be a good uniformizing family of stable noded Riemann surfaces.

By the final condition of a stabilization, we may choose a $(2n-2)$-dimensional complement $H_{u(z)}$ such that
	\[
	du_z (T_z\Sigma ) \oplus H_{u(z)} = T_{u(z)} Q.
	\]
We call such a complement a \textbf{linear constraint} associated with the point $z\in S$. 

Consider the constrained subspace of sections
	\[
	E_S := \{ \eta \in H^{3,\delta_0}_c (u^*TQ)	\mid \eta(z_s) \in H_{u(z_s)} \text{ for } z_s \in S\}
	\]
and let $U \subset E_S$ be an suitably small open neighborhood of the zero section.
As before, by using the gluing construction one can define a $\ssc$-retraction
	\begin{align*}
	\pi : \hb^{\sharp D} \times V \times U 	& \to \hb^{\sharp D} \times V \times U\\
			(a,v,\eta)									& \mapsto (a,v,\pi_a(\eta)).
	\end{align*}
Then the image $\cV : = \pi (\hb^{\sharp D} \times V \times U)$ is a $\ssc$-retract on which the gluing map is injective.

\begin{definition}[{\cite[Def.~3.9]{HWZGW}}]
	\label{def:good-uniformizing-family-of-stable-maps}
	Having chosen a stabilization $S$, a \textbf{good uniformizing family of stable maps} centered at $(\Sigma,j,M,D,u)$ is a family of stable maps
	\[
	(a,v,\eta) \mapsto (\Sigma_a, j(a,v), M_a, D_a, \oplus_a \exp_u \eta), \quad (a,v,\eta )\in \cV.
	\]
	In particular, the section $\eta$ satisfies the linear constraint $H_{u(z_s)}$ at each stabilizing point $z_s\in S$.
	In addition, the glued stable map may be described on the glued Riemann surface as follows:
	\[
	\oplus_a \exp_u \eta :=
	\begin{cases}
	\exp_u \eta 			& \text{on } \Sigma \setminus \cup_{\{x_a,y_a\}\in D} (D_{x_a}\sqcup D_{y_a}), \\
	\oplus_a \exp_u \eta 	& \text{on } Z_a \text{ for every gluing parameter } a \in \hb.
	\end{cases}
	\]
	We call the parameters $(a,v,\eta)$ \textbf{local $\ssc$-coordinates} centered at the stable map $(\Sigma,j,M,D,u)$.
\end{definition}

\begin{proposition}[{\cite[Prop.~3.12]{HWZGW}}]
	Consider a stable map $(\Sigma,j,M,D,u)$ and let $G:= \Aut(\Sigma,j,M,D,u)$. Consider a good uniformizing family of stable maps centered at $(\Sigma,j,M,D,u)$:
		\[
		(a,v,\eta) \mapsto (\Sigma_a, j(a,v), M_a, D_a, \oplus_a \exp_u \eta), \quad (a,v,\eta )\in \cV.
		\]
	For every open neighborhood of $0$ in $\cV$ there exists an open subneighborhood of $0$, call it $\cU$, and a group action
		\begin{align*}
		G\times \cU 		& \to \cU \\
		(\phi,(a,v,\eta)) 	& \mapsto (\phi\ast(a,v), \eta \circ \phi^{-1})
		\end{align*}
	(where $\phi\ast(a,v)$ is the action from equation \eqref{eq:natural-representation}) such that the following holds.
	\begin{enumerate}
		\item Let $(\phi,(a,v,\eta)) \in G\times \cU$, and denote $(b,w,\xi):= (\phi\ast (a,v), \eta\circ \phi^{-1})$. Then there exists a morphism
			\[
			\phi_{(a,v,\eta)} : (\Sigma_a,j(a,v),M_a,D_a, \oplus_a \exp_u \eta ) \to 
								(\Sigma_b,j(b,w),M_b,D_b, \oplus_b \exp_u \xi).
			\]
		\item Consider a morphism 
				\[
				\psi : (\Sigma_a,j(a,v),M_a,D_a, \oplus_a \exp_u \eta ) \to (\Sigma_{a'},j(a',v'),M_{a'},D_{a'}, \oplus_{a'} \exp_u \eta' )
				\]
			for parameters $(a,v,\eta), (a',v',\eta') \in \cU_S$. Then there exists a unique element $\phi \in G$ such that $(\phi\ast(a,v), \eta \circ \phi^{-1}) = (a',v',\eta')$ and moreover $\psi = \phi_{(a,v,\eta)}$.
	\end{enumerate}
	We call $\cU$ a \textbf{local uniformizer} centered at $(\Sigma,j,M,D,u)$.
\end{proposition}

\subsubsection{Local smooth structures on the Gromov--Witten polyfolds}

A polyfold is locally homeomorphic to the quotient of a $\ssc$-retract by a finite group action (compare with our mantra regarding the local topology of an orbifold.).
In the current context, given a local uniformizer for a good uniformizing family of stable maps, the projection to an equivalence class is $\Aut$-invariant. Hence we obtain a well-defined map
	\[
	\frac{\{(\Sigma_a,j(a,v), M_a,D_a, \oplus_a \exp_u \eta)\}_{(a,v,\eta)\in \cU}}{\Aut(\Sigma,j,M,D,u)} \to \cZ_{A,g,k}.
	\]
This map is a local homeomorphism (see \cite[\S~2.1]{HWZGW}). This further gives us a picture of the local smooth structure of a GW-polyfold

In order to understand the full smooth structure on $\cZ_{A,g,k}$, it is necessary to construct a polyfold structure on $\cZ_{A,g,k}$. We refer to \cite[\S~3.5]{HWZGW} for such a construction.

\begin{theorem}[{\cite[Thm.~3.37]{HWZGW}}]
	Having fixed the exponential gluing profile and a strictly increasing sequence $(\delta_i)_{i\geq 0}\subset (0,2\pi)$, the second countable, paracompact, Hausdorff topological space $\cZ_{A,g,k}$ possesses a natural equivalence class of polyfold structures.
\end{theorem}

It was further proven that the polyfold GW-invariants are independent of the choice of strictly increasing sequence $(\delta_i)_{i\geq 0}\subset (0,2\pi)$, see \cite[Cor.~1.5]{schmaltz2019naturality}.

\subsection{Other Gromov--Witten-type polyfolds}
	\label{subsec:other-gw-polyfolds}

We now introduce two variants of the GW-polyfolds which will play an important role in proving the GW-axioms.

Fix a pair $(g,k)$ of integers such that $g\geq 0$, $k\geq 0$, and $2g+k \geq 3$.
A \textbf{splitting} $S$ of $(g,k)$ consists of the following:
	\begin{itemize}
		\item a pair of integers $k_0,k_1 \geq 0$ such that $k_0 + k_1 = k$,
		\item a pair of integers $g_0,g_1 \geq 0$ such that $g_0 + g_1 = g$,
	\end{itemize}
such that the following holds: $k_0 + g_0 \geq 2$ and $k_1 + g_1 \geq 2$.

\begin{definition}
	Fix a homology class $A\in H_2(Q;\Z)$, and fix a pair of integers $g,\ k$ such that $g\geq 0$, $k\geq 0$ and $2g+k \geq 3$. 
	Let $A_0, A_1$ be such that $A_0+A_1=A$, and let $S$ be a splitting of $(g,k)$.
	The \textbf{underlying set of the split Gromov--Witten polyfolds} with respect to $S$ and $A_0,A_1$ is defined as 
		\[
		\cZ_{A_0+A_1,S} := \{(\Sigma_0, j_0, M_0, D_0 , u_0, \Sigma_1, j_1 , M_1, D_1, u_1  ) \mid \cdots, u_0 (z_{k_0+1}) = u_1 (z'_1) 		\} / \sim
		\]
	with data as follows.
	\begin{itemize}
		\item The $(\Sigma_i,j_i,M_i,u_i)$ is a stable map with respect to the GW-polyfold $\cZ_{A_i,k_i+1,g_i}$ for $i=0,1$.
		\item We require the following incidence relation between the last marked point of $M_0$ and the first marked point of $M_1$:
				\[u_0 (z_{k_0+1}) = u_1 (z'_1).\]
			  We moreover require that $u_0$ and $u_1$ are of class $H^{3,\delta_0}$ at the punctures $z_{k_0+1}$ and $z'_1$, respectively.
		\item As usual, the equivalence relation is given by biholomorphisms of the form $\phi:(\Sigma_0\sqcup \Sigma_1,j_0\sqcup j_1)\to (\Sigma'_0\sqcup \Sigma'_1,j'_0\sqcup j'_1)$ such that $\phi$ preserves the ordering of the marked points and maps pairs to pairs; notice that since $\phi$ preserves the $(k_0+1)$th and $1$st marked points this is already enough to imply that $\phi(\Sigma_i)=\Sigma'_i$ for $i=0,1$. We then require in addition that $u'_i\circ \phi = u_i$ for $i=0,1$.
	\end{itemize}
\end{definition}

There are multiple ways to give $\cZ_{A_0+A_1,S}$ a natural polyfold structure.
Consider the map
	\[
	ev_{k_0+1} \times ev_1 : \cZ_{A_0,g_0,k_0+1} \times \cZ_{A_1,g_1,k_1+1} \to Q\times Q;
	\]
it follows from Proposition~\ref{prop:evaluation-smooth-submersion} this map is transverse to the diagonal $\Delta \subset Q\times Q$.
The preimage $(ev_{k_0+1} \times ev_1)^{-1} (\Delta)$ may be identified with $\cZ_{A_0+A_1,S}$ and intuitively should possess a natural smooth structure.
Indeed, this is precisely the situation described by \cite{filippenko2018constrained}, which considers the general problem of developing a polyfold regularization theorem for constrained moduli spaces.
A natural polyfold structure on $\cZ_{A_0+A_1,S}$ follows from \cite[Thm.~1.5]{filippenko2018constrained}; we make the technical remark that this requires shifting the polyfold filtration up one level.

On the other hand, it is straightforward to directly define a polyfold structure on $\cZ_{A_0+A_1,S}$; this does not require any shifting of the polyfold filtrations.
In the construction of a good uniformizing family of stable maps we may simply restrict to sections $\eta_0\in H^{3,\delta_0}_c (u_0^*TQ)$,  $\eta_1\in H^{3,\delta_0}_c (u_1^*TQ)$ which have matching asymptotic values at the punctures $z_{k_0+1}$, $z'_1$, i.e.,
	\[
	\eta_0 (z_{k_0+1}) = \eta_1 (z'_1).
	\]
In essence, to construct an appropriate good uniformizing family we treat $\{z_{k_0}, z'_1\}$ as a nodal pair but without an associated gluing parameter.
Compatibility of such different good uniformizing families of stable maps may then be shown by considering compatibility in the general case but where we restrict a gluing parameter to zero.

\begin{definition}
	Fix a homology class $A\in H_2(Q;\Z)$ and fix a pair of integers $g,\ k$ such that $g\geq 1$, $k\geq 0$ and $2g+k \geq 3$.
	The \textbf{underlying set of the increased arithmetic genus Gromov--Witten polyfold (genus GW-polyfold)} is defined as
		\[
		\cZ^\text{g}_{A,g-1,k+2} := \{(\Sigma,j,M,D,u) \mid \cdots, u(z_{k+1})= u(z_{k+2}) \} / \sim
		\]
	with data as follows.
	\begin{itemize}
		\item The tuple $(\Sigma,j,M,D,u)$ is a stable map with respect to the GW-polyfold $\cZ_{A,g-1,k+2}$.
		\item We require the following incidence relation between the second-to-last and last marked points:
				\[u (z_{k+1}) = u (z_{k+2}).\]
			  We moreover require that $u$ is of class $H^{3,\delta_0}$ at the punctures $z_{k+1}$ and $z_{k+2}$.
		\item As usual, the equivalence relation is given by biholomorphisms of the form $\phi:(\Sigma,j)\to (\Sigma',j')$ such that $\phi$ preserves the ordering of the marked points and maps pairs to pairs, and such that $u'\circ \phi = u$.
	\end{itemize}
\end{definition}

Consider the map
	\[
	ev_{k+1} \times ev_{k+2} : \cZ_{A,g-1,k+2} \to Q\times Q.
	\]
As above, the preimage $(ev_{k+1} \times ev_{k+2})^{-1} (\Delta)$ may be identified with $\cZ^\text{g}_{A,g-1,k+2}$.
A natural polyfold structure on $\cZ^\text{g}_{A,g-1,k+2}$ can be seen exactly as in the above case of the split GW-polyfolds.

We remark that considering the $(k+1)$th and $(k+2)$th marked points as a nodal pair, a stable curve in $\cZ^\text{g}_{A,g-1,k+2}$ has arithmetic genus $g$ and $k$ marked points.
However, it is vital to note that the morphisms would differ with this viewpoint; a morphism must fix the $(k+1)$th and $(k+2)$th marked points, however it may permute the nodal pairs.
As an explicit example of this phenomena, consider the sphere $S^2$ with three marked points $\{0,1,\infty\}$; the only morphism is the identity.
However if we consider $\{0\}$ as the only marked point and consider $\{1,\infty\}$ as a nodal pair we obtain a surface with arithmetic genus $1$ with two distinct morphisms: the identity, and the element of $\text{PGL}(2,\C)$ which fixes $0$ and exchanges the points $1$ and $\infty$.


\section{Maps between Deligne--Mumford orbifolds/Gromov--Witten polyfolds}
	\label{sec:maps-between-orbifolds-polyfolds}

In this section we consider certain naturally defined maps between DM-orbifolds\slash GW-polyfolds.

\subsection{Maps between orbifolds/polyfolds}

In principle, the definition of a map between orbifold-type spaces is somewhat subtle; the appropriate notion is a ``generalized map'' which is itself an equivalence class of a functor between two ep-groupoids (see \cite[Def.~16.5]{HWZbook}).
However, in the present context this subtlety is not a concern due to the naturality of the maps we are considering; in these cases generalized maps arise naturally.
In the cases we consider, it is sufficient to define a map on the level of the underlying sets and then to check (scale) smoothness by writing appropriate local expressions in terms of good uniformizing families of stable noded Riemann surfaces/good uniformizing families of stable maps.

In the present context, it is sufficient to know the following: a generalized map $f: \cO \to \cP$ between orbifolds/polyfolds defines a continuous map between the underlying topological spaces (also written as $f: \cO \to \cP$) and also gives rise to an equivariant lift between local uniformizers:
	\[
	\begin{tikzcd}[row sep = small]
	U \arrow[r, "\hat{f}"] \arrow[d]			& V \arrow[d] \\
	U / G \arrow[r, "\abs{\hat{f}}"] \arrow[d]	& V / H \arrow[d] \\
	\cO \arrow[r, "f"]							& \cP.
	\end{tikzcd}
	\]
The lift $\hat{f}$ is equivariant with respect to an induced group homomorphism between the local isotropy groups, $G \to H$. 
Then the induced map $\abs{\hat{f}}$ between the quotients $U/G$, $V/H$ is identifiable via local homeomorphisms with a restriction of the continuous map $f: \cO \to \cP$.
Hence, (scale) smoothness may therefore be checked by considering the local expressions given by the lift $\hat{f}$.

We also recall the important definition of scale smoothness.
Consider two $\ssc$-smooth retractions $\pi : U \to U$, $\pi' : U' \to U'$, and let $O:= \pi (U),$ $O':= \pi' (U')$ be the associated $\ssc$-retracts.
A map between two $\ssc$-retracts, $f: O \to O'$ is \textbf{$\ssc$-smooth} (resp.\ $\ssc^k$) if the composition $f\circ \pi : U \to U'$ is $\ssc$-smooth (resp.\ $\ssc^k$).
Note that for a finite-dimensional space the identity is a $\ssc$-retract, and hence we also have a consistent notion of $\ssc$-smoothness when the source or target space is finite-dimensional.

\subsection{Natural maps between the Deligne--Mumford orbifolds/Gromov--Witten polyfolds}


To begin, we consider the \textbf{identity map} 
	\[\id : \dmspace^\varphi_{g,k} \to \dmspace^\phi_{g,k}\]
on the DM-orbifolds with respect to different gluing profiles $\varphi$, $\phi$.
Since the topology is independent of gluing profile, it is clear this map is continuous for any choices of gluing profiles.

We begin by comparing the gluing constructions for different choices of gluing profile; it is easy to see that different gluing parameters will produce identical glued Riemann surfaces 
	$\Sigma^\varphi_a = \Sigma^\phi_b$
precisely when $\varphi (r_a) = \phi (r_b)$ and $\theta_a = \theta_b$.

Using this fact, we may write a local expression for the identity at an arbitrary point $[\Sigma,j,M,D] \in \dmspace^\varphi_{g,k}$ in terms of good uniformizing families centered at a stable noded Riemann surface representative $(\Sigma,j,M,D)$ as follows:
	\begin{align*}
	\hat{\id} : \{\Sigma^\varphi_a,j(a,v),M_a,D_a \}_{(a,v)\in U}
				  &\to \{\Sigma^\phi_b,j(b,w), M_b,D_b \}_{(b,w)\in V} 	\\
			(r_a e^{-2\pi i\theta_a},v) &\mapsto (\phi^{-1}\circ \varphi(r_a) e^{-2\pi i\theta_a}, v).
	\end{align*}

\begin{proposition}
	\label{prop:identity-exp-to-log}
	The identity map from the exponential to the logarithmic DM-orbifolds,
	\[
	\id : \dmexp_{g,k} \to \dmlog_{g,k}
	\]
	is smooth.
\end{proposition}
\begin{proof}
	We assert that all local expressions for the identity map are smooth. 	This reduces to the claim that the function
		\[
		\C \to \C, \qquad r_a e^{-2\pi i\theta_a} \mapsto \varphi^{-1}_{\log}\circ \varphi_{\exp}(r_a) e^{-2 \pi i\theta_a}
		\]
	is smooth.
	The inverse of $\varphi_{\log}(r) = -\tfrac{1}{2\pi} \log (r)$ is given by $\varphi^{-1}_{\log} (R) = e^{-2\pi R}$, hence $\varphi^{-1}_{\log}\circ \varphi_{\exp}(r_a) = e^{-2\pi (e^{1/r_a}-e)}$.
	We may write this function in rectangular coordinates as
		\[
		\R^2 \to \R^2, \qquad (x,y) \mapsto \frac{\varphi^{-1}_{\log}\circ \varphi_{\exp}(\sqrt{x^2+y^2})}{\sqrt{x^2+y^2}} (x,y).
		\]
	This function is smooth, except possibly at $(0,0)$.
	Considering partial derivatives of the coordinate functions, one may see that it is enough to show that the coordinate function 
		\[
		\R \to \R, \qquad x \mapsto \sign (x) e^{-2\pi \left(e^{{1}/{\abs{x}}}-e\right)} 
		\]
	is smooth at $0$.
	The derivatives of this function may be computed using the chain rule; they will involve sums and multiples of terms of the form
		\[
		\frac{C  e^{-2\pi \left(e^{1/\abs{x}} -e\right) + a/\abs{x}}}{x^b}
		\]
	for a constant $C$ and positive integers $a,b$; such terms have limit $0$ as $x\to 0$.	
\end{proof}

\begin{remark}
	Consider now the identity map from the logarithmic to the exponential DM-orbifolds,
		\[\id : \dmlog_{g,k} \to \dmexp_{g,k}.\]
	As we have already observed this identity map is continuous; however in general it is not differentiable---the only the exception is the trivial case $(g,k)=(0,3)$, in which case $\dmspace_{0,3} = \{\pt\}$.
	As above, the gluing parameters transform according to the function
		\[
		\C \to \C, \qquad r_a e^{-2\pi i\theta_a} \mapsto \varphi^{-1}_{\exp}\circ \varphi_{\log}(r_a) e^{-2 \pi i\theta_a}.
		\]
	We may further write a coordinate function
		\[
		\R \to \R, \qquad x \mapsto \sign (x) \frac{1}{\log ( -\tfrac{1}{2\pi} \log (\abs{x}) + e )}.
		\]
	The limits as $x\to 0$ of the first derivatives of this function do not exist.
\end{remark}


The \textbf{evaluation map at the $i$th-marked point} is defined on the level of underlying sets by 
	\begin{align*}
		ev_i : 	\cZ_{A,g,k} &\to Q \\
				[\Sigma,j,M,D,u] &\mapsto u(z_i).
	\end{align*}
It is well-defined regardless of whether there is a puncture at the marked point $z_i$ or not.

\begin{proposition}
	\label{prop:evaluation-smooth-submersion}
	The evaluation map is a $\ssc$-smooth submersion.
\end{proposition}
\begin{proof}
	We check the $\ssc$-smoothness at an arbitrary stable curve $[\Sigma,j,M,D,u]\in \cZ_{A,g,k}$.
	Let $(a,v,\eta)\in\cU$ be a local uniformizer centered at a representative $(\Sigma,j,M,D,u)$.
	Choose a chart on $Q$ which identifies $u(z_i)\in Q$ with $0\in \R^{2n}$ and assume moreover this chart is chosen such that the given Riemannian metric $g$ on $Q$ is identifiable with the Euclidean metric on $\R^{2n}$.
	The local expression for a retraction composed with the evaluation map is
	\[
	(a,v,\eta) \mapsto (a,v,\pi_a (\eta)) \mapsto \exp_u\eta (z_i) = \eta(z_i).
	\]
	which is $\ssc$-smooth.
	The linearization of this expression also clearly spans the tangent space $T_{u(z_i)} Q$, which proves the claim that the evaluation map is submersive.
\end{proof}


The \textbf{projection map}, 
	\begin{align*}
	\pi:\cZ_{A,g,k}		& \rightarrow \dmspace_{g,k} \\
	[\Sigma,j,M,D,u]	& \mapsto  [(\Sigma,j,M,D)_{stab}],
	\end{align*}
is defined on the level of underlying sets by taking a stable curve and, after removing unstable components, associating the underlying stable domain.
	
We describe this process on the level of the underlying sets as follows.  Consider a point $[\Sigma,j,M,D,u]\in \cZ_{A,g,k}$, and let $(\Sigma,j,M,D,u)$ be a stable map representative.  
First forget the map $u$ and consider, if it exists, a component $C\subset \Sigma$ satisfying $2g_C+\sharp (M\cup |D|)_C<3$.
Then we have the following cases. 
	\begin{enumerate}
		\item $C$ is a sphere without marked points and with one nodal point, say $x$. Then  we remove the sphere, the nodal point $x$ and its partner $y$, where $\{x,y\}\in D$.
		\item $C$ is a sphere with two nodal points. In this case there are two  nodal pairs $\{x,y\}$ and $\{x',y'\}$, where $x$ and $x'$ lie on the sphere.  We remove the sphere and the two nodal pairs  but add the nodal pair $\{y,y'\}$.
		\item $C$ is a sphere with one node and one marked point. In that case we  remove the sphere but replace the corresponding nodal point  on the other component by the marked point.
	\end{enumerate}
Once we have removed all unstable components in this manner, we end up with a stable noded marked Riemann surface we denote as $(\Sigma,j,M,D)_{stab}$.

\begin{proposition}
	\label{prop:projection-map-smooth}
	The projection map 
	\begin{align*}
	\pi: \cZ_{A,g,k}	& \to \dmlog_{g,k}\\
	[\Sigma,j,M,D,u] 		& \mapsto [(\Sigma,j,M,D)_{\text{stab}}]
	\end{align*}
	defined between the GW-polyfold and the logarithmic DM-orbifold is $\ssc$-smooth.
\end{proposition}
\begin{proof}
	We check scale smoothness at an arbitrary stable curve $[\Sigma,j,M,D,u]\in \cZ_{A,g,k}$.  
	Consider a good uniformizing family of stable maps centered at a stable map representative:
		\[(a,v,\eta) \to (\Sigma_a,j(a,v),M_a,D_a,\oplus_a \exp_u \eta), \qquad (a,v,\eta)\in \cU\]
	Recall that the construction of a good uniformizing family of stable maps requires we choose a stabilization $S\subset \Sigma$ to the underlying Riemann surface in order to make it stable.
	Thus, we may also consider the following a good uniformizing family of stable Riemann surfaces:
		\[(a,v) \mapsto (\Sigma_a,j(a,v), (M\cup S)_a, D_a), \qquad (a,v)\in U\]
	Crucially, we use the exponential gluing profile for both of these good uniformizing families.
	
	We may write a local expression for the projection map as a composition in the following way.
	First, forget the stable map from the good uniformizing family,
	\begin{align*}
	\{(\Sigma_a,j(a,v),M_a,D_a,\oplus_a \exp_u \eta)\}_{(a,v,\eta)\in \cU}
		& \to \{(\Sigma_a,j(a,v), (M\cup S)_a, D_a)\}_{(a,v)\in U} \\
	(a,v,\eta) & \mapsto (a,v) 
	\end{align*}
	(note also that this expression is only defined locally due to the stabilization).
	Then switch from the exponential to the logarithmic gluing profile,
	\begin{align*}
	\{(\Sigma_a,j(a,v), (M\cup S)_a, D_a)\}_{(a,v)\in U} 
		& \xrightarrow{\hat{\id}} \{(\Sigma_b,j(b,w), (M\cup S)_b, D_b)\}_{(b,w)\in V} \\
	(r_a e^{-2\pi i\theta_a},v) & \mapsto (\varphi_{\log}^{-1} \circ \varphi_{\exp} (r_a) e^{-2 \pi i\theta_a}, v).
	\end{align*}
	We may then forget all of the stabilizing points in $S$, which we may write as a composition of maps which each forgets a single stabilizing point, by composing with the following sequence of maps:
	\[
	\dmlog_{g,k+\sharp S} \xrightarrow{ft} \dmlog_{g,k+\sharp S-1} \xrightarrow{ft} \cdots \xrightarrow{ft} \dmlog_{g,k}
	\]
	By Proposition~\ref{prop:ftk-log-dmspace} the forgetting maps on the logarithmic DM-orbifolds are smooth. Hence, the local expression for $\pi$ is a composition of $\ssc$-smooth maps, and is therefore smooth.

\end{proof}


Fix a permutation $\sigma \in S_k$ in the symmetric group, i.e., a bijection $\sigma: \{1,\ldots, k\}\to \{1,\ldots,k\}$.
The \textbf{permutation map} is defined on the underlying set of a DM-space
		\[
		\sigma: \dmspace_{g,k} \to \dmspace_{g,k}, \qquad [\Sigma,j,M,D]  \mapsto [\Sigma,j,M^\sigma,D]
		\]
or on the underlying set of a GW-polyfold 
		\[
		\sigma: \cZ_{A,g,k} \to \cZ_{A,g,k}, \qquad [\Sigma,j,M,D,u]  \mapsto [\Sigma,j,M^\sigma,D,u]
		\]
by relabeling the marked points as follows: with $M = \{z_1,\ldots,z_k\}$ we define $M^\sigma := \{z'_1,\ldots,z'_k\}$ by the relabeling $z'_i:= z_{\sigma(i)}$.

\begin{proposition}\label{permutation-map-sc-smooth}
	The permutation map between logarithmic or exponential DM-orbifolds is a diffeomorphism. The permutation map between GW-polyfolds is a $\ssc$-diffeomorphism.
\end{proposition}
\begin{proof}
	In either case, the choice of good uniformizing family (along with data such as the stabilization) may be made identically, without regard for the specific numbering of the marked points. The permutation map in such local coordinates will be the identity.
\end{proof}


\subsection{The \texorpdfstring{$k$th-marked}{kth-marked} point forgetting map on the logarithmic Deligne--Mumford orbifolds}
	\label{subsec:kth-marked-point-forgetting-map-log-dm}
Suppose that $2g+k >3$.  We define the \textbf{$k$th-marked point forgetting map}
	\[
	ft_k :\dmspace_{g,k} \to \dmspace_{g,k-1}
	\]
on the underlying sets of the DM-spaces as follows.  Let $[\Sigma,j,M,D]\in \dmspace_{g,k}$, and let $(\Sigma,j,M,D)$ be a representative given by a stable noded Riemann surface.  To define $ft_k ([\Sigma,j,M,D])$ we distinguish three cases for the component $C_k$ which contains the $k$th-marked point, $z_k\in C_k$.
	\begin{enumerate}
		\item \label{case-1}
		The component $C_k\setminus \{z_k\}$ satisfies the DM-stability condition. It follows that  $2 g_{C_k} + \sharp (M \cup \abs{D})_{C_k}>3$ and we define 	
			\[
			ft_k ([\Sigma,j,M,D]) = [\Sigma,j,M\setminus\{z_k\}, D].
			\]
		\item The component $C_k\setminus \{z_k\} $ does not satisfy the DM-stability condition.  It follows that $2 g_{C_k} + \sharp (M \cup \abs{D})_{C_k}=3$.  From $2g+k>3$ and the assumption that the set $\Sigma / \sim$ (where $x_a\sim y_a$ for nodal pairs $\{x_a,y_a\}\in D$) is connected, we can deduce
		\[
		2 g_{C_k} = 0, \qquad \sharp (M \cup \abs{D})_{C_k}=3.
		\]
		There are now two possibilities.
		\begin{enumerate}
			\item[(2a)] \label{case-2a}
			$z_k\in M_{C_k}$, and two nodal points $z_a, z_b \in \abs{D}_{C_k}$
			Then $ft_k ([\Sigma,j,M,D])$ is defined as the equivalence class of the stable noded Riemann surface obtained from $(\Sigma,j,M,D)$ as follows.  
			Delete $z_k$, delete the component $C_k$, and delete the two nodal pairs.  We add a new nodal pair $\{x_a,y_b\}$ given by two points of the former nodal pairs.
			\item[(2b)] \label{case-2b}
			$z_i, z_k\in M_{C_k}$, and one nodal point $z_a \in \abs{D}_{C_k}$
			Then $ft_k ([\Sigma,j,M,D])$ is defined as the equivalence class of the stable noded Riemann surface obtained from $(\Sigma,j,M,D)$ as follows.
			Delete $z_k$, delete the component $C_k$, and delete the nodal pair.  We add a new marked point $z_i$, given by the former nodal point which did not lie on $C_k$.
		\end{enumerate}
	\end{enumerate}


\begin{remark}[Special cases for the $k$th-marked point forgetting map]
	\label{rmk:special-cases-forgetting}
	By definition, if $2g+(k-1) < 3$ then $\dmspace_{g,k-1}=\emptyset$; hence, one can also consider the trivial map $ft_k : \dmspace_{g,k} \to \emptyset.$
	
	By considering possible stable configurations, one can see that $ft_k$ is given by case 1 whenever $[\Sigma,j,M,D]$ lies on the top dimensional stratum of the DM-space.
	Additionally, there will exist a point $[\Sigma,j,M,D]$ with the configuration specified by case 2a whenever 
	\[
	(g=0, k\geq 5), \quad (g=1,k=1), \quad (g\geq 2, k\geq 1).
	\]
	Finally, there will exist a point $[\Sigma,j,M,D]$ with the configuration specified by case 2b whenever 
	\[
	(g=0, k\geq 4), \qquad (g\geq 1, k\geq 2).
	\]
\end{remark}

\begin{remark}[The universal curve]
	The preimage of a point $[\Sigma,j,M,D] \in \dmspace_{g,k-1}$ via $ft_k$ consists of the Riemann surface $\Sigma$ with nodes identified, i.e.,
	\[
	ft_k^{-1} ([\Sigma,j,M,D]) \simeq \Sigma / \sim, \quad \text{where $x_a\sim y_a$ for nodal pairs $\{x_a,y_a\}\in D$}.
	\]
	According to such a description $\dmspace_{g,k}$ is sometimes called the \textbf{universal curve} over $\dmspace_{g,k-1}$.  We note that this is not a fiber bundle (the ``fibers'' are not constant and can vary locally), nor is it a fibration (the ``fibers'' are not homotopy equivalent).  However, considered as homology classes, the fibers are homologous.
\end{remark}

\subsubsection{Local expressions for the $k$th-marked point forgetting map on the DM-orbifolds.}
	\label{subsubsec:local-expression-forgetting-map-on-dm-orbifold}
We now write down local expressions for $ft_k$ in the coordinates given by the (alternative) good uniformizing families. We work first with regards to an arbitrary gluing profile on the DM-orbifolds to obtain general expressions.


Consider a stable Riemann surface $[\Sigma,j,M,D]\in \dmspace_{g,k}$.
We may write a local expression for $ft_k$ for each of the three cases as follows.

\noindent\textit{Case 1.}
By assumption $C_k \setminus \{z_k\}$ is stable.
Hence we can take an alternative good uniformizing family centered at a representative $(\Sigma,j,M,D)$: $(a,v,y)\mapsto (\Sigma_a, j(a,v), M_{(a,y)}, D_a )$.
Observe that forgetting the parametrization of the $k$th-marked point gives a good uniformizing family centered at $(\Sigma,j,M\setminus\{z_k\},D)$: $(a,v)\mapsto (\Sigma_a, j(a,v), (M\setminus\{z_k\})_a, D_a )$.

A local expression for $\hat{ft}_k$ in terms of these good uniformizing families is given by
	\begin{align*}
	\hat{ft}_k : \{(\Sigma_a,j(a,v),M_{(a,y)}, D_a)\}_{(a,v,y)\in U}
		& \to \{(\Sigma_b, j(b,w), (M\setminus\{z_k\})_b, D_b)\}_{(b,w)\in V} \\
		(a,v,y) & \mapsto (a,v).
	\end{align*}

\noindent\textit{Case 2a.}
For notational simplicity, let us assume that $\{x_a,y_a\}$ and $\{x_b,y_b\}$ are the only nodal pairs on the stable Riemann surface, i.e., we consider $[\Sigma, \allowbreak j, \allowbreak M, \allowbreak \{\{x_a,y_a\}, \allowbreak  \{x_b,y_b\}\}]$.
Hence after forgetting the $k$th-marked point there is only one nodal pair, which we denote as $\{x_c,y_c\}$.

There is a unique biholomorphism 
	\[C_k \setminus \{y_a,x_b\} \simeq \R\times S^1\]
which sends the marked point $z_k$ to the point $(0,0)$, the puncture $y_a$ to $-\infty$, and the puncture $x_b$ to $+\infty$.
We may choose the small disk structure at $y_a$ such that there is a biholomorphism between $D_{y_a}\setminus \{y_a\}$ and $\R^-\times S^1\subset \R\times S^1$.  Likewise, we choose the small disk structure at $x_b$ such that there is a biholomorphism between $D_{x_b}\setminus \{x_b\}$ and $\R^+\times S^1$.

We replace $D_{x_a}\sqcup C_k \sqcup D_{y_b}$ with the glued cylinder
	\[
	Z_{a,b} :=
	\begin{cases}
	[0,R_a+R_b]\times S^1 & a\neq 0, b\neq 0, \\
	\R^+\times S^1 \sqcup \R^-\times S^1 & a\neq 0,b=0 \text{ or } a=0, b\neq 0,\\
	\R^+\times S^1 \sqcup \R\times S^1 \sqcup \R^-\times S^1 & a=0, b=0.
	\end{cases}
	\]
We thus obtain the glued Riemann surface 
	\[\Sigma_{a,b}: = \Sigma \setminus (D_{x_a}\sqcup C_k \sqcup D_{y_b}) \sqcup Z_{a,b}.\]
The movement of the marked point $z_k$ is parametrized as follows:
\begin{itemize}
	\item For $a=b=0$, the marked point $z_k$ is given by $(0,0)$ on the component $\R \times S^1$.
	\item For $a\neq 0$ and $b=0$, the marked point $z_k$ is given by $(R_a,\theta_a)\in \R^+\times S^1$.  Analogously for $a=0$ and $b\neq 0$, the marked point is given by $(R_b,\theta_b)\in \R^-\times S^1$.
	\item For $a\neq 0$ and $b\neq 0$, the marked point $z_k$ is given by $(R_a,\theta_a)$ in the $(s^+,t^+)$ coordinates and $(-R_b,-\theta_b)$ in the $(s^-,t^-)$ coordinates.
\end{itemize}

Consider a good uniformizing family centered at $(\Sigma,j,M,\{\{x_a,y_a\}, \{x_b,y_b\}\})$:
	\[
	(a,b,v)\mapsto (\Sigma_{a,b},j(a,b,v),M_{a,b},\{\{x_a,y_a\}, \{x_b,y_b\} \}_{a,b} ), \qquad (a,b,v) \in \hb\times \hb \times V.
	\]
We may use the same good complex deformation after forgetting the $k$th-marked point, hence we may also consider a good uniformizing family centered at  $\hat{ft}_k (\Sigma,j,M,\{\{x_a,y_a\}, \{x_b,y_b\}\}) = (\Sigma\setminus C_k, j, M\setminus \{z_k\}, \{\{x_c,y_c\}\})$:
	\[
	(c,w)\mapsto ((\Sigma\setminus C_k)_c, j(c,w), (M\setminus \{z_k\})_c, \{\{x_c,y_c\}\}_c), \qquad (c,w) \in \hb \times V
	\]
Comparing the families of Riemann surfaces, we see
	\begin{align*}
	&\hat{ft}_k (\Sigma_{a,b},j(a,b,v),M_{a,b},\{\{x_a,y_a\}, \{x_b,y_b\}\}_{a,b}) \\
	&\qquad = ((\Sigma\setminus C_k)_c, j(c,w), (M\setminus \{z_k\})_c, \{x_c,y_c\}_c)
	\end{align*}
precisely when $w=v$ and when $c= a\ast_{\varphi} b$ where $a\ast_{\varphi} b$ is defined by:
	\[\label{local-expression-for-ftk}
	a\ast_{\varphi} b :=
	\begin{cases}
	\varphi^{-1}(\varphi(r_a)+\varphi(r_b)) e^{-2\pi i (\theta_a+\theta_b)} & \text{when } a\neq 0 \text{ and } b\neq 0, \\
	0 & \text{when } a=0 \text{ or } b=0.
	\end{cases}
	\]

A local expression for $\hat{ft}_k$ is given by:	
	\begin{align*}
	\hat{ft}_k :	&\{(\Sigma_{a,b},j(a,b,v),M_{a,b},\{\{x_a,y_a\}, \{x_b,y_b\}\}_{a,b})\}_{(a,b,v)} \\
		 			&\qquad\to \{((\Sigma\setminus C_k)_c, j(c,w), (M\setminus \{z_k\})_c, \{x_c,y_c\}_c)\}_{(c,w)} \\
					&(a,b,v)  \mapsto (a\ast_{\varphi} b, v).
	\end{align*}

\begin{figure}[ht]
	\centering
	\includegraphics[width=\textwidth]{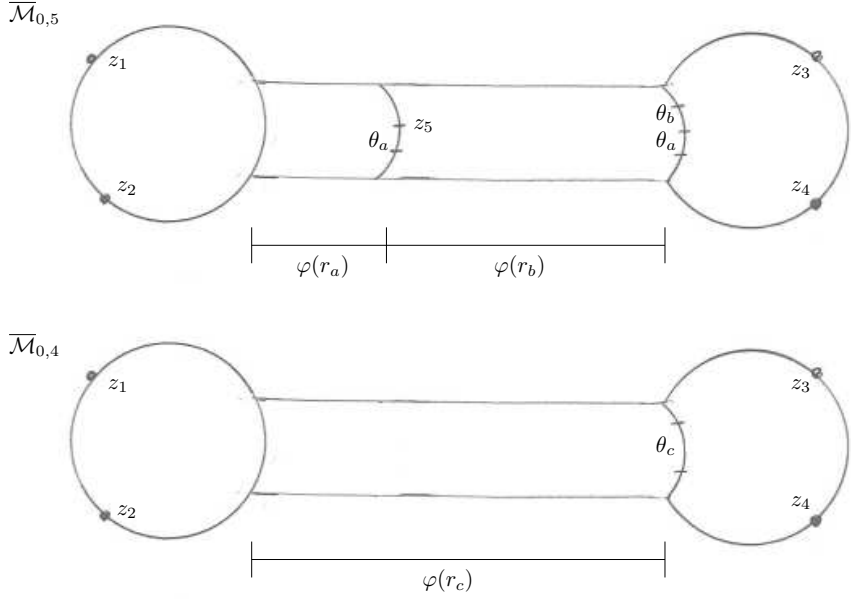}
	\caption{The $5$th-marked point forgetting map}
\end{figure}

\noindent\textit{Case 2b.}
Again for simplicity, let us assume that $\{x_a,y_a\}$ is the only nodal pair on $[\Sigma,j,M,D]$ and hence $ft_k([\Sigma,j,M,D])$ contains no nodal pairs.
Hence we consider assume $[\Sigma,j,M,\{\{x_a,y_a\}\}]$ which maps to $ft_k ([\Sigma,j,M,\{\{x_a,y_a\}\}]) = [\Sigma\setminus C_k, j, M\setminus \{z_k\}, \emptyset]$.

Again, there is a unique biholomorphism between $C_k \setminus \{y_a,z_i\}$ and $\R\times S^1$ which sends the marked point $z_k$ to the point $(0,0)$, the puncture $y_a$ to $-\infty$, and the puncture at the marked point $z_i$ to $+\infty$.
We may choose the small disk structure at $y_a$ such that there is a biholomorphism between $D_{y_a}\setminus \{y_a\}$ and $\R^-\times S^1\subset \R\times S^1$.
For $a\neq 0$ we may write the ``glued'' cylinder as:
	\[
	\Sigma_a = \Sigma \setminus (D_{x_a} \sqcup D_{y_a}) \sqcup Z_a \sqcup \R^+\times S^1 \sqcup \{+\infty\} \simeq \Sigma \setminus C_k.
	\]
The movement of the marked point $z_k$ is parametrized as follows:
\begin{itemize}
	\item For $a=0$, the marked point $z_k$ is given by $(0,0)$ on the component $\R \times S^1$, while the marked point $z_i$ is given by the puncture at $+\infty$ on $\R\times S^1$.
	\item For $a\neq 0$ the marked point $z_k$ is given by $(R_a,\theta_a)\in \R^+\times S^1$, while the marked point $z_i$ is again given by the puncture at $+\infty$ on $\R^+\times S^1$.
\end{itemize}

Consider a good uniformizing family centered at $(\Sigma,j,M,\{\{x_a,y_a\}\})$:
	\[
	(a,v)\mapsto (\Sigma_a,j(a,v),M_a,\{\{x_a,y_a\}\}_a), \qquad (a,v) \in \hb \times V
	\]
We may use the same good complex deformation after forgetting the $k$th-marked point, hence we may also consider a good uniformizing family centered at $\hat{ft}_k (\Sigma,j,M,\{\{x_a,y_a\}\}) = (\Sigma\setminus C_k, j, M\setminus \{z_k\}, \emptyset)$:
	\[
	w\mapsto (\Sigma\setminus C_k, j(w), M\setminus \{z_k\},\emptyset), \qquad w\in V
	\]

Comparing the families of Riemann surfaces, we see
	\[
	\hat{ft}_k (\Sigma_a,j(a,v),M_a,\{\{x_a,y_a\}\}_a) = (\Sigma\setminus C_k, j(w), M\setminus \{z_k\}, \emptyset)
	\]
precisely when $w=v$.

A local expression for $\hat{ft}_k$ is given by:	
	\begin{align*}
	\hat{ft}_k : \{(\Sigma_a,j(a,v),M_a,\{\{x_a,y_a\}\}_a)\}_{(a,v)} &\to \{(\Sigma\setminus C_k, j(w), M\setminus \{z_k\}, \emptyset)\}_{w} \\
	(a,v)  & \mapsto v.
	\end{align*}

\begin{proposition}
	\label{prop:ftk-log-dmspace}
	The $k$th-point forgetting map, considered on the Deligne--Mumford orbifolds constructed with the logarithmic gluing profile,
	$$
	ft_k: \dmlog_{g,k}\to \dmlog_{g,k-1},
	$$
	is a smooth map.  Moreover, it is holomorphic.
\end{proposition}
\begin{proof}
	We check the smoothness of local expressions of $ft_k$.  The only points in $\dmlog_{g,k}$ where the local expression of $\hat{ft}_k$ will not be trivially holomorphic is when $[\Sigma,j,M,D]$ belongs to case \ref{case-2a}, i.e., a representative contains a component $S^2$ with precisely $3$ special points, one of which is the $k$th-marked point while the other two are nodal points.
	The inverse of $\varphi_\text{log}(r)=-\tfrac{1}{2\pi}\log (r)$ is given by $\varphi^{-1}_\text{log}(R)= e^{-2\pi R}$.  Hence it follows that:
	\[
	a\ast_{\log} b = r_a\cdot r_b e^{-2\pi i (\theta_a+\theta_b)} =	a\cdot b,
	\]
	i.e., complex multiplication of the gluing parameters. Therefore the local expressions for $\hat{ft}_k$ are all holomorphic.
\end{proof}


For $i\in \{1,\ldots,k-1\}$, we may define a \textbf{canonical section}
	\[
	s_i : \dmspace_{g,k-1} \to \dmspace_{g,k}
	\]
on the DM-spaces by introducing a new component $S^2$ which contains the marked points $z_i$ and $z_k$.
This map is a section of $ft_k: \dmspace_{g,k} \to \dmspace_{g,k-1}$ in the sense that $ft_k \circ s_i = \id_{\dmspace_{g,k-1}}$.
For any choice of gluing profile, this map is a {smooth embedding} of $\dmspace_{g,k-1}$ into $\dmspace_{g,k}$.


\subsection{Inclusion and marked point identifying maps for the split/genus Gromov--Witten polyfolds}
	\label{subsec:inclusion-marked-point-identifying-maps}

Associated to a splitting $S$ of $(g,k)$ as defined in \S~\ref{subsec:other-gw-polyfolds} we may consider the natural \textbf{marked point identifying map}
	\[
	\phi_S : \dmspace_{k_0+1 , g_0}\times \dmspace_{k_1+1 , g_1} \to \dmspace_{g,k}
	\]
which identifies the last marked point of a stable noded Riemann surface in $\dmspace_{k_0+1 , g_0}$ with the first marked point of a stable noded Riemann surface in $\dmspace_{k_1+1, g_1}$, and which maps the first $k_0$ marked points of $\dmspace_{g_0,k_0+1}$ to $\{1,\ldots,k_0\}$ and likewise maps the last $k_1$ marked points of $\dmspace_{g_1,k_1+1}$ to $\{k_0+1,\ldots,k_0+k_1\}$.

Let $A_0,$ $A_1$ be such that $A_0+A_1=A$.
Together with the inclusion map, this map lifts to a $\ssc$-smooth map defined on the split GW-polyfolds:
	\[
	\begin{tikzcd}
	\cZ_{A_0,g_0,k_0+1} \times \cZ_{A_1,g_1,k_1+1}	&	\\
	\cZ_{A_0+A_1,S}	\arrow[u, hook, "i"] \arrow[r, "\phi_S"] & \cZ_{A,g,k}
	\end{tikzcd}
	\]

There is also a natural map 
	\[\psi: \dmspace_{g-1,k+2} \to \dmspace_{g,k}\]
which identifies the last two marked points of a stable noded Riemann surface increasing the arithmetic genus by one.
Note that this is a degree two map.
Given a point in the top stratum of $\dmspace_{g-1,k+2}$, by exchanging the marked points $z_{k+1}$ and $z_{k+2}$ we will in general obtain distinct points in $\dmspace_{g-1,k+2}$, but both of these points will map to the same point in $\dmspace_{g,k}$.
The exception is if both points $z_{k+1}$ and $z_{k+2}$ lie on a single component $S^2$ with one other special point; however in this case the image will be a singular point $\dmspace_{g,k}$ and should be counted with weight $\tfrac{1}{2}$.

Together with the inclusion map, this map also lifts to a $\ssc$-smooth map defined on the genus GW-polyfolds:
	\[
	\begin{tikzcd}
	\cZ_{A,g-1,k+2} 		&	\\
	\cZ^\text{g}_{A,g-1,k+2}	\arrow[u, hook, "i"] \arrow[r, "\psi"] & \cZ_{A,g,k}
	\end{tikzcd}
	\]


\section{The polyfold Gromov--Witten invariants}
	\label{sec:the-polyfold-gw-invariants}

In this section, we give a brief introduction to the polyfold GW-invariants as defined in \cite{HWZGW}.
In addition to the definition of these invariants via the branched integral of \cite{HWZint}, it is also necessary to give an equivalent definition of these invariants in terms of intersection numbers, as introduced in \cite{schmaltz2019steenrod}.

\subsection{Abstract perturbation of the Cauchy--Riemann section}
	\label{subsec:cauchy-riemann-section}

At the outset, fix a compatible almost complex structure $J$ on $Q$ so that $\ww(\cdot,J\cdot)$ is a Riemannian metric on $Q$.

\begin{definition}
	\label{def:strong-polyfold-bundle}
	The \textbf{underlying set of the strong polyfold bundle $\cW_{A,g,k}$} is defined as the set of equivalence classes
		\[
		\cW_{A,g,k}:= \{ (\Sigma,j,M,D,u,\xi ) \mid \cdots \} / \sim
		\]
	with data as follows.
	\begin{itemize}
		\item $(\Sigma,j,M,D,u)$ is a stable map representative of a stable curve in $\cZ_{A,g,k}$.
		\item $\xi$ is a continuous section along $u$ such that the map
			\[
			\xi(z):T_z\Sigma\rightarrow T_{u(z)}Q,\qquad \text{for } z\in
			\Sigma
			\]
		is a complex anti-linear map.
		\item As in Definition~\ref{def:class-3delta} we may write a local expression for $\xi$ of the form
			\[
			[s_0,\infty) \times S^1 \to \R^{2n}, \quad (s,t) \mapsto d\varphi_{u(\sigma(s,t))} \cdot \xi (\sigma(s,t)) \cdot \left(\tfrac{\partial}{\partial_s} \sigma(s,t)\right)
			\]
		which is defined for $s_0$ large. We require this local expression is of class $H^{2,\delta_0}$ around the nodal points in $\abs{D}$.
		We require a similar coordinate expression is of class $H^2_\text{loc}$ near the other points in $\Sigma$.
		(If $u$ has a puncture at the marked point $z_i$, then we also require that $\xi$ is of class $H^{2,\delta_0}$ at the marked point.)
		\item The equivalence relation is given by 
			\[(\Sigma,j,M,D,u,\xi)\sim (\Sigma',j',M',D',u',\xi')\]
		if there exists a biholomorphism $\phi :(\Sigma,j)\to (\Sigma',j')$ which satisfies $\xi'\circ d\phi=\xi$ in addition to $u'\circ \phi = u$, $\phi(M)= M'$, $\phi (\abs{D})=\abs{D'}$, and which preserves ordering and pairs.
	\end{itemize}
\end{definition}

By \cite[Thm.~1.10]{HWZGW}, the set $\cW_{A,g,k}$ possesses a polyfold structure such that
	\[
	P:\mathcal{W}_{A,g,k}\to \mathcal{Z}_{A,g,k}, \qquad [\Sigma,j,M,D,u,\xi]\mapsto [\Sigma,j,M,D,u]
	\]
defines a ``strong polyfold bundle'' over the GW-polyfold $\cZ_{A,g,k}$.

The \textbf{Cauchy--Riemann section} $\delbarj$ of the strong polyfold bundle $P:\cW_{A,g,k}\to \cZ_{A,g,k}$ is defined on the underlying sets by
	\[
	\delbarj([\Sigma,j,M,D,u]):=[\Sigma,j,M,D,u, \frac{1}{2}( du+J(u)\circ du\circ j)].
	\]
By \cite[Thm.~1.11]{HWZGW}, the  Cauchy--Riemann section is an $\ssc$-smooth Fredholm section with Fredholm index given by
	\[\ind \delbarj =2c_1(A)+(\dim_{\R} Q-6)(1-g)+2k.\]

We define the \textbf{unperturbed Gromov--Witten moduli space} as the zero set of this section,
	\[
	\CM_{A,g,k}(J) := \{[\Sigma,j,M,D,u] \mid \delbarj([\Sigma,j,M,D,u])  =0 \} \subset \cZ_{A,g,k}.
	\]
This set is precisely the same as the stable map compactification of a GW-moduli space as we discussed in the introduction.
Moreover, one may show that the subspace topology on $\CM_{A,g,k}(J)$ is equivalent to the Gromov topology defined on the stable map compactification.

There exist ``regular perturbations'' $p$ which ``regularize'' the unperturbed GW-moduli spaces; thus we obtain a \textbf{perturbed Gromov--Witten moduli space}
	\[
	\cS_{A,g,k} (p) := \text{`` } (\delbarj +p)^{-1}(0) \text{ ''} \subset \cZ_{A,g,k}
	\]
which has the structure of a compact oriented \textit{weighted branched suborbifold}. 
For the precise definition of a ``regular perturbation,'' we refer to \cite[Cor.~15.1]{HWZbook}.
We recall some important properties of such a suborbifold.

\begin{definition}
	\label{def:weighed-branched-suborbifold}
	Let $\Q^+:= \Q \cap [0,\infty)$.
	A \textbf{weighted branched suborbifold} $\cS$ of the GW-polyfold $\cZ_{A,g,k}$ is a subset $\cS \subset \cZ_{A,g,k}$ such that:
	\begin{itemize}
		\item $\cS$ consists entirely of \textit{smooth} stable curves,
		\item $\cS$ comes equipped with a rational valued \textbf{weight function} $\theta : \cS \to \Q^+.$
	\end{itemize}
	Given a stable curve $[\Sigma,j,M,D,u] \in \cS$, let 
	\[
	(a,v,\eta) \to (\Sigma_a,j(a,v),M_a,D_a,\oplus_a \exp_u \eta), \quad (a,v,\eta) \in \cU
	\]
	be a local uniformizer for $\cZ_{A,g,k}$ centered at a representative $(\Sigma,j,M,D,u)$.
	Then $\cS$ may be locally described as follows:
	\begin{itemize}
		\item \textbf{Local branches.} There exists a finite collection $M_i$, $i\in I$ of finite-dimensional manifolds $(M_i)_{i\in I}$ and proper embeddings 
		\[
		M_i \hookrightarrow \{(\Sigma_a,j(a,v),M_a,D_a,\oplus_a \exp_u \eta)\}_{(a,v,\eta) \in \cU}
		\]
		such that the union $\cup_{i\in I} M_i$ is $\Aut(\Sigma,j,M,D,u)$-invariant, and moreover 
		\[
		\frac{\bigcup_{i \in I} M_i}{\Aut(\Sigma,j,M,D,u)} \to \cS
		\]
		is a local homeomorphism.
		\item \textbf{Weights.} Associated to each $M_i$ there exists a positive rational number $w_i$ such that the function 
			\[
			\hat{\theta} : \bigcup_{i \in I} M_i \to \Q^+, \qquad	x \mapsto \sum_{ \{i\in I \mid x\in M_i\} } w_i
			\]
		is an $\Aut(\Sigma,j,M,D,u)$-invariant lift of the weight function $\theta:\cS \to \Q^+$.
	\end{itemize}
	The data ${(M_i)}_{i\in I}$, ${(w_i)}_{i\in I}$ is called a \textbf{local branching structure}.
	We say that $\cS$ is \textbf{compact} if the underlying topological space equipped with the subspace topology is compact.
	We may define an \textbf{orientation} on a weighted branched suborbifold as an $\Aut$-invariant orientation of each local branch, denoted as $(M_i,o_i)$.
\end{definition}
We note that this is only a partial description of a weighted branched suborbifold but it is sufficient for our purposes, see {\cite[Def.~9.1]{HWZbook}} for a full definition in the context of ep-groupoid theory.


\subsection{Polyfold Gromov--Witten invariants as branched integrals}
	\label{subsec:branched-integrals}

We briefly consider the branched integration theory, as originally developed in \cite{HWZint}.
Given a smooth map from a polyfold to a manifold/orbifold,
	\[f:\cZ \to \cO,\]
we may pullback a differential form $\ww \in \Omega^n(\cO)$ and obtain a ``$\ssc$-differential form'' $f^*\ww \in \Omega^n_\infty (Z)$.
We can consider the restriction of a $\ssc$-differential form to a weighted branched suborbifold; locally this restriction is the same as the normal pullback on the local branches.

\begin{theorem}[{\cite[Thm.~9.2]{HWZbook}}]
	\label{def:branched-integral}
	Let $\cZ$ be a polyfold.
	Consider a $\ssc$-smooth differential form $\ww\in \Omega^n_\infty (Z)$ and an $n$-dimensional compact oriented weighted branched suborbifold $\cS\subset \cZ$.
	Then there exists a well-defined \textbf{branched integral}, denoted as $\int_{\cS} \ww.$
\end{theorem}

In our current situation, we may note that the GW-polyfolds admit $\ssc$-smooth partitions of unity,  see the discussion in \cite[\S~5.5]{HWZbook}.
Furthermore, we may note that the automorphism groups are all effective; this may be seen by considering the explicit action of the automorphism group, for any automorphism we may find a stable map which breaks the symmetry.
Therefore, we may compute the the branched integral as described in \cite[Rem.~3.6]{schmaltz2019steenrod} as follows.

Cover the compact topological space $\cS$ with finitely many open sets of the form $\cup_{i\in I} M_i / \Aut$.
As in the remark, we can define a $\ssc$-smooth partition of unity on $\cZ$ so that its restriction to $\cS$ is also a partition of unity with respect to the open cover $\cup_{i\in I} M_i / \Aut$. We can then write:
	\[
	\int_{\cS} \ww = \sum_n \int_{\cS} \beta_n \cdot \ww = \sum_n \frac{1}{\sharp \Aut(\Sigma,j,M,D,u) } \sum_{i\in I} w_i \int_{(M_i,o_i)} \beta_n \cdot \ww
	\]
where the first sum is over the open cover indexed by $n$ and $\beta_n$ is the associated partition of unity.

The polyfold GW-invariant is defined as the homomorphism
	\[
	\GW_{A,g,k} : H_*(Q;\Q)^{\otimes k} \otimes H_*(\dmlog_{g,k};\Q ) \to \Q
	\]
defined via the branched integration of \cite{HWZint}:
	\[
	\GW_{A,g,k} (\alpha_1,\ldots,\alpha_k;\beta) := \int_{\cS_{A,g,k}(p)} ev_1^* \PD (\alpha_1)\wedge \cdots \wedge ev_k^* \PD (\alpha_k) \wedge\pi^* \PD (\beta).
	\]
It was justified in \cite[Rem.~3.15]{schmaltz2019steenrod} that this homomorphism is rationally valued, using equivalence of the branched integral with the intersection number.

\begin{theorem}[Change of variables, {\cite[Thm.~2.47]{schmaltz2019naturality}}]
	\label{thm:change-of-variables}
	Let $\cS_i\subset \cZ_i$ be $n$-dimensional compact oriented weighted branched suborbifolds with weight functions $\theta_i:\cS_i \to \Q^+$ for $i=1,2$.	
	Let $g:\cZ_1 \to \cZ_2$ be a $\ssc^1$-map between polyfolds, which has a well-defined restriction $g|_{\cS_1}	: \cS_1 \to \cS_2$ between the branched suborbifolds. In addition, assume the following:
	\begin{itemize}
		\item $g|_{\cS_1}: \cS_1 \to \cS_2$ is a homeomorphism between the underlying topological spaces and is \textit{weight preserving}, i.e., $\theta_2\circ g=\theta_1$,
		\item A lift to the local branching structures $\hat{g}: \bigcup_{i \in I} M_i \to \bigcup_{j \in I'} M_j'$ is a local homeomorphism, and moreover maps branches to branches and is orientation and weight preserving on each branch.
\end{itemize}
	Then given a $\ssc$-smooth differential form $\ww \in \Omega^n_\infty (Z_2)$,
	\[
	\int_{\cS_2} \ww = \int_{\cS_1} g^* \ww.
	\]
\end{theorem}

\subsection{The Steenrod problem and polyfold Gromov--Witten invariants as intersection numbers}
	\label{subsec:intersection-numbers}

We now describe the definition of the GW-invariants as intersection numbers, as developed in \cite{schmaltz2019steenrod}.

The polyfold GW-invariants take as input homological data coming from a closed orientable orbifold of the form $Q^k \times \dmlog_{g,k}$.
In order to define the polyfold GW-invariants as intersection numbers, we need to understand how to take this homological data and interpret it in a suitable way.

This is done via the Steenrod problem. As proved by Thom in \cite[Thm.~II.1]{thom1954quelques}, there exists a basis of $H_*(Q;\Q)$ which consists of the fundamental classes of closed embedded submanifolds $\cX\subset Q$.
We call such a submanifold a \textbf{representing submanifold}.
Since $\dmlog_{g,k}$ is an orbifold for $g\neq 0$, we need the following theorem.

\begin{theorem}[The Steenrod problem for orbifolds, {\cite[Thm.~1.2]{schmaltz2019steenrod}}]
	Let $\cO$ be a closed orientable orbifold.
	There exists a basis $\{[\cX_i]\}$ of $H_*(\cO;\Q)$ which consists of the fundamental classes of ``closed embedded full suborbifolds $\cX_i\subset \cO$ whose normal bundles have fiberwise trivial isotropy action'' (see \cite[Defs.~2.12,~2.14]{schmaltz2019steenrod}).
	We will call such a suborbifold a \textbf{representing suborbifold}.
\end{theorem}

Consider the evaluation and projection maps defined on a GW-polyfold:
	\[
	\cZ_{A,g,k} \xrightarrow{ev_1\times\cdots\times ev_k \times \pi} Q^k \times \dmlog_{g,k}.
	\]
Let $\cS_{A,g,k}(p)$ be a perturbed GW-moduli space, and let $\cX_1\times \cdots \times \cX_k \times \cB \subset Q^k\times \dmlog_{g,k}$ be a representing suborbifold. Consider the diagram:
	\[
	\begin{tikzcd}[column sep= huge, row sep=tiny]
	\cS_{A,g,k}(p) \arrow[r,"ev_1\times \cdots \times ev_k \times \pi"] & Q^k \times \dmlog_{g,k} \arrow[d,phantom,"\cup\quad"] \\
	& \cX_1\times \cdots \times \cX_k \times \cB.
	\end{tikzcd}
	\]
Then there exists a well-defined notion of \textbf{transversal intersection} \cite[Def.~3.8]{schmaltz2019steenrod}
	\[(ev_1\times \cdots \times ev_k \times \pi)|_{\cS_{A,g,k}(p)} \pitchfork (\cX_1\times \cdots \times \cX_k \times \cB)\]
and moreover a well-defined \textbf{intersection number}  \cite[Def.~3.13]{schmaltz2019steenrod}
	\[
	\left(ev_1\times\cdots\times ev_k\times\pi\right)|_{\cS_{A,g,k}(p)} \cdot \left(\cX_1 \times\cdots\times \cX_k \times \cB\right).
	\]
When $\dim \cS_{A,g,k}(p) + \dim \left(\cX_1 \times\cdots\times \cX_k \times \cB\right) = \dim (Q^k \times \dmlog_{g,k})$ the intersection number is given by the signed weighted count of a finite number of points of intersection.

We may then define the polyfold GW-invariants as the intersection number
	\[
	\GW_{A,g,k} ([\cX_1],\ldots,[\cX_k];[\cB]) := \left(ev_1\times\cdots\times ev_k\times\pi\right)|_{\cS_{A,g,k}(p)} \cdot \left(\cX_1 \times\cdots\times \cX_k \times \cB\right).
	\]

\begin{figure}[ht]
	\centering
	\includegraphics[width=\textwidth]{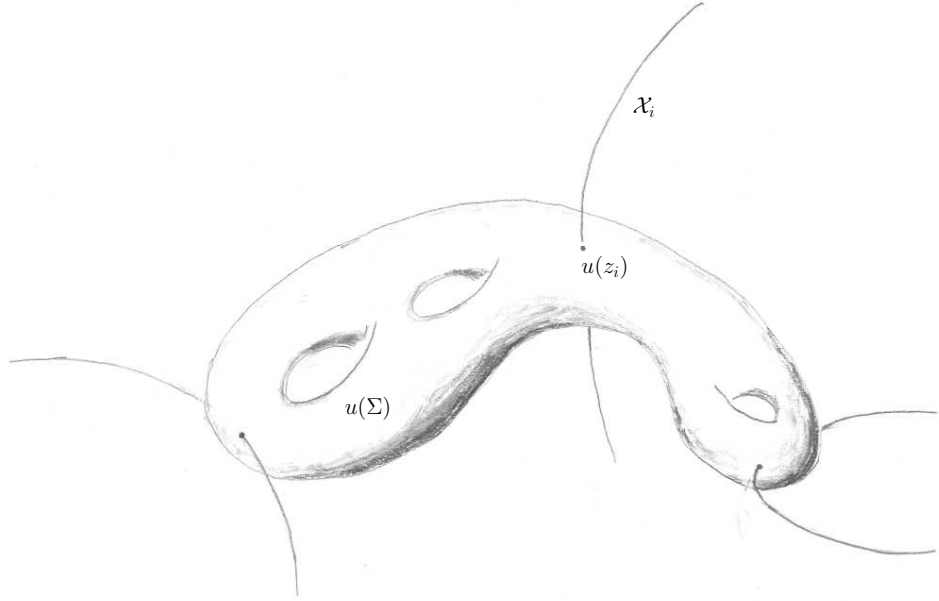}
	\caption{A Gromov--Witten invariant as an intersection number}
\end{figure}

By \cite[Thm.~1.6, Cor.~1.7]{schmaltz2019steenrod} the branched integral and the intersection number are related by the following equation:
	\begin{align*}
	&\int_{\cS_{A,g,k}(p)} ev_1^* \PD [\cX_1] \wedge \cdots \wedge ev_k^* \PD[\cX_k] \wedge\pi^* \PD [\cB] \\
	&\qquad = \left(ev_1\times\cdots\times ev_k\times\pi\right)|_{\cS_{A,g,k}(p)} \cdot \left(\cX_1 \times\cdots\times \cX_k \times \cB\right).
	\end{align*}

Transversality of a perturbed solution space of a polyfold with representing submanifolds/suborbifolds may \textit{always} be achieved through either of the following:
\begin{itemize}
	\item Through the perturbation of the representing suborbifold; due to the properties of the normal bundle representing suborbifolds may always be perturbed (see \cite[Prop.~3.9]{schmaltz2019steenrod}).
	\item Assuming the map defined on the ambient polyfold is a submersion, we may obtain transversality through choice of a suitable regular perturbation (see \cite[Prop.~3.10]{schmaltz2019steenrod}).
\end{itemize}

We showed in Proposition~\ref{prop:evaluation-smooth-submersion} that the evaluation map is a smooth submersion, hence we may always choose a perturbation such that $ev_i \pitchfork \cX_i$.
This is important in the contexts of the splitting and genus reduction axioms where we consider a representing submanifold given by the diagonal $\Delta \subset Q\times Q$; in these cases, it is important that we do not perturb this representing submanifold or else we will lose the geometric meaning of the intersection.

However, the projection map 
	\[
	\pi: \cZ_{A,g,k} \to \dmlog_{g,k}
	\]
is not a submersion (unless $(g,k)=(0,3)$) as we explain in Problem~\ref{prob:projection-not-submersion}.
Hence in this case we must perturb the representing suborbifold $\cB \subset \dmlog_{g,k}$ in order to obtain transversality.


\subsection{Pulling back abstract perturbations and maps between perturbed moduli spaces}
\label{subsec:pulling-back-abstract-perturbations}

The natural approach for obtaining a well-defined map between perturbed GW-moduli spaces is to pullback an abstract perturbation. The technical details for this approach are contained in \cite{schmaltz2019naturality}.

The general setup for this problem may be given as follows.
Consider a $\ssc$-smooth map between polyfolds, $f: \cY \to \cZ$, and consider a pullback diagram of strong polyfold bundles:
\[\begin{tikzcd}
f^* \cW \arrow[d, "f^*\delbarj\quad"'] \arrow[r, "\text{proj}_2"'] & \cW \arrow[d, "\quad \delbarj"] &  \\
\cY \arrow[r, "f"'] \arrow[u, bend left] & \cZ. \arrow[u, bend right] & 
\end{tikzcd}\]
It is possible to obtain transversality for both strong bundles through the choice of an appropriately generic perturbation.
The main technical point therefore is ensuring that we can control the compactness of the pullback perturbation.
This is achieved by a mild topological hypothesis on the map $f$.

We say that $f$ satisfies the \textbf{topological pullback condition} if for all $[y] \in \cS(\delbarj)\subset \cZ$ and for any open neighborhood $\cV \subset \cY$ of the fiber $f^{-1} ([y])$ there exists an open neighborhood $\cU_{[y]}\subset \cZ$ of $[y]$ such that $f^{-1} (\cU_{[y]}) \subset \cV$.
(Note that if $f^{-1} ([y])=\emptyset$, this implies that there exists an open neighborhood $\cU_{[y]}$ of $[y]$ such that $f^{-1} (\cU_{[y]})=\emptyset$.)

\begin{theorem}[{\cite[Thm.~1.7]{schmaltz2019naturality}}]
	\label{thm:pullback-regular-perturbation}
	If $f$ satisfies the topological pullback condition then there exists a regular perturbation $p$ which pulls back to a regular perturbation $f^*p$.
	
	It follows that we can consider a well-defined restriction between perturbed moduli spaces,
	\[f|_{\cS(f^*p)} : \cS(f^*p) \to \cS (p).\]
	This map is weight preserving.
\end{theorem}

By checking that the strong bundles are related by pullback, and by checking the topological pullback condition holds, we may obtain well-defined maps between perturbed GW-moduli spaces.
In particular, by pulling back via the permutation map between GW-polyfolds, $\sigma: \cZ_{A,g,k} \to \cZ_{A,g,k}$, we obtain:
\begin{itemize}
	\item permutation maps,
	\[
	\sigma : \cS_{A,g,k}(\sigma^* p ) \to \cS_{A,g,k}(p).
	\]
\end{itemize}

By pulling back via the inclusion maps from \S~\ref{subsec:inclusion-marked-point-identifying-maps}, we obtain:
\begin{itemize}
	\item inclusion maps for the split perturbed GW-moduli spaces,
	\[
	i: \cS_{A_0+A_1,S}(i^*p) \hookrightarrow (\cS_{A_0,g_0,k_0+1} \times \cS_{A_1,g_1,k_1+1})(p),
	\]
	\item inclusion maps for the genus perturbed GW-moduli spaces,
	\[
	i:\cS^\text{g}_{A,g-1,k+2}(i^*p) \hookrightarrow \cZ_{A,g-1,k+2} (p).
	\]
\end{itemize}

Observe that since the unperturbed GW-moduli space $\CM_{A,g,k}$ is compact, there are only finitely many decompositions $A_0+A_1$ for which the associated split GW-polyfolds can contain a nonempty unperturbed GW-moduli space.
We may then pullback via 
	\[\sqcup \phi_S: \displaystyle{\bigsqcup_{A_0+A_1 = A}} \cZ_{A_0+A_1,S} \to \cZ_{A,g,k}\]
where the (finite) disjoint union is only considered for split GW-polyfolds which contain a nonempty unperturbed GW-moduli space.
We thus obtain
\begin{itemize}
	\item marked point identifying maps for a disjoint union of split perturbed GW-moduli spaces:
	\[
	\sqcup \phi_S: \displaystyle{\bigsqcup_{A_0+A_1 = A}} \cS_{A_0+A_1,S} (\phi_S^*p )	\to \cS_{A,g,k}(p).
	\]
\end{itemize}

Finally, we also obtain:
\begin{itemize}
	\item marked point identifying maps for the genus perturbed GW-moduli spaces,
	\[
	\psi: \cS^\text{g}_{A,g-1,k+2}(\psi^* p)  \to \cS_{A,g,k} (p).
	\]
\end{itemize}

\begin{theorem}
	We can give a perturbed Gromov--Witten moduli space the structure of a stratified space, whose codimension $2m$ strata are given by the $m$-noded stable curves.
\end{theorem}
\begin{proof}
	In the literature, we may keep track of possible degenerations of an unperturbed GW-moduli space $\CM_{A,g,k}(J)$ by means of a tree $T$ with labelings $(A_\alpha, g_\alpha, k_\alpha)$ on each vertex $\alpha$ such that 	
	\[A = \sum_{\alpha} A_\alpha, \quad g= \sum_{\alpha} g_\alpha, \quad k = \sum_{\alpha} k_\alpha.\]
	Since $\CM_{A,g,k}(J)$ is compact, the set of such trees which model a possible degeneration is finite.
	We can define a GW-polyfold modeled on a tree as in \S~\ref{subsec:other-gw-polyfolds}; the theorem is then obtained by pulling back an abstract perturbation via the map
	\[
	\displaystyle{\bigsqcup_T} \cZ_T \to \cZ_{A,g,k}.
	\]
\end{proof}


\section{Problems arise}
	\label{sec:problems-arise}

At this point, there are substantial obstructions in polyfold GW-theory which we must understand before we have any hope to prove the GW-axioms.

\subsection{Problems}

\begin{problem}
	\label{prob:ftk-continuous-not-c0}
	In general the map $ft_k: \dmexp_{g,k} \to \dmexp_{g,k-1}$ is continuous but not differentiable.
\end{problem}

The construction of the GW-polyfolds requires the modification of the gluing profile used to define the DM-orbifolds, giving rise to exponential DM-orbifolds.
As a consequence of this nonstandard smooth structure the map $ft_k: \dmexp_{g,k} \to \dmexp_{g,k-1}$ is in general continuous but not differentiable. 
Differentiability fails at points which are nodal Riemann surfaces which contain components
$S^2$ with precisely $3$ special points, two of which are nodal, and one of which is the
$k$th-marked point.
In \S~\ref{subsec:differentiability-failure-ftk} we show this failure via explicit calculation.

\begin{problem}
	\label{prob:no-natural-forgetting-map}
	In general there does not exist a natural map which forgets the $k$th-marked point on the GW-polyfolds.
\end{problem}

The GW-stability condition \eqref{eq:gw-stability-condition} imposed on stable curves in the polyfold $\cZ_{A,g,k}$ may not hold in $\cZ_{A,g,k-1}$ once the $k$th-marked point is removed.
Recall once again the GW-stability condition:
\begin{itemize}
	\item For each connected component $C\subset \Sigma$ we require at least one of the following:
	\begin{equation*}
	2g_C+\sharp (M\cup\abs{D})_C \geq 3 \quad \text{or} \quad \int_C u^* \ww > 0.
	\end{equation*}
\end{itemize}
A stable curve in $\cZ_{A,g,k}$ may contain a ``destabilizing ghost component,'' i.e., a component $C_k\simeq S^2$ with precisely $3$ special points, one of which is the $k$th-marked point, and such that $\int_{C_k} u^*\ww =0, \ u|_{C_k} \neq \text{const}$.
In this case, after removal of the $k$th-marked point the GW-stability condition is no longer satisfied, and we cannot consider the resulting data as a stable curve in $\cZ_{A,g,k-1}$.
Hence even on the level of the underlying sets of the GW-polyfolds there does not exist a natural $k$th-marked point forgetting map.
In \S~\ref{subsec:classifying-destabilizing-ghost-components} we classify possible destabilizing ghost components, and also consider the fringe situations where it is possible to obtain a (trivial) well-defined $k$th-marked point forgetting map.

\begin{figure}[ht]
	\centering
	\includegraphics[width=\textwidth]{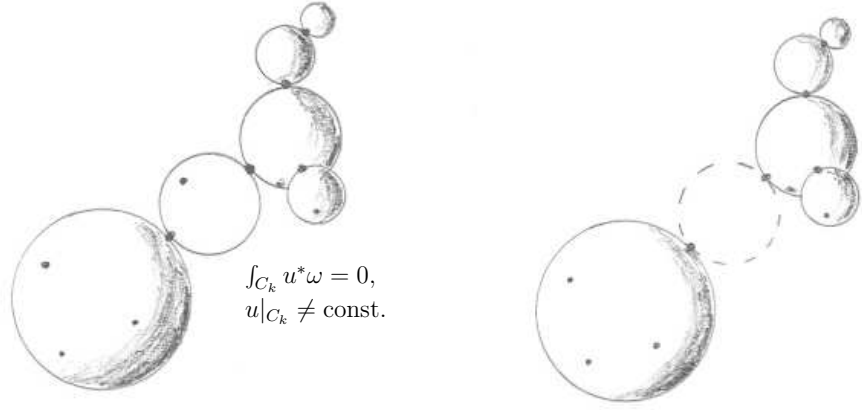}
	\caption{There does not exist a natural map which forgets the $k$th-marked point.}
\end{figure}

\begin{problem}
	\label{prob:restriction-is-not-continuous}
	In general the well-defined restriction $ft_k:\cZ_{A,g,k}^\text{const} \to \cZ_{A,g,k-1}$ is not continuous.
\end{problem}

We can attempt to restrict to a subset $\cZ^\text{const}_{A,g,k}\subset \cZ_{A,g,k}$ with a stronger stability condition:
\begin{itemize}
	\item For each connected component $C\subset \Sigma$ we require at least one of the following:
	\[
	2 g_C+\sharp (M\cup\abs{D})_C \geq 3 \quad \text{or} \quad \int_C u^* \ww > 0.
	\]
	Additionally, if the $k$th-marked point $z_k$ lies on a component $C_k$ with 
	\[
	2 g_{C_k} +\sharp (M\cup \abs{D})_{C_k} = 3 \quad \text{and}\quad \int_{C_k} u^*\ww =0
	\]
	then we require that $u|_{C_k}$ is constant, hence necessarily $u|_{C_k}\equiv u(z_k)$.
\end{itemize}

By the energy identity, any map $u: C_k \to Q$ with $\delbarj u = 0$ and $\int_{C_k} u^* \ww =0$ must be constant.  Hence it follows that the above subset is large enough to contain the entire unperturbed Gromov--Witten solution space, i.e., $\CM_{A,g,k}(J) \subset \cZ^\text{const}_{A,g,k}$.

With this stability condition the $k$th-marked point forgetting map is well-defined on $\cZ^\text{const}_{A,g,k}$, considered as a set.
Consider the subspace topology on $\cZ^\text{const}_{A,g,k}\subset \cZ_{A,g,k}$, and the usual polyfold topology on $\cZ_{A,g,k-1}$.
In general, with respect to these topologies the $k$th-marked point forgetting map is not continuous.
In \S~\ref{subsec:failure-of-continuity-ftk} we demonstrate lack of continuity by exhibiting a sequence which converges in the subspace topology on $\cZ^\text{const}_{A,g,k}$ but for which the image of the sequence does not converge.

\begin{problem}
	\label{prob:projection-not-submersion}
	While the projection $\pi : \cZ_{A,g,k} \to \dmlog_{g,k}$ is $\ssc$-smooth, in general it is not a submersion.
\end{problem}

We observed in Proposition~\ref{prop:projection-map-smooth} that the projection $\pi$ locally factors through the (smooth) identity map $\id :\dmexp_{g,k} \to \dmlog_{g,k}$.
However, in general this identity map is not a submersion.
The only exception is when $(g,k)=(0,3)$ as in this case $\dmspace_{0,3} = \{\pt\}$.
In the proof of Proposition~\ref{prop:identity-exp-to-log} we saw that the gluing parameters transform according to the rectangular coordinate expression
	\[
	F:\R^2 \to \R^2, \qquad (x,y) \mapsto \frac{e^{-2\pi \left(e^{1/\sqrt{x^2+y^2}}-e\right)}}{\sqrt{x^2+y^2}} (x,y).
	\]
It is elementary to compute partial derivatives of $F$ and obtain:
	\begin{align*}
	\frac{\partial F_1}{\partial x} (x,0) &= \frac{-2\pi e^{\left(-2\pi (e^{1/\abs{x}}-e)  + 1/\abs{x} \right)} }{x^2},	\\
	\frac{\partial F_1}{\partial y} (0,y) &= \frac{\partial F_2}{\partial x} (x,0) = 0,	\\
	\frac{\partial F_2}{\partial y} (0,y) &= \frac{-2\pi e^{\left(-2\pi (e^{1/\abs{y}}-e)  + 1/\abs{y} \right)} }{y^2}.
	\end{align*}
Taking limits, one may observe that the Jacobian of this function at $(0,0)$ is therefore
	\[
	\bm{J}(0,0) =
	\begin{bmatrix}
	0 & 0 \\
	0 & 0
	\end{bmatrix}.
	\]

\subsection{Differentiability of the \texorpdfstring{$k$th-marked}{kth-marked} point forgetting map on the exponential Deligne--Mumford orbifolds}
	\label{subsec:differentiability-failure-ftk}

We begin by remarking that the topology on the DM-space is independent of the choice of gluing profile, hence it automatically follows from Proposition~\ref{prop:ftk-log-dmspace} that $ft_k: \dmexp_{g,k}\to \dmexp_{g,k-1}$ is, at the very least, continuous.

However in general, $ft_k :\dmexp_{g,k} \to \dmexp_{g,k-1}$ is not $C^1$.
Differentiability fails at points in $\dmexp_{g,k}$ which contain precisely $3$ special points, one of which is the $k$th-marked point while the other two are nodal points.
In particular, differentiability fails in case 2a from \S~\ref{subsubsec:local-expression-forgetting-map-on-dm-orbifold}, and as noted in Remark~\ref{rmk:special-cases-forgetting} such points will occur when $(g=0,k\geq 5),$ $(g=1,k=1),$ or $(g\geq 2, k\geq 1)$.
We previously derived coordinate expressions for $ft_k$ for an arbitrary gluing profile; in case 2a such an expression has the following simplified form:
	\[
	\hat{ft}_k: (a,b,v) \mapsto (a\ast_{\exp} b, v).
	\]
Hence consider the function
	\[
	\C \times \C \to \C, \qquad (a,b) \mapsto a\ast_{\exp} b
	\]
where once again
	\[
	a\ast_{\exp} b :=
	\begin{cases}
	\varphi_\text{exp}^{-1}(\varphi_\text{exp}(r_a)+\varphi_\text{exp}(r_b)) e^{-2\pi i (\theta_a+\theta_b)} & \text{when } a\neq 0 \text{ and } b\neq 0, \\
	0 & \text{when } a=0 \text{ or } b=0.
	\end{cases}
	\]	
The inverse of $\varphi_\text{exp}(r)=e^{1/r}-e$ is given by $\varphi^{-1}_\text{exp}(R)= \tfrac{1}{\log(R+e)}$, and so if $a\neq 0$ and $b\neq 0$ we have
	\[
	a\ast_{\exp} b = \tfrac{1}{\log \left(e^{1/r_a}+e^{1/r_b}-e\right)}e^{-2\pi i(\theta_a+\theta_b)}.
	\]
This expression is not $C^1$.	
To see this, we rewrite the equation in rectangular coordinates as a function $F:\R^4\to \R^2$, 
	\begin{multline*}
	F(x_1,x_2,x_3,y_4):= \\ 
	\begin{cases}
		(0,0) 	& \text{if } x_1=x_2=0 \text{ or } x_3=x_4=0,\\
		\frac{1}{\log \left(e^{1/r_a}+e^{1/r_b}-e\right)}(\cos (\vartheta_a+\vartheta_b),\sin (\vartheta_a+\vartheta_b)) &\text{else},
	\end{cases}
	\end{multline*}
where $r_a^2 := x_1^2+x_2^2, r_b^2:=x_3^2+x_4^2$ and $\vartheta_a:=-2\pi\tan^{-1}(\tfrac{x_2}{x_1}), \vartheta_b:=-2\pi \tan^{-1}(\tfrac{x_4}{x_3})$.
We now compute the Jacobian matrix $\bm{J}$ of partial derivatives of $F$ at $(0,0,0,0)\in \R^4$.  From the above expression for $F$ we see that $\tfrac{\partial F_i}{\partial x_j} = 0$ for all $i=1,2$ and $j=1,2,3,4$, hence $\bm{J}_{i,j} = 0$.
	
If $F$ were differentiable, the Jacobian matrix would give the total derivative, and we could compute the directional derivative at $(0,0,0,0)$ via the equation 
		\[\nabla_v F =\bm{J}\cdot v = (0,0).\]
We may directly compute the directional derivative as follows.
Let $v\in \R^4$ be a unit vector, we may write
	\[v = r_1 \cos \theta_1 \tfrac{\partial}{\partial x_1} + r_1 \sin \theta_1 \tfrac{\partial}{\partial x_2} + r_2 \cos \theta_2 \tfrac{\partial}{\partial x_3} + r_2 \cos \theta_2 \tfrac{\partial}{\partial x_4}\]
where $r_1^2 +r_2^2 = 1$.
Then one may calculate:
	\begin{align*}
	\nabla_v F
	&=	\lim_{h\to 0} \frac{F(0+hv)}{h} \\
	&=	\lim_{h\to 0} \frac{1}{h \log \left(e^{1/(r_1 h)}+e^{1/(r_2 h)}-e\right)} (\cos (\theta_1+\theta_2), \sin(\theta_1+\theta_2)) \\
	&= \min \{r_1,r_2\} (\cos (\theta_1+\theta_2), \sin(\theta_1+\theta_2)).
	\end{align*}
This contradicts the assumption that $F$ was differentiable.

\subsection{Classifying destabilizing ghost components and fringe definitions of the \texorpdfstring{$k$th-marked}{kth-marked} point forgetting map on the Gromov--Witten polyfolds}
	\label{subsec:classifying-destabilizing-ghost-components}

Consider a stable curve $[\Sigma,j,M,D,u]$ with $k$ marked points, and let $(\Sigma,j,M,D,u)$ be a stable map representative.  We say that this stable curve/stable map contains a \textbf{destabilizing ghost component} if the connected component $C_k\subset \Sigma$ with $z_k \in C_k$ satisfies
	\[
	2 g_{C_k} +\sharp (M\cup \abs{D})_{C_k} = 3 \quad \text{and}\quad \int_{C_k} u^*\ww =0.
	\]
We may classify destabilizing ghost components as follows:
	\begin{itemize}
		\item Type $(0,1)$, i.e., $(g_{C_k}, M_{C_k}) = (0,1)$. There are two nodal points on $C_k$.
		\item Type $(0,2)$, i.e., $(g_{C_k}, M_{C_k}) = (0,2)$. There is one nodal point on $C_k$.
		\item Type $(0,3)$, i.e., $(g_{C_k}, M_{C_k}) = (0,3)$. Then $C_k$ contains no nodal points; it follows that $C_k$ is the only component and hence $[A]=0$. This situation can only arise when $(A,g,k)=(0,0,3)$.
		\item Type $(1,1)$, i.e., $(g_{C_k}, M_{C_k}) = (1,1)$. Then $C_k$ contains no nodal points; it follows that $C_k$ is the only component and hence $[A]=0$. This situation can only arise when $(A,g,k)=(0,1,1)$.
\end{itemize}

\begin{figure}[ht]
	\centering
	\includegraphics[width=\textwidth]{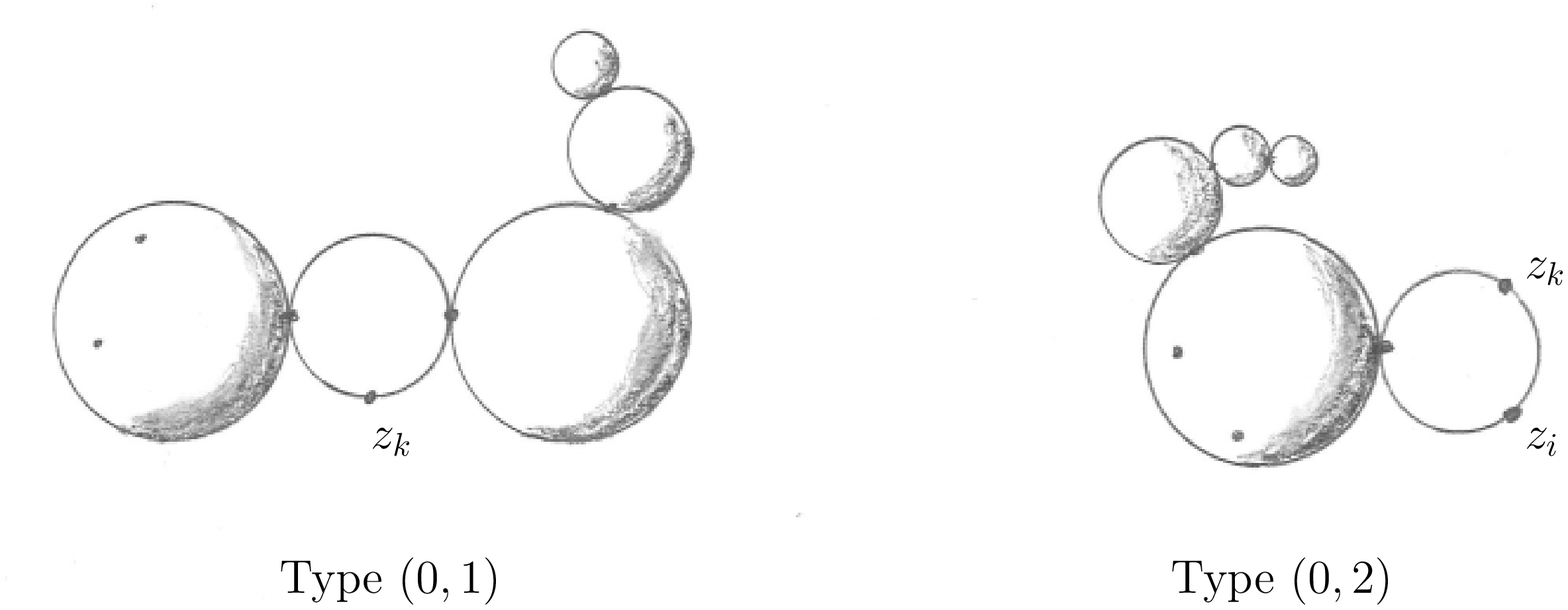}
	\caption{Destabilizing ghost components}
\end{figure}

We now consider fringe definitions of the $k$th-marked point forgetting map on the GW-polyfolds.
Notice that by definition, the following sets of stable curves must be empty:
	\[\cZ_{0,0,2} = \cZ_{0,0,1} = \cZ_{0,1,0} = \cZ_{0,0,0}= \emptyset.\]
This is because the GW-stability condition \eqref{eq:gw-stability-condition} cannot be satisfied. 
It follows that there exist trivially well-defined $\ssc$-smooth maps $ft_3 : \cZ_{0,0,3} \to \cZ_{0,0,2}=\emptyset$ and $ft_1: \cZ_{0,1,1} \to \cZ_{0,1,0}=\emptyset$, in addition to the trivially defined $k$th-marked point forgetting map with source any of the above GW-polyfolds.

Since destabilizing ghost components of types $(0,3)$ and $(1,1)$ may only arise in the fringe situations $(A,g,k)=(0,0,3)$ and $(A,g,k)=(0,1,1)$ respectively, we already have a well-defined $k$th-marked point forgetting map in these cases.

On the other hand, just as in Remark~\ref{rmk:special-cases-forgetting}, a destabilizing ghost component of type $(0,1)$ will always arise in the following situations:
	\[
	(A,g=0, k\geq 5),\quad (A,g=1, k\geq 2), \quad (A,g\geq 2, k\geq 1)
	\]
while a destabilizing ghost component of type $(0,2)$ will always arise in the following situations:
	\[
	(A,g=0, k\geq 4), \qquad (A,g\geq 1, k\geq 2).
	\]

\subsection{Failure of continuity for the \texorpdfstring{$k$th-marked}{kth-marked} point forgetting map restricted to the subset of stable curves with constant destabilizing ghost components}
	\label{subsec:failure-of-continuity-ftk}

Consider the subset $\cZ^\text{const}_{A,g,k}\subset \cZ_{A,g,k}$ of stable curves with constant destabilizing ghost components.
We may define a $k$th-marked point forgetting map
	\[
	ft_k: \cZ^\text{const}_{A,g,k} \to \cZ_{A,g,k-1}.
	\]
If $\cZ^\text{const}_{A,g,k}$ contains a destabilizing ghost component of type $(0,1)$ or of type $(0,2)$ this map is not continuous.

Failure of continuity occurs at stable curves in $\cZ^\text{const}_{A,g,k}$ which contain a component $S^2$ with precisely $3$ special points, one of which is the $k$th-marked point, and such that $\int_{S^2} u^*\ww =0$.
We demonstrate lack of continuity by exhibiting a sequence which converges to a stable curve $[\Sigma,j,M,D,u]$ but for which the image of the sequence does not converge.
For simplicity, we will assume that $[\Sigma,j,M,D,u]$ satisfies the following:
\begin{itemize}
	\item $u$ is constant on a region surrounding the component $S^2$ (this simplifies the local forms for $ft_k$ since our reference curves are now constant),
	\item the two other special points on $S^2$ are both nodal points.
\end{itemize}

Consider a $\ssc$-Banach space which consists of a gluing parameter $a\in\hb\subset\C$ and of maps
	\[\eta^+ : \R^+\times S^1\to \R^{2n}, \qquad	\eta^0 :\R \times S^1 \to \R^{2n}\]
which converge to asymptotic constants
	\[\lim_{s\to\infty} \eta^+ = \lim_{s\to-\infty} \eta^0 = c \quad \text{and}\quad \lim_{s\to\infty} \eta^0 = c'.\]
On the level $m$ we give this space the following norm:
	\begin{equation}
	\label{eq:norm-1}
	\begin{split}
	\abs{(a,\eta^+,\eta^0)}_m^2 = &\abs{a}^2 + \abs{c}^2 + \abs{c'}^2 + \sum_{|\alpha|\leq m+3} \int_{\R^+ \times S^1} \abs{D^\alpha (\eta^+-c)}^2 e^{2\delta_m\abs{s}}ds\ dt \\
	&\qquad + \sum_{|\alpha|\leq m+3} \int_{\R^- \times S^1} \abs{D^\alpha (\eta^0-c)}^2 e^{2\delta_m\abs{s}}ds\ dt \\
	&\qquad + \sum_{|\alpha|\leq m+3} \int_{\R^+ \times S^1} \abs{D^\alpha (\eta^0-c')}^2 e^{2\delta_m\abs{s}}ds\ dt. 
	\end{split}
	\end{equation}
We now construct the sequence.  Choose a smooth cut-off function $\beta : \R \to [0,1]$ such that:
\begin{itemize}
	\item $\beta (s)=1$ for all $-\tfrac{1}{2} \leq s \leq \tfrac{1}{2}$,
	\item $\beta (s)=0$ for all $-1 \leq s \leq 1$.
\end{itemize}
Choose a vector $v\in \R^{2n}$.  Then define a vector field $\gamma :\R\times S^1 \to \R^n$ by
	\[
	\gamma (s,t) := \beta(s)\cdot v.
	\]
We may renormalize $v$ by a constant so that
	\[
	\sum_{\abs{\alpha}\leq 3}\int_{\R \times S^1}\abs{D^{\alpha} \gamma (s, t)}^2e^{2\delta_0 \abs{s}}\ ds\ dt=1.
	\]
Then $\xi_n := \tfrac{1}{\sqrt{n}} \gamma$ is a vector field such that 
	\[\sum_{\abs{\alpha}\leq 3}\int_{\R \times S^1}\abs{D^{\alpha} \xi_n (s, t)}^2e^{2\delta_0 \abs{s}}\ ds\ dt=\tfrac{1}{n}.\]
Now choose $a_n\in\hb$, $a_n \neq 0$ small enough such that $e^{2\delta_0 R_{a_n}} > \tfrac{2n}{\abs{v}^2}$.
It follows that
	\[
	\abs{(a_n,0,\xi_n)}_0^2 = \abs{a_n}^2 + \tfrac{1}{n} \to 0 \quad \text{as } n\to \infty.
	\]

On the other hand, we may consider a second $\ssc$-Banach space consisting of maps
	\[\eta' : \R^+\times S^1 \to \R^{2n}\]
with asymptotic constant given by $\lim_{s\to \infty} \eta' = c''$ and with $m$ level norm
	\begin{equation}
	\label{eq:norm-2}
	\abs{\eta'}_m^2 = \abs{c''}^2 + \sum_{|\alpha|\leq m+3} \int_{\R^+ \times S^1} \abs{D^\alpha (\eta'-c'')}^2 e^{2\delta_m\abs{s}}ds\ dt.
	\end{equation}
Consider the sequence $\oplus_{a_n}(0,\xi_n)$, using the gluing procedure described in \cite[\S~2.4]{HWZGW}.
In this norm:
	\begin{align*}
	\abs{\oplus_{a_n}(0,\xi_n)}^2_0
	&=\sum_{\abs{\alpha}\leq 3}\int_{\R^+ \times S^1}\abs{D^{\alpha} \xi_n(s-R_{a_n})  }^2e^{2\delta_0 \abs{s}}ds\ dt \\
	& > \int_{[R_{a_n}, R_{a_n}+1/2]\times S^1} \left|\tfrac{1}{\sqrt{n}} \cdot v\right|^2 e^{2\delta_0 \abs{s}}ds\ dt \\
	&> \frac{1}{2}\cdot \frac{1}{n}\cdot \abs{v}^2 \cdot e^{2\delta_0 R_{a_n}} >1. 
	\end{align*}
for all $n$.

The topology of a neighborhood of $[\Sigma,j,M,D,u]$ is determined by the $0$-level norm \eqref{eq:norm-1}, moreover, because the gluing parameters $a_n$ are not equal to $0$ the sequence $(a_n,0,\xi_n)$ can be used to define a sequence $x_n \in \cZ^\text{const}_{A,g,k}$ which converges to $[\Sigma,j,M,D,u]$.
On the other hand, the topology of a neighborhood of $ft_k ([\Sigma,j,M,D,u])$ is determined by the $0$-level norm \eqref{eq:norm-2}; the image $\oplus_{a_n} (0,\xi_n)$ correspond to a sequence $ft_k(x_n) \in \cZ_{A,g,k-1}$ which does not converge.


\section{The universal curve polyfold}
	\label{sec:universal-curve-polyfold}

We show how to construct a polyfold $\cZ^\text{uc}_{A,g,k}$ of \textit{universal curves} over the GW-polyfold $\cZ_{A,g,k-1}$.
The underlying set of this polyfold may be identified with $\cZ^\text{const}_{A,g,k}$, and hence we can also consider this as a GW-polyfold of stable curves with constant destabilizing ghost components.
However, we give it a \textit{new} polyfold structure, with a \textit{new} $\ssc$-smooth structure, and a \textit{new} topology.

This new polyfold $\cZ^\text{uc}_{A,g,k}$ uses a modified gluing construction, designed to more accurately anticipate the geometry of the desired solution space.
When the destabilizing ghost component is of type $(0,1)$, it interpolates the gluing parameters surrounding a ghost component directly.
When the destabilizing ghost component is of type $(0,2)$, it forgets the gluing parameter, and relabels the remaining nodal point as a marked point.

This new polyfold carries the full abstract perturbation theory developed in \cite{HWZ3}.
We thus obtain well-defined GW-invariants for this new polyfold, see Theorem~\ref{thm:gw-invariants-new-polyfold}.
That these invariants coincide with the original polyfold GW-invariants constructed in \cite{HWZGW} is proved in Theorem~\ref{thm:equivalence-polyfold-invariants}.

Crucially, on the universal curve polyfold we may consider a well-defined $k$th-marked point forgetting map
	\[
	ft_k :\cZ^\text{uc}_{A,g,k}\to \cZ_{A,g,k-1}.
	\]
and indeed, the preimage of a point $[\Sigma,j,M,D,u] \in \cZ_{A,g,k-1}$ via $ft_k$ consists of the Riemann surface $\Sigma$ with nodes identified, i.e.,
	\[
	ft_k^{-1} ([\Sigma,j,M,D,u]) \simeq \Sigma / \sim, \quad \text{where $x_a\sim y_a$ for nodal pairs $\{x_a,y_a\}\in D$}.
	\]	
This map is $\ssc^0$, and fails to be $\ssc^1$ at stable curves which contain a destabilizing ghost component of type $(0,1)$.
As a consequence, pulling back perturbations is not quite as automatic as in \S~\ref{subsec:pulling-back-abstract-perturbations}, and will require a slightly more hands-on approach.

Throughout this section, we will assume that $(A,g,k) \neq (0,0,2),$ $(0,0,1),$  $(0,0,0),$  $(0,1,0)$ and $(A,g,k) \neq (0,0,3),$  $(0,1,1)$ as in these cases we may consider the trivially defined $k$th-marked point forgetting map $ft_k :\cZ_{A,g,k} \to \emptyset$ (see the discussion in \S~\ref{subsec:classifying-destabilizing-ghost-components}).

\subsection{Constructing the universal curve polyfold}

We begin by describing the underlying set of the universal curve polyfold.

\begin{definition}
	The \textbf{universal curve polyfold} $\cZ^\text{uc}_{A,g,k}$ is defined as the set of stable curves with constant destabilizing ghost components
	\[
	\cZ^\text{uc}_{A,g,k}:= \{ (\Sigma,j,M,D,u)  \mid  \cdots, \text{ uc-stability condition} \}/ \sim
	\]
	where $(\Sigma,j,M,D)$ is a connected noded Riemann surface (where we do not require the DM-stability condition), and which satisfies the same conditions as Definition~\ref{def:gw-polyfold}---except here we replace the GW-stability condition \eqref{eq:gw-stability-condition} with the following.
	\begin{itemize}
		\item For each connected component $C\subset \Sigma$ the following \textbf{uc-stability condition} holds.  We require at least one of the following:
		\[
		2 g_C+\sharp (M\cup\abs{D})_C \geq 3 \quad \text{or} \quad \int_C u^* \ww > 0.
		\]
		Additionally, if the $k$th-marked point $z_k$ lies on a component $C_k$ which is unstable after forgetting $z_k$ then we require that $u|_{C_k}$ is constant (hence necessarily $u|_{C_k}\equiv u(z_k)$), i.e.,
		\begin{equation}
		\label{eq:new-stability}
		2 g_{C_k} +\sharp (M\cup \abs{D})_{C_k} = 3 \quad \text{and}\quad \int_{C_k} u^*\ww =0 \implies u|_{C_k} \equiv \text{const.}
		\end{equation}
		\item We require that $u$ be of class $H^{3,\delta_0}$ at all marked points $\{z_1,\ldots z_{k-1}\}$.  We require that $u$ be of class $H^3_\text{loc}$ at the marked point $z_k$ (see Definition~\ref{def:class-3delta}).
	\end{itemize}
	We call a tuple $(\Sigma,j,M,D,u)$ which satisfies the above a \textbf{stable map with constant destabilizing ghost component $C_k$}, and call an equivalence class satisfying the above a \textbf{stable curve with constant destabilizing ghost component $C_k$}.
\end{definition}

\subsubsection{New gluing constructions at destabilizing ghost components}
	\label{new-gluing-construction}

Consider a stable map $(\Sigma,j,M,D,u)$ which satisfies the uc-stability condition \eqref{eq:new-stability}, and suppose the marked point $z_k$ lies on a destabilizing ghost component $C_k$.

By the classification of destabilizing ghost components in \S~\ref{subsec:classifying-destabilizing-ghost-components} and by the assumption $(A,g,k) \neq (0,0,3),$ $(0,0,1)$ we need only consider destabilizing ghost components of type $(0,1)$ and type $(0,2)$.

In what follows, we will define new gluing constructions for these cases, designed to more accurately model the expected behavior of the GW-moduli spaces on regions near a destabilizing ghost component.
This new gluing procedure remains identical to the DM-gluing considered on the underlying Riemann surface $\Sigma$ in \S~\ref{subsubsec:gluing-profiles-gluing-construction-noded-reimann-surface}, it only modifies the gluing construction at a stable map as described in \S~\ref{subsubsec:the-gluing-construction-stable-map}.
In order to ensure $\ssc$-smoothness of the expressions for gluing and anti-gluing it is important to use the exponential  gluing profile given by $\varphi_{\exp} (r)=e^{1/r}-e.$

\subsubsection{Gluing at destabilizing ghost components of type $(0,1)$}



For simplicity we will assume that $\{x_a,y_a\}, \{x_b,y_b\}$ are the only nodal pairs.
We recall the gluing construction for a noded Riemann surface of case $2a$ from \S~\ref{subsubsec:local-expression-forgetting-map-on-dm-orbifold}.
Writing the gluing parameters $a\neq 0$ or $b\neq 0$ in polar coordinates as 
	\begin{align*}
	a = r_a e^{-2\pi i\theta_a}, \qquad & r_a\in (0,\tfrac{1}{2}),\ \theta_a \in \R/ \Z, \\
	b = r_b e^{-2\pi i\theta_b}, \qquad & r_b\in (0,\tfrac{1}{2}),\  \theta_b \in \R/\Z,
	\end{align*}
we replace $D_{x_a}\sqcup C_k \sqcup D_{y_b}$ with the glued cylinder
	\[
	Z_{a,b} :=
	\begin{cases}
	[0,R_a+R_b]\times S^1 & a\neq 0, b\neq 0 \\
	\R^+\times S^1 \sqcup \R^-\times S^1 & \text{if } a\neq 0,b=0 \text{ or } a=0,b\neq 0\\
	\R^+\times S^1 \sqcup \R\times S^1 \sqcup \R^-\times S^1 & \text{if } a=0,b=0.
	\end{cases}
	\]
We thus obtain the glued Riemann surface 
	\[\Sigma_{a,b}: = \Sigma \setminus (D_{x_a}\sqcup C_k \sqcup D_{y_b}) \sqcup Z_{a,b}.\]

We now define the new stable map gluings.
At the outset, fix a smooth cutoff function $\beta:{\mathbb R}\rightarrow [0,1] $ which satisfies the following:
\begin{itemize}
	\item $\beta (-s)+\beta (s)=1\quad \text{for all } s\in \R$
	\item $\beta (s)=1\quad \text{for all } s\leq -1$
	\item $\tfrac{d}{ds}\beta(s)<0\quad \text{for all } s\in (-1, 1).$
\end{itemize}
Consider a pair of continuous maps
	\[
	h^+:\R^+\times S^1\to \R^{2n}, \qquad h^-:\R^-\times S^1\to \R^{2n}
	\]
with common asymptotic constant $c:= \lim_{s\to  \infty} h^+(s,t)= \lim_{s\to  -\infty} h^-(s,t)$.
For given gluing parameters $a,b\in \hb$, $a\neq 0$, $b\neq 0$ we define the \textbf{glued map of type $(0,1)$},  $\oplus^\text{uc}_{a,b} (h^+,h^-) : Z_{a,b} \to \R^{2n}$, by
	\begin{align*}
	\oplus^\text{uc}_{a,b} (h^+,h^-) (s,t):=
		& \beta\left(	s - \tfrac{R_a+R_b}{2}	\right) \cdot h^+(s,t) \\
		&+ \left(1-\beta \left(	s-\tfrac{R_a+R_b}{2}	\right)\right) \cdot h^-(s-R_a-R_b,t-\theta_a-\theta_b)
	\end{align*}
For other values of $a$ and $b$ we define $\oplus^\text{uc}_{a,b} (h^+,h^-)$ by
	\begin{align*}
	&\oplus^\text{uc}_{a,b} (h^+,h^-):= \\
	&\qquad \begin{cases}
	(h^+,h^-):\R^+\times S^1\sqcup \R^-\times S^1 \to \R^{2n} 	& \text{if } a\neq 0, b=0 \\
																& \text{or } a=0, b\neq 0, \\
	(h^+,c,h^-) : \R^+\times S^1\sqcup \R\times S^1 \sqcup \R^-\times S^1 \to \R^{2n} & \text{if } a=0, b=0,
	\end{cases}
	\end{align*}
where $c:\R\times S^1 \to \R^{2n}$ is the constant map to $c\in \R^{2n}$

Consider the following space of sections $H^{3,\delta_0}_{c,\text{const}}(\Sigma, u^*TQ)$, consisting of sections $\eta : \Sigma \to u^*TQ$ such that:
$\eta$ is of class $H^{3,\delta_0}$ around the nodal points and of class $H^3_\text{loc}$ at the other points of $\Sigma$, $\eta$ has matching asymptotic values at the nodal pairs and is constant on the destabilizing ghost component, i.e., $\eta(x_a) = \eta(C_k)=\eta(y_b)$.
We now consider local expressions for the map $u$ and section $\eta$ as follows.
In a neighborhood of the point $u(x_a)=u(C_k)=u(y_b) \in Q$ choose a chart which identifies the point with $0\in \R^{2n}$. 
Furthermore, given a Riemannian metric $g$ on $Q$ we may assume that this chart is chosen such that this metric is identifiable with the Euclidean metric on $\R^{2n}$.
Localized to these coordinate neighborhoods, we may view the base map $u$ as maps 
	\[u^+:\R^+\times S^1\to \R^{2n}, \quad 0:\R\times S^1 \to \R^{2n}, \quad u^-:\R^-\times S^1\to \R^{2n}\]
and likewise the section $\eta$ as maps
	\[\eta^+:\R^+\times S^1\to \R^{2n}, \quad c:\R\times S^1 \to \R^{2n}, \quad \eta^-:\R^-\times S^1\to \R^{2n}\]
such that $c:=\lim_{s\to  \infty} \eta^+ = \lim_{s\to  -\infty} \eta^-$ and where $c:\R\times S^1 \to \R^{2n}$ is the constant map to $c\in \R^{2n}$.

Given a gluing parameters $a,b \in \hb$ we define the \textbf{new glued stable map} $\oplus^\text{uc}_{a,b} \exp_u (\eta) :\Sigma_{a,b} \to Q$ as follows:
	\[
	\oplus^\text{uc}_{a,b} \exp_u \eta :=
	\begin{cases}
	\exp_u \eta 							& \text{ on } \Sigma \setminus (D_{x_a}\sqcup C_k\sqcup D_{y_a}), \\
	\oplus^\text{uc}_{a,b} (u^+ + \eta^+, u^- + \eta^- ) 	& \text{ on } Z_{a,b}.
	\end{cases}
	\]

We will sometimes use the abbreviations
$R_{a,b}:= \varphi (\abs{a}) +\varphi(\abs{b})$ and $\theta_{a,b} := \theta_a+\theta_b.$
If $a,b\in \hb$ are gluing parameters with $a\neq 0$, $b\neq 0$, we define the cutoff function $\beta_{a,b}:\R\to \R$ by 
	\[
	\beta_{a,b} (s):=\beta \left( s-\tfrac{R_{a,b}}{2}\right).
	\]

In order to describe data that would otherwise be lost in the gluing procedure, we define corresponding cylinders $C_{a,b}$ by
	\[
	C_{a,b} :=
	\begin{cases}
	\R\times S^1	&\text{when } a\neq 0, b\neq 0 \\
	\emptyset		&\text{otherwise}.
	\end{cases}
	\]
For $a\neq 0$ and $b\neq 0$ define the \textbf{new anti-glued section} 
as the map $\ominus^\text{uc}_{a,b}(\eta^+,\eta^-) : C_{a,b}\to \R^{2n}$,
	\begin{align*}
	\ominus^\text{uc}_{a,b} (\eta^+,\eta^-)(s,t):=
	& -\left(1-\beta_{a,b}(s) \right) \cdot \left(	\eta^+(s,t)	-\av_{a,b}(\eta^+,\eta^-)\right)\\
	&\qquad + \beta_{a,b}(s) \cdot \left(\eta^-(s-R_{a,b},t-\theta_{a,b}) -	\av_{a,b}(\eta^+,\eta^-)\right)
	\end{align*}
where
	\[
	\av_{a,b} (\eta^+,\eta^-):= \frac{1}{2} \left(	\int_{S^1} \eta^+(\tfrac{R_{a,b}}{2},t)dt + \int_{S^1} \eta^-(-\tfrac{R_{a,b}}{2},t)dt	\right).
	\]
For other values of $a$ and $b$ we define $\ominus^\text{uc}_{a,b} (\eta^+,\eta^-)$ as the unique map $\emptyset \to \R^{2n}$.

Letting $U\subset H^{3,\delta_0}_{c,\text{const}}(\Sigma, u^*TQ)$ be a sufficiently small neighborhood of the zero section, we have an analog of Theorem~\ref{thm:ssc-retraction}.
\begin{proposition}
	There exists a $\ssc$-retraction, i.e., a $\ssc$-smooth map which satisfies $\pi \circ \pi = \pi$,
	\begin{align*}
	\pi: \hb \times\hb \times V \times U 	& \to \hb \times\hb \times V \times U \\
	(a,b,v,\eta)				& \mapsto (a,b,v,\pi_{a,b} (\eta)),
	\end{align*}
	such that the restriction of the above family to the subset $\cV := \pi (\hb \times\hb \times V \times U)$ is injective.
\end{proposition}
\begin{proof}
	We may consider a simplified local expression 
	\begin{align*}
	\pi: \hb\times\hb \times E^\pm &\to  \hb\times\hb\times E^\pm \\
	 (a,b,\xi^+,\xi^-)&\mapsto (a,b,\eta^+,\eta^-)
	\end{align*}
	where $\eta^\pm$ are uniquely defined by the equations:
	\[
	\oplus^\text{uc}_{a,b}(\eta^+,\eta^-) = \oplus^\text{uc}_{a,b}(\xi^+,\xi^-) \quad \text{and}\quad \ominus^\text{uc}_{a,b} (\eta^+,\eta^-) =0.
	\]
	Abbreviate $\gamma_{a,b}=\beta_{a,b}^2(s^+)+(1-\beta_{a,b}(s^+))^2$.  We may write down the following explicit formulas:
	\begin{align*}
	\eta^+(s^+, t^+) =	&\left(1-\frac{\beta_{a,b}}{\gamma_{a,b}}\right)\cdot \av_{a,b}(\xi^+, \xi^-) +\frac{\beta_{a,b}^2}{\gamma_{a,b}}\cdot \xi^+(s^+, t^+)\\
	&+ \frac{\beta_{a,b} (1-\beta_{a,b})}{\gamma_{a,b}}\cdot \xi^-(s^+-R_{a,b}, t^+ -\theta_{a,b})
	\end{align*}
	for $(s^+, t^+)\in \R^+\times S^1$.  A similar calculation leads to the following formula for $\eta^-$:
	
	\begin{align*}
	\eta^-(s^-, t^-)=	&\left( 1-\frac{\beta_{a,b} (- s^-)}{\gamma_{a,b} (-s^-)}\right)\cdot \av_{a,b} (\xi^+, \xi^-)\\
	&+\frac{\beta_{a,b} (-s^- ) (1-\beta_{a,b} (-s^- ))}{\gamma_{a,b}(-s^- )}\cdot \xi^+(s^-+R_{a,b}, t^- +\theta_{a,b}) \\
	&+\frac{\beta_{a,b} (-s^- )^2}{\gamma_{a,b}(-s^- )}\xi^-(s^-, t^-)
	\end{align*}
	for $(s^-, t^-)\in \R^-\times S^1$.
	

	The argument that the retraction $\pi$ is $\ssc$-smooth then follows precisely the same reasoning as in the proof that the projection $\pi_a$ defined via the usual gluing and anti-gluing is $\ssc$-smooth, namely by checking differentiability of individual terms in the explicit formula for $\eta^\pm$ and applying the chain rule.  Full details must follow the lengthy arguments given in \cite[\S~2.4]{HWZsc}.
\end{proof}

As in \cite[pp.~60--61]{HWZGW}, there are analogous new hat gluings and new hat antigluings, used to define the strong polyfold bundles.

\subsubsection{Gluing at destabilizing ghost components of type $(0,2)$}

For simplicity, we will assume that $\{x_a,y_a\}$ is the only nodal pair.
Let $a\in \hb$ be the associated gluing parameter.
Recall the gluing construction for a noded Riemann surface of case $2b$ from \S~\ref{subsubsec:local-expression-forgetting-map-on-dm-orbifold}, in particular, note that when $a\neq 0$, the Riemann surface $\Sigma_a$ is obtained by simply deleting the component $C_k$.

In this case we use the single gluing parameter to parametrize movement of the marked point $z_k$; we do not interpolated maps across the node.
We consider the space $\eta \in H^{3,\delta_0}_{c,\text{const}}(\Sigma, u^*TQ)$ consisting of sections $\eta: \Sigma \to u^*TQ$ such that: $\eta$ is of class $H^{3,\delta_0}$ around the nodal points and of class $H^3_\text{loc}$ at the other points of $\Sigma$, $\eta$ is constant on the destabilizing ghost component, i.e., $\eta(x_a) = \eta (C_k)$.
The new glued stable map $\oplus^\text{uc}_a \exp_u \eta : \Sigma_a \to Q$ is then defined by:
	\[
	\oplus^\text{uc}_a \exp_u \eta :=
	\begin{cases}
	\exp_u \eta |_{\Sigma\setminus C_k}	 & a\neq 0, \\
	\exp_u \eta |_{\Sigma\setminus C_k}, \exp_u \eta (C_k)|_{C_k}			 & a=0.
	\end{cases}
	\]

\subsubsection{Good uniformizing families of stable maps with constant destabilizing ghost components}

As in \S~\ref{subsubsec:families-stable-maps-in-general} we may choose a stabilization together with linear constraints.
The Riemann surface $(\Sigma,j,M\cup S,D)$ is now stable; let 
	\[
	(a,v)\mapsto (\Sigma_a, j(a,v), (M\cup S)_a, D_a), \quad (a,v) \in \hb^{\sharp D} \times V
	\]
be a good uniformizing family of stable noded Riemann surfaces.
Using the linear constraints, we consider the constrained subspace of sections
	\[
	E_S := \{ \eta \in H^{3,\delta_0}_{c,\text{const}} (\Sigma, u^*TQ)	\mid \eta(z_s) \in H_{u(z_s)} \text{ for } z_s \in S\}
	\]
and let $U \subset E_S$ be an suitably small open neighborhood of the zero section.
One can then define a $\ssc$-retraction
	\begin{align*}
	\pi : \hb^{\sharp D} \times V \times U 	& \to \hb^{\sharp D} \times V \times U\\
	(a,v,\eta)									& \mapsto (a,v,\pi_a(\eta)).
	\end{align*}
The image $\cU : = \pi (\hb^{\sharp D} \times V \times U)$ is a $\ssc$-retract on which the gluing map is injective.

\begin{definition}
	Analogously to Definition~\ref{def:good-uniformizing-family-of-stable-maps}, we define a \textbf{good uniformizing family of stable maps associated to the new gluing} centered at $(\Sigma,j,M,D,u)$ as a family of stable maps
	\[
	(a,v,\eta) \mapsto (\Sigma_a, j(a,v), M_a, D_a, \oplus^\text{uc}_a \exp_u \eta), \quad (a,v,\eta )\in \cU.
	\]
	In particular, the section $\eta$ satisfies the linear constraint $H_{u(z_s)}$ at each stabilizing point $z_s\in S$.

	If the stable map $(\Sigma,j,M,D,u)$ contains a destabilizing ghost component of type $(0,1)$, the new glued map $\oplus^\text{uc}_a \exp_u \eta : \Sigma_a \to Q$ can be described as follows:
	\[
	\oplus^\text{uc}_{a,b} \exp_u \eta :=
	\begin{cases}
	\exp_u \eta					& \text{on } \Sigma \setminus \cup_{\{x_a,y_a\}\in D} (D_{x_a}\sqcup D_{y_a}), \\
	\oplus^\text{uc}_{a,b} \exp_u \eta 	& \text{on } Z_{a,b}, \\
	\oplus_{a'} \exp_u \eta  & \text{on } Z_{a'} \text{ for all other gluing parameters}.
	\end{cases}
	\]
	
	If the stable map $(\Sigma,j,M,D,u)$ contains a destabilizing ghost component of type $(0,2)$, we define the new glued map $\oplus^\text{uc}_a \exp_u \eta : \Sigma_a \to Q$ can be described as follows:
	\[
	\oplus^\text{uc}_a \exp_u (\eta) :=
	\begin{cases}
	\exp_u \eta 			& \text{ on } \Sigma \setminus \cup_{\{x_a,y_a\}\in D} (D_{x_a}\sqcup D_{y_a}) \\
	\exp_u \eta				& \text{ on the small disks bordering the $i$th-marked point} \\
	\oplus_{a'} \exp_u \eta	& \text{ on } Z_{a'} \text{ for all other gluing parameters}.
	\end{cases}
	\]	
\end{definition}

\subsubsection{Local smooth structures on the universal curve polyfold}

In practice, the construction of a topology and the construction of a polyfold structure on a candidate set $\cZ$ are intertwined. A key step in the construction is proving a compactness property holds for the morphism set as in \cite[Prop.~3.22]{HWZGW}. 
Alternatively, one should use the updated approach of \cite[Part~IV]{HWZbook} to recast the process in terms of ``groupoidal topological categories.''

Following the arguments of \cite[\S~3.4]{HWZGW}, we may assert that the set $\cZ^\text{uc}_{A,g,k}$ has a natural second countable, paracompact, Hausdorff topology.
By construction of this topology, the map
	\[
	\frac{\{(\Sigma_a,j(a,v), M_a,D_a, \oplus^\text{uc}_a \exp_u \eta)\}_{(a,v,\eta)\in \cU}}{\Aut(\Sigma,j,M,D,u)} \to \cZ^\text{uc}_{A,g,k}.
	\]
is a local homeomorphism.

Following the arguments of \cite[\S~3.5]{HWZGW}, we can construct a polyfold structure on the set of universal curves.

\begin{theorem}
	Fix a strictly increasing sequence $(\delta_i)_{i\geq 0}\subset (0,2\pi)$.
	The second countable, paracompact, Hausdorff topological space $\cZ^\text{uc}_{A,g,k}$ possesses a natural equivalence class of polyfold structures.
\end{theorem}

\subsection{Invariants for the universal curve polyfold}

Just as in \S~\ref{sec:the-polyfold-gw-invariants} we may define invariants for the universal curve polyfold.

\begin{definition}
	\label{def:new-strong-polyfold-bundle}
	As in Definition~\ref{def:strong-polyfold-bundle}, the underlying set of the strong polyfold bundle $\cW^{\text{uc}}_{A,g,k}$ is defined as the set of equivalence classes
		\[
		\cW^{\text{uc}}_{A,g,k}:= \{ (\Sigma,j,M,D,u,\xi ) \mid \cdots \} / \sim
		\]
	with data as follows.
	\begin{itemize}
		\item $(\Sigma,j,M,D,u)$ is a stable map representative of a stable curve in $\cZ^{\text{uc}}_{A,g,k}$.
		\item $\xi$ is a continuous section along $u$ such that the map
		\[
		\xi(z):T_z\Sigma\rightarrow T_{u(z)}Q,\quad \text{for } z\in
		\Sigma
		\]
		is a complex anti-linear map.
		\item At a destabilizing ghost component $C_k \subset \Sigma$ we define $\xi$ to be the zero section.
		\item As in Definition~\ref{def:strong-polyfold-bundle} we require that the local expressions for $\xi$ are of class $H^{2,\delta_0}$ around the nodal points in $\abs{D}$ and around the marked points $\{z_1,\ldots,z_{k-1}\}$.
		We require the local expression is of class $H^2_\text{loc}$ near the other points in $\Sigma$.
		\item The equivalence relation is given by 
		\[(\Sigma,j,M,D,u,\xi)\sim (\Sigma',j',M',D',u',\xi')\]
		if there exists a biholomorphism $\phi :(\Sigma,j)\to (\Sigma',j')$ which satisfies $\xi'\circ d\phi=\xi$ in addition to $u'\circ \phi = u$, $\phi(M)= M'$, $\phi (\abs{D})=\abs{D'}$, and which preserves ordering and pairs.
	\end{itemize}
\end{definition}	
		
We claim that we can give $\cW^{\text{uc}}_{A,g,k}$ a polyfold structure such that
	\[
	P:\cW^{\text{uc}}_{A,g,k}\to \cZ^{\text{uc}}_{A,g,k}, \qquad [\Sigma,j,M,D,u,\xi]\mapsto [\Sigma,j,M,D,u]
	\]
defines a strong polyfold bundle over the universal curve polyfold $\cZ^{\text{uc}}_{A,g,k}$.
The Cauchy--Riemann section $\delbarj$ of the strong polyfold bundle  $P:\cW^{\text{uc}}_{A,g,k}\to \cZ^{\text{uc}}_{A,g,k}$ is a proper $\ssc$-smooth Fredholm section.  The Fredholm index of $\delbarj$ is given by
	\[\ind \delbarj =2c_1(A)+(\text{dim}_{\R} Q-6)(1-g)+2k.\]
Notice that the unperturbed GW-moduli space $\CM_{A,g,k}(J)$ is identical for both the universal curve polyfold $\cZ^{\text{uc}}_{A,g,k}$ and the usual GW-polyfold $\cZ_{A,g,k}$
	
By \cite[Cor.~15.1]{HWZbook} there exist regular perturbations of the Cauchy--Riemann section, hence the perturbed moduli space 	
	\[
	\cS^{\text{uc}}_{A,g,k}(p) : = \text{`` }(\delbarj + p)^{-1} (0) \text{ ''} \subset \cZ^{\text{uc}}_{A,g,k}
	\]
has the structure of a compact oriented weighted branched suborbifold.

As in \S\S~\ref{subsec:branched-integrals} and \ref{subsec:intersection-numbers}, we may use this perturbed moduli space to define invariants.

\begin{theorem}
	\label{thm:gw-invariants-new-polyfold}
	We define the \textbf{Gromov--Witten invariant for the universal curve polyfold} as the homomorphism
		\[
		\GW^{\text{uc}}_{A,g,k} : H_*(Q;\Q)^{\otimes k} \otimes H_*(\dmlog_{g,k};\Q ) \to \Q
		\]
	defined via either the branched integration of \cite{HWZint}:
		\[
		\GW^{\text{uc}}_{A,g,k} (\alpha_1,\ldots,\alpha_k;\beta) := \int_{\cS^\text{uc}_{A,g,k}(p)} ev_1^* \PD (\alpha_1)\wedge \cdots \wedge ev_k^* \PD (\alpha_k) \wedge\pi^* \PD (\beta)
		\]
	 or the intersection number of \cite{schmaltz2019steenrod}:
	 	\[
	 	\GW^{\text{uc}}_{A,g,k} ([\cX_1],\ldots,[\cX_k];[\cB]) := \left(ev_1\times\cdots\times ev_k\times\pi\right)|_{\cS^{\text{uc}}_{A,g,k}(p)} \cdot \left(\cX_1 \times\cdots\times \cX_k \times \cB\right).
	 	\]
	This invariant does not depend on the choice of perturbation, nor on choice of basis of representing submanifolds/suborbifolds.
\end{theorem}

\subsection{Naturality of the polyfold invariants for the universal curve and the Gromov--Witten polyfolds}

Consider the unperturbed GW-moduli space $\CM_{A,g,k}(J)$ defined in the introduction as the stable map compactification of the set of $J$-holomorphic curves.
We have distinct polyfolds $\cZ^{\text{uc}}_{A,g,k}$ and $\cZ_{A,g,k}$ which contain $\CM_{A,g,k}(J)$ as a compact subset; after regularization we obtain distinct perturbed moduli spaces and moreover distinct polyfold invariants $\GW^{\text{uc}}_{A,g,k}$ and  $\GW_{A,g,k}$ which, a priori, we cannot assume are equivalent.
This is precisely the type of problem described in \cite{schmaltz2019naturality}.

We may consider a commutative diagram of inclusion maps between polyfolds and between strong polyfold bundles:
	\[\begin{tikzcd}
	\cW^{\text{uc}}_{A,g,k} \arrow[r, hook] \arrow[d, "\delbarj \quad"'] & \cW_{A,g,k} \arrow[d, "\quad \delbarj"] &  \\
	\cZ^{\text{uc}}_{A,g,k} \arrow[r, hook] \arrow[u, bend left] & \cZ_{A,g,k} \arrow[u, bend right] & 
	\end{tikzcd}\]
in addition to a commutative diagram:
	\[\begin{tikzcd}
	&  & Q^k \times \dmlog_{g,k} \\
	\cZ^{\text{uc}}_{A,g,k} \arrow[r, hook] \arrow[rru, "ev\times\pi"] & \cZ_{A,g,k} \arrow[ru, "ev\times\pi"'] & 
	\end{tikzcd}
	\]
One should immediately note that $\cW^{\text{uc}}_{A,g,k}$ is not the pullback bundle of $\cW_{A,g,k}$ (recall we require the complex anti-linear sections $\xi$ are constant on destabilizing ghost components in $\cW^{\text{uc}}_{A,g,k}$).

It is straightforward to check that the maps in these diagrams satisfy the same properties as described in \cite[\S~3.3]{schmaltz2019naturality}.
The critical hypothesis is the existence of a ``intermediary subbundle'' \cite[Def.~3.15]{schmaltz2019naturality}.

\begin{proposition}
	The set
		\begin{align*}
		\cR := \{	[\Sigma,j,M,D,u,\xi] \in \cW_{A,g,k} \mid 	&\supp \xi \subset K \subset \Sigma \setminus (S^2_k \sqcup \{x_a,y_b\}) \\
																&\qquad\qquad\qquad\text{ for some compact } K\},
		\end{align*}
	(i.e., the subset of complex anti-linear sections which are supported away from a possible destabilizing ghost component and any adjacent nodes)
	is an intermediary subbundle of the strong polyfold bundle $\cW_{A,g,k}$.
\end{proposition}
\begin{proof}
	The proof follows the same reasoning as the proofs of \cite[Props.~3.17,~3.18]{schmaltz2019naturality}.
	In particular, by \cite[Cor.~A.4]{schmaltz2019naturality} one can find vectors which span the cokernel and vanish on the destabilizing ghost component and on disk-like regions of any adjacent nodes.
\end{proof}

We therefore satisfy the hypothesis of \cite[Thm.~1.3]{schmaltz2019naturality}, and hence we immediately obtain the following theorem.

\begin{theorem}
	\label{thm:equivalence-polyfold-invariants}
	The polyfold GW-invariants associated to the universal curve polyfold $\cZ^{\text{uc}}_{A,g,k}$ and to the GW-polyfold $\cZ_{A,g,k}$ are identical, i.e., 
		\[\GW^{\text{uc}}_{A,g,k}= \GW_{A,g,k}.\]
\end{theorem}

\subsection{Pulling back perturbations to the universal curve polyfold via the \texorpdfstring{$k$th-marked}{kth-marked} point forgetting map}
	\label{subsec:pulling-back-via-ftk}

The entire point of defining the universal curve polyfold is to be able to consider a well-defined $k$th-marked point forgetting map
	\[
	ft_k : \cZ^\text{uc}_{A,g,k} \to \cZ_{A,g,k-1}.
	\]
We now how to pullback abstract perturbations via this map.

\subsubsection{The $k$th-marked point forgetting map redux}

Recall that we require the stable curves in $\cZ^{\text{uc}}_{A,g,k}$ are of class $H^{3,\delta_0}$ at all marked points $\{z_1,\ldots z_{k-1}\}$, and of class $H^3_\text{loc}$ at the marked point $z_k$.
In order to get a well-defined map, we then require that the stable curves in $\cZ_{A,g,k-1}$ are of class $H^{3,\delta_0}$ at all marked points $\{z_1,\ldots z_{k-1}\}$.

\begin{definition}
	\label{def:kth-marked-point-forgetting-map-redux}
	We define the \textbf{$k$th-marked point forgetting map}
	\[
	ft_k :\cZ^{\text{uc}}_{A,g,k}\to \cZ_{A,g,k-1}
	\]
	on the underlying sets of the GW-polyfolds as follows.  Let $[\Sigma,j,M,D,u]\in \cZ^{\text{uc}}_{A,g,k}$ be a stable curve and let $(\Sigma,j,M,D,u)$ be a stable map representative. To define $ft_k$ we distinguish three cases for the component $C_k$ which contains the $k$th-marked point, $z_k \in C_k$.
	\begin{enumerate}
		\item The component $C_k\setminus \{z_k\}$ satisfies the usual GW-stability condition \eqref{eq:gw-stability-condition}, i.e.,
		\[
		2 g_{C_k} +\sharp ((M\setminus\{z_k\}) \cup \abs{D})_{C_k} \geq 3 \quad \text{or}\quad \int_{C_k} u^*\ww >0.
		\]
		We therefore define
		\[
		ft_k ([\Sigma,j, M,D,u]) = [\Sigma,j, M\setminus\{z_k\},D,u].
		\]
		\item[(2a)] The component $C_k$ is a destabilizing ghost component of type $(0,1)$.  Then $ft_k ([\Sigma,\allowbreak j,\allowbreak  M,\allowbreak D,\allowbreak u])$ is give by the stable curve obtained as follows.  Delete $z_k$, delete the component $C_k$, and delete the two nodal pairs.  We add a new nodal pair $\{x_a,y_b\}$ given by two points of the former nodal pairs.
		\item[(2b)] The component $C_k$ is a destabilizing ghost component of type $(0,2)$.  Then $ft_k ([\Sigma,\allowbreak j,\allowbreak  M,\allowbreak D,\allowbreak u])$ is give by the stable curve obtained as follows.  Delete $z_k$, delete the component $C_k$, and delete the nodal pair.  We add a new marked point $z_i$, given by the former nodal point which did not lie on $C_k$.
	\end{enumerate}
\end{definition}

\subsubsection{Local expressions for the $k$th-marked point forgetting map on the universal curve polyfold}

We now write down local expressions for $ft_k$ in terms of local scale coordinates, in an analagous way as in \S~\ref{subsubsec:local-expression-forgetting-map-on-dm-orbifold}.

\noindent\textit{Case 1.}
Using an alternative good uniformizing family to parametrize the movement of the $k$th-marked point directly, 
and noting that we may choose identical stabilizations for the source and target,
a local expression for $\hat{ft}_k$ is given by:
	\begin{align*}
	\hat{ft}_k : & \{(\Sigma_a,j(a,v),M_{(a,y)}, D_a, \oplus_a^{\text{uc}} \exp_u \eta)\}_{(a,v,y, \eta)\in \cU} \\
	& \qquad\to \{(\Sigma_b, j(b,w), (M\setminus\{z_k\})_b, D_b, \oplus_b \exp_u \zeta)\}_{(b,w,\zeta)\in \cV} \\
	&(a,v,y,\eta) \mapsto (a,v, \eta).
	\end{align*}
\noindent\textit{Case 2a.}
By construction, the interpolation given by the new gluing at a destabilizing ghost component of type $(0,1)$ for a parameter $(a,b)$ is identical to the usual gluing at a parameter $c$ are the same precisely when $c = a\ast_{\exp} b$, i.e.,
	\[
	\oplus_{a,b}^\text{uc} (\eta^+,\eta^-) = \oplus_{a\ast_{\exp} b} (\eta^+,\eta^-).
	\]
The anti-gluings are related by a similar equation, and it isn't hard to then show the appropriate $\ssc$-retractions are identical when considered on gluing parameters $(a,b)$ and $a\ast_{\exp} b$.
Using this observation, a local expression for $\hat{ft}_k$ is given by:	
	\begin{align*}
	\hat{ft}_k :	&\{(\Sigma_{a,b},j(a,b,v),M_{a,b},\{\{x_a,y_a\}, \{x_b,y_b\}\}_{a,b}, \oplus_{a,b}^\text{uc} \exp_u \eta)\}_{(a,b,v,\eta)\in \cU} \\
	&\qquad\to \{((\Sigma\setminus C_k)_c, j(c,w), (M\setminus \{z_k\})_c, \{x_c,y_c\}_c, \oplus_c \exp_u \zeta)\}_{(c,w,\zeta)\in \cV} \\
	&(a,b,v, \eta)  \mapsto (a\ast_{\exp} b, v, \eta).
	\end{align*}
\noindent\textit{Case 2b.}
In this case, the gluing parameter is only used to keep track of the $k$th-marked point and there is no interpolation involved. A local expression for $\hat{ft}_k$ is given by:	
\begin{align*}
	\hat{ft}_k : &\{(\Sigma_a,j(a,v),M_a,\{\{x_a,y_a\}\}_a, \exp_u \eta)\}_{(a,v,\eta)\in \cU} \\ 
	&\qquad\to \{(\Sigma\setminus C_k, j(w), M\setminus \{z_k\}, \emptyset, \exp_u \zeta)\}_{(w,\zeta)\in \cV} \\
	&(a,v,\eta)   \mapsto (v,\eta).
	\end{align*}

\subsubsection{Pulling back perturbations via the $k$th-marked point forgetting map}

We now may consider pullbacks of perturbations via $ft_k$ as in \S~\ref{subsec:pulling-back-abstract-perturbations}.

\begin{theorem}
	\label{thm:pulling-back-via-ftk}
	We can construct a regular perturbation which pulls back to a regular perturbation via the $k$th-marked point forgetting map.
	Thus, we can consider a well-defined restriction between perturbed GW-moduli spaces,
	\[
	ft_k : \cS^{\text{uc}}_{A,g,k} (ft_k^*p) \to \cS_{A,g,k-1}(p).
	\]
\end{theorem}
\begin{proof}
	The methods of \cite[\S~4]{schmaltz2019naturality} which we recalled in \S~\ref{subsec:pulling-back-abstract-perturbations} are complicated by the fact that $ft_k$ is $\ssc^0$, and fails to be $\ssc^1$ at stable curves which contain a destabilizing ghost component of type $(0,1)$.  This is an unavoidable consequence of the fact that our construction of the GW-polyfolds uses as a base the exponential DM-orbifolds, and we have shown in Problem~\ref{prob:ftk-continuous-not-c0} that $ft_k :\dmexp_{g,k} \to \dmexp_{g,k-1}$ fails to be $C^1$ at precisely the components of type $(0,1)$.

	We may still consider the pullback via $ft_k$ of the strong polyfold bundle $\cW_{A,g,k-1} \to \cZ_{A,g,k-1}$ and the Cauchy--Riemann section $\delbarj$, as illustrated in the below commutative diagram.
	\[
		\begin{tikzcd}
		ft_k^* \cW_{A,g,k-1} \arrow[d, "ft_k^* \delbarj\quad "'] \arrow[r, "\text{proj}_2"'] & \cW_{A,g,k-1} \arrow[d, "\quad \delbarj"] &  \\
		\cZ^{\text{uc}}_{A,g,k} \arrow[r, "ft_k"'] \arrow[u, bend left] & \cZ_{A,g,k-1} \arrow[u, bend right] & 
		\end{tikzcd}
	\]
	However, the pullback $ft_k^* \mathcal{W}_{A,g,k-1}$ does not carry a $\ssc$-smooth structure.  We may replace the \'etale condition with a $\ssc^0$-\'etale condition (where the source and target maps are required to be surjective local homeomorphisms) and hence we may consider $ft_k^* \mathcal{W}_{A,g,k-1}$ as carrying a \textit{topological} polyfold structure.
	
	
	The strong polyfold bundle $\cW^{\text{uc}}_{A,g,k}\to \cZ^{\text{uc}}_{A,g,k}$ carries a $\ssc$-smooth structure.  Moreover, there is a natural $\ssc^0$-homeomorphism $\cW^{\text{uc}}_{A,g,k} \simeq ft_k^* \cW_{A,g,k-1}$.  We may therefore consider the pullback of a parametrized $\ssc^+$-multisection $\Lambda^t : \cW_{A,g,k-1} \to \Q^+$ as defining a parametrized $\ssc^0$-multisection $\text{proj}_2^* \Lambda^t :\mathcal{W}^{ft}_{A,g,k} \to \Q^+$.  Observe that since we are pulling back via a $\ssc^0$-map, the local section structures can only be assumed to be $\ssc^0$.
	
	A multisection which is $\ssc^0$ is unsuitable for running a transversality argument.  However, if we are careful in our construction of the $\ssc^+$-multisection $\Lambda$ we can actually ensure that the pullback local section structures $ft_k^* s^t_1, \ldots, ft_k^*s^t_j$ will be $\ssc$-smooth.  The main idea is the following: in the local expressions we can pinpoint exactly where the failure of differentiability occurs; the map between the gluing parameters $(a,b)\mapsto a\ast_{\exp} b $ fails to be $C^1$ at points $(a,b)$ with $a\ast_{\exp} b=0$.  We can define a cutoff function $\beta : \hb \subset \C\to [0,1]$ to be constant on a small neighborhood of the gluing parameter $c=0$.  Hence, while the expression $a \ast_{\exp} b$ is not $C^1$, the cutoff 
		\[\C\times \C\to \R, \qquad (a,b)\mapsto \beta (a\ast_{\exp} b)\]
	is smooth.
	
	Following this observation, it is easy to show that the methods of \cite[\S~4]{schmaltz2019naturality} may be used to achieve simultaneous transversality.  In order to achieve simultaneous compactness, we note that auxiliary norms are only assumed to be $\ssc^0$, and hence the pullback of an auxiliary norm by $ft_k$ gives a well-defined auxiliary norm on the strong polyfold bundle $\cW^{\text{uc}}_{A,g,k}\to \cZ^{\text{uc}}_{A,g,k}$.
	It is then a topological exercise to show that the map $ft_k$ satisfies the topological pullback condition.
\end{proof}


\section{The Gromov--Witten axioms}
	\label{sec:gw-axioms}

We restate and prove the Gromov--Witten axioms for the polyfold Gromov--Witten invariants.

\begin{axiom*}{Effective axiom}
	If $\ww (A) <0 $ then $\GW_{A,g,k} = 0$.
\end{axiom*}
\begin{proof}
	The energy of a smooth map $u:\Sigma \to Q$ is defined as
	\[
	E(u):= \frac{1}{2} \int_\Sigma \abs{du}_j^2 \text{dvol}_\Sigma.
	\]
	By the energy identity, a $J$-holomorphic map must satisfy $\ww(A) = E(u) \geq 0$ (for example, see \cite[Lem.~2.2.1]{MSbook}).
	Hence, the unperturbed GW-moduli space is the empty set, i.e., $\CM_{A,g,k}(J)= \emptyset$, and hence the Cauchy--Riemann section is trivially transverse without perturbation.  Therefore,
	\[
	\GW_{A,g,k} (\alpha_1,\ldots,\alpha_k;\beta) = 	\int_{\emptyset} ev_1^* \PD (\alpha_1)\wedge\cdots \wedge ev_k^* \PD(\alpha_k) \wedge \pi^* \PD (\beta) = 0.
	\]
\end{proof}

\begin{axiom*}{Grading axiom}
	If $\GW_{A,g,k} (\alpha_1,\ldots,\alpha_k; \beta) \neq 0$ then
	\[
	\sum_{i=1}^k (2n - \deg (\alpha_i)) + (6g-6+2k - \deg(\beta)) = 2c_1(A) + (2n-6)(1-g) + 2k.
	\]
\end{axiom*}
\begin{proof}
	The left hand side is the codegree of $\alpha_1\times \cdots \alpha_k \times \beta$ in the product $Q^k\times \dmlog_{g,k}$, while the right hand side is the dimension of the perturbed GW-moduli space.  Hence, this follows directly from the definition of the GW-invariants.
\end{proof}

\begin{axiom*}{Homology axiom}
	There exists a homology class
	\[
	\sigma_{A,g,k} \in H_{2c_1(A) + (2n-6)(1-g) + 2k} (Q^k\times \dmspace_{g,k};\Q)
	\]
	such that
	\[
	\GW_{A,g,k} (\alpha_1,\ldots,\alpha_k; \beta) = \langle p_1^* \PD(\alpha_1) \smallsmile \cdots \smallsmile p_k^*\PD (\alpha_k) \smallsmile p_0^*\PD(\beta), \sigma_{A,g,k} \rangle
	\]
	where $p_i: Q^k \times \dmspace_{g,k} \to Q$ denotes the projection onto the $i$th factor and the map $p_0:Q^k \times \dmspace_{g,k}\to\dmspace_{g,k}$ denotes the projection onto the last factor.
\end{axiom*}
\begin{proof}
	The polyfold GW-invariants define homomorphisms 
		\[\GW_{A,g,k}: H_*(Q;\Q)^{\otimes k} \otimes H_* (\dmspace_{g,k}; \Q) \to \Q.\]
	Such a homomorphism defines a cohomology class in $H^*(Q^k \times \dmspace_{g,k}; \Q)$ for $*= \sum_{i=1}^k (\dim_{\R} Q - \deg (\alpha_i)) + (6g-6+2k - \deg(\beta)).$
	The Poincar\'e dual of this cohomology class is the required homology class $\sigma_{A,g,k}$ of codegree $*$.
	
\end{proof}

\begin{axiom*}{Zero axiom}
	If $A=0,\ g=0$ then $\GW_{0,0,k} (\alpha_1,\ldots,\alpha_k;\beta) = 0$ whenever $\deg (\beta) >0$, and
	\[
	\GW_{0,0,k} (\alpha_1,\ldots,\alpha_k; [\pt]) 
	= \int_Q \PD(\alpha_1) \wedge \cdots \wedge \PD(\alpha_k).
	\]
\end{axiom*}
\begin{proof}
	Any map $u :\Sigma \to Q$ with $\delbarj u = 0$ and $u_*[\Sigma ] = 0$ must be constant.  Moreover, at all constant, genus $0$ stable curves the linearization of the Cauchy--Riemann operator is surjective (e.g., see \cite[Lem.~6.7.6]{MSbook}).
	It therefore follows that the unperturbed GW-moduli space $\CM_{0,0,k} (J)$	is transversally cut out.
	Observe that $\CM_{0,0,k} (J) \simeq Q\times \dmexp_{0,k}$; via this identification we may identify the map $ev \times \pi: \CM_{0,0,k}(J) \to Q^k\times \dmlog_{0,k}$ with the map $\Delta \times \id_{\dmspace_{0,k}}: Q\times \dmexp_{0,k} \to Q^k \times \dmlog_{0,k}$ where $\Delta$ is the diagonal $x \mapsto (x,\ldots ,x)$.
	
	We may now write
	\begin{align*}
	\GW_{0,0,k} (\alpha_1,\ldots,\alpha_k; [\pt])
	&= \int_{\CM_{0,0,k} (J) } ev^* (\PD (\alpha_1) \otimes \cdots \otimes \PD(\alpha_k)) \wedge \pi^* \PD (\beta) \\
	&= \int_{Q\times \dmexp_{0,k}} \Delta^* (\PD (\alpha_1) \otimes \cdots \otimes \PD(\alpha_k)) \wedge \text{id}_{\dmspace_{0,k}}^* \PD (\beta) \\
	&= \int_{Q} \PD (\alpha_1) \wedge \cdots \wedge \PD(\alpha_k) \cdot \int_{\dmexp_{0,k}} \PD (\beta).
	\end{align*}
	If $\deg(\beta)>0$ then $\deg (\PD (\beta )) < \dim (\dmexp_{0,k})$ and hence $\int_{\dmexp_{0,k}} \PD (\beta) =0$.
	On the other hand, if $\beta = [\pt]$, then $\int_{\dmexp_{0,k}} \PD [\pt] =1$.
\end{proof}

\begin{axiom*}{Symmetry axiom}
	Fix a permutation $\sigma: \{1,\ldots, k\}\to \{1,\ldots,k\}$.  Consider the permutation map $\sigma:\dmlog_{g,k}\to \dmlog_{g,k}, \ [\Sigma,j,M,D]  \mapsto [\Sigma,j,M^\sigma,D]$ where $M = \{z_1,\ldots,z_k\}$ and where $M^\sigma := \{z'_1,\ldots,z'_k\}$, $z'_i:= z_{\sigma(i)}$.
	Then
	\[
	\GW_{A,g,k} (\alpha_{\sigma(1)},\ldots,\alpha_{\sigma(k)}; \sigma_*\beta) = (-1)^{N(\sigma;\alpha_i)} \GW_{A,g,k} (\alpha_1,\ldots,\alpha_k; \beta)
	\]
	where $N(\sigma;\alpha_i):= \sharp 	\{	i<j \mid \sigma(i)> \sigma(j), \deg (\alpha_i)\deg(\alpha_j)\in 2\Z +1	\}	$.
\end{axiom*}
\begin{proof}
	In essence, the proof follows from the change of variables theorem~\ref{thm:change-of-variables}.
	
	Consider the permutation map between GW-polyfolds, $\sigma: \cZ_{A,g,k} \to \cZ_{A,g,k}$.
	As discussed in \S~\ref{subsec:pulling-back-abstract-perturbations} we may pullback an abstract perturbation via this map, yielding a map between the perturbed GW-moduli spaces
	\[
	\sigma: \mathcal{S}_{A,g,k}(\sigma^* p) \to \mathcal{S}_{A,g,k}(p).
	\]
	
	Consider the following commutative diagram of maps:  
	\[
		\begin{tikzcd}
		&  & Q \\
		\mathcal{S}_{A,g,k}(\sigma^* p) \arrow[r, "\sigma"'] \arrow[d, "\pi"'] \arrow[rru, "ev_{\sigma(i)}"] & \mathcal{S}_{A,g,k}(p) \arrow[d, "\pi'"] \arrow[ru, "ev'_i"'] &  \\
		\dmlog_{g,k} \arrow[r, "\sigma"'] & \dmlog_{g,k}. & 
		\end{tikzcd}
	\]
	
	We may then compute the GW-invariants as follows:
	\begin{align*}
	&\GW_{A,g,k} (\alpha_{\sigma(1)},\ldots,\alpha_{\sigma(k)}; \sigma_*\beta)\\
	&\qquad = \int_{\mathcal{S}_{A,g,k}(p)} ev_1^{'*} \PD (\alpha_{\sigma(1)}) \wedge \cdots \wedge ev_k^{'*} \PD (\alpha_{\sigma(k)}) \wedge \pi'^* \PD(\sigma_*\beta) \\
	&\qquad = \int_{\mathcal{S}_{A,g,k}(\sigma^* p)} \sigma^*\left(ev_1^{'*} \PD (\alpha_{\sigma(1)}) \wedge \cdots \wedge ev_k^{'*} \PD (\alpha_{\sigma(k)}) \wedge \pi'^* \PD(\sigma_*\beta)\right) \\
	&\qquad = \int_{\mathcal{S}_{A,g,k}(\sigma^* p)} ev_{\sigma(1)}^* \PD (\alpha_{\sigma(1)}) \wedge \cdots \wedge ev_{\sigma(k)}^* \PD (\alpha_{\sigma(k)}) \wedge \pi^* \PD(\beta) \\
	&\qquad = (-1)^{N(\sigma;\alpha_i)} \int_{\mathcal{S}_{A,g,k}(\sigma^* p)} ev_{1}^* \PD (\alpha_1)) \wedge \cdots \wedge ev_{k}^* \PD (\alpha_k) \wedge \pi^* \PD(\beta) \\
	&\qquad = (-1)^{N(\sigma;\alpha_i)} \GW_{A,g,k} (\alpha_1,\ldots,\alpha_k; \beta).
	\end{align*}
	In the second equality, we have $\int_{\mathcal{S}_{A,g,k}(p)} \ww = \int_{\mathcal{S}_{A,g,k}(\sigma^*p)} \sigma^* \ww$ by the change of variables theorem~\ref{thm:change-of-variables}.
	In the third equality, we use commutativity of the diagram to see that $ev'_i \circ \sigma = ev_{\sigma(i)}$ hence $\sigma^* ev_i^{'*} = ev^*_{\sigma(i)}$;
	we also see that $\pi'\circ \sigma = \sigma \circ \pi$ hence $\sigma^*\pi'^* =\pi^*\sigma^*$.
	Since the map $\sigma: \dmlog_{g,k} \to \dmlog_{g,k}$ is a diffeomorphism it follows that $\sigma^* \PD (\sigma_* \beta) = \PD (\beta) $ for all $\beta\in H_*(\dmlog_{g,k}; \Q)$, and therefore
		\[
		\sigma^* \pi^{'*} \PD (\sigma_* \beta) = \pi^*\sigma^* \PD(\sigma_* \beta) = \pi^* \PD(\beta).
		\]
	In the final equality, the sign $(-1)^{N(\sigma;\alpha_i)}$ is introduced by permutation of the differential forms.
\end{proof}

\begin{figure}[ht]
	\centering
	\includegraphics[width=\textwidth]{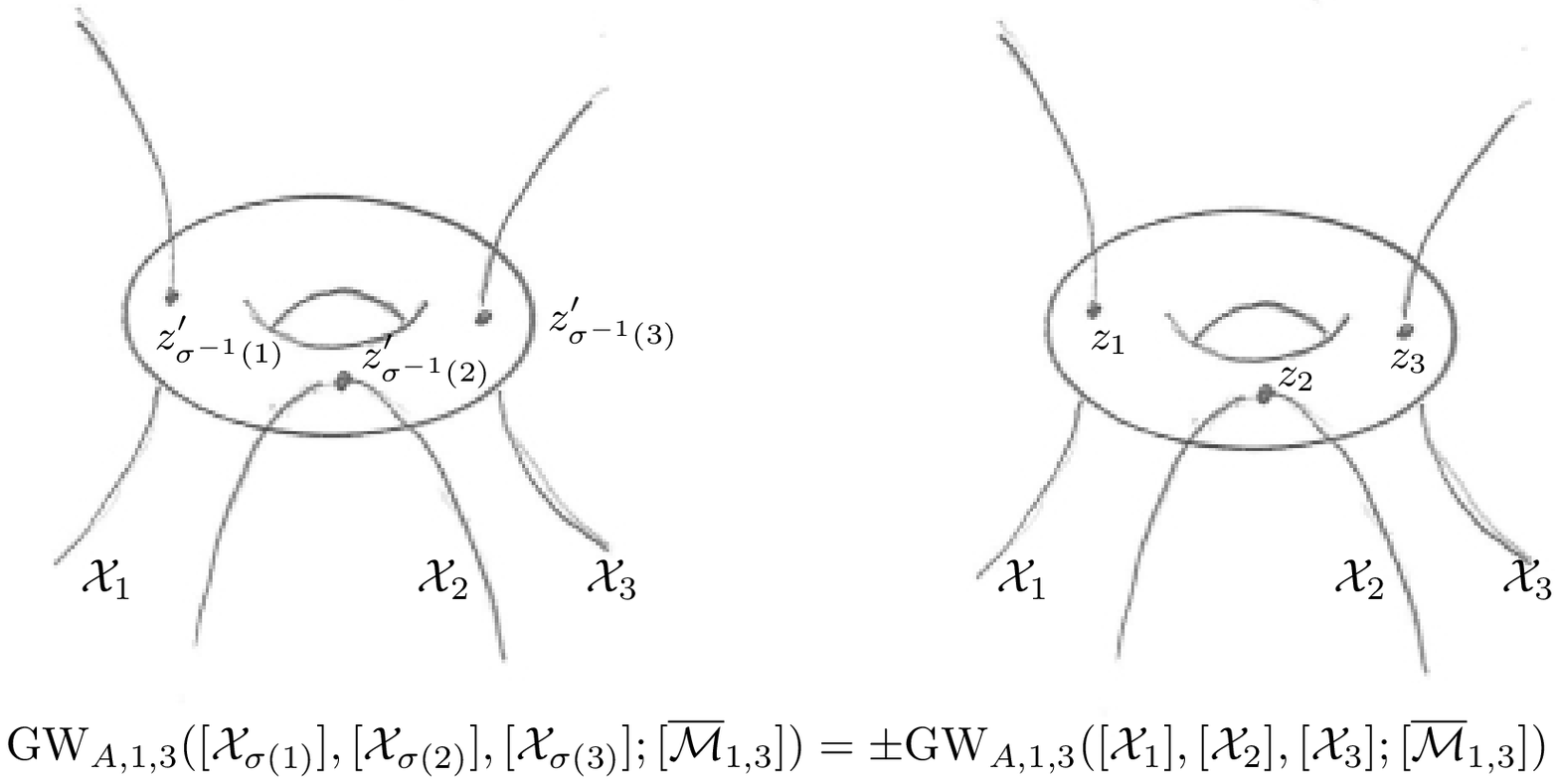}
	\caption{Symmetry axiom}
\end{figure}

Recall from Definition~\ref{def:basic-classes} that $(A,g,k)$ is a basic if it is equal to one of the following: $(A,0,3)$, $(A,1,1)$, or $(A,g\geq 2,0)$.
Again, for such values of $g$ and $k$ we will have $\dmspace_{g,k-1} = \emptyset$ by definition.

\begin{axiom*}{Fundamental class axiom}
	Consider the fundamental classes $[Q]\in H_{2n}(Q;\Q)$ and $[\dmlog_{g,k}] \in H_{6g-6+2k}(\dmlog_{g,k};\Q)$.  Suppose that $A\neq 0$ and that $(A,g,k)$ is not basic.  Then 
	\[
	\GW_{A,g,k} (\alpha_1,\ldots,\alpha_{k-1},[Q]; [\dmlog_{g,k}]) = 0.
	\]

	Consider the canonical section $s_i :\dmlog_{g,k-1} \to \dmlog_{g,k}$ defined by doubling the $i$th-marked point.  Then
	\[
	\GW_{A,g,k} (\alpha_1,\ldots,\alpha_{k-1},[Q]; s_{i*}\beta) = \GW_{A,g,k-1} (\alpha_1,\ldots,\alpha_{k-1};\beta).
	\]
\end{axiom*}
\begin{proof}
	As shown in \S~\ref{subsec:pulling-back-via-ftk} we may pullback perturbations via the well-defined $k$th-marked point forgetting map $ft_k :\cZ^\text{uc}_{A,g,k} \to \cZ_{A,g,k-1}$, yielding a map between perturbed GW-moduli spaces
		\[ft_k : \cS^\text{uc}_{A,g,k} (ft_k^*p)\to \cS_{A,g,k-1}(p).\]
	By construction, the preimage of a point $[\Sigma,j,M,D,u]\in \cS_{A,g,k-1}(p)$ consists of the Riemann surface $\Sigma$ with nodes identified, i.e.,
		\[
		ft_k^{-1} ([\Sigma,j,M,D,u]) \simeq \Sigma/\sim, \quad \text{where } x_a\sim y_a \text{ for nodal pairs } \{x_a,y_a\} \in D.
		\]
	In other words, the moduli space $\cS^\text{uc}_{A,g,k}(ft_k^*p)$ is the universal curve over $\cS_{A,g,k-1}(p)$.
	In particular, it follows that given a perturbed solution in $\cS^\text{uc}_{A,g,k}(ft_k^*)$ we may move the $k$th-marked freely and still have a perturbed solution.
	
	Consider the first assertion, and suppose the converse, i.e., 
		\[\GW_{A,g,k} (\alpha_1,\ldots,\alpha_{k-1},[Q]; [\dmlog_{g,k}]) \neq 0.\]
	Without loss of generality, we may assume that $\cX_i\subset Q$ are submanifold representatives of the homology classes $\alpha_i$.
	As noted in \S~\ref{subsec:intersection-numbers} we may assume that
		$(ev\times \pi) |_{\cS_{A,g,k-1}(p)} \pitchfork (\cX_1 \times \cdots \times \cX_{k-1});$
	since the ``submanifolds'' $Q$ and $\dmexp_{g,k}$ already span their respective components it necessarily follows that
		$(ev\times \pi) |_{\cS^\text{uc}_{A,g,k}(ft_k^*p)} \pitchfork (\cX_1 \times \cdots \times \cX_{k-1} \times Q \times \dmlog_{g,k}).$
	If the polyfold GW-invariant is nonzero, the intersection 
		\[
		(ev \times \pi)(\cS^\text{uc}_{A,g,k}(ft_k^*(p))) \cap (\cX_1\times \cdots \times \cX_{k-1} \times Q\times \dmlog_{g,k}) \neq \emptyset
		\]
	must consist of finitely many isolated points.
	However, the $k$th-marked point is unconstrained and so any intersection point can \textit{never} be isolated.
	This is a contradiction.

	We prove the second assertion.
	Intuitively, the canonical section $s_i$ forces the $k$th-marked point to lie on a component together with the $i$th-marked point; hence the constraint at the $i$th-marked point automatically constrains the $k$th-marked point.
	The map $ft_k: \cS^\text{uc}_{A,g,k} (ft_k^*p)\to \cS_{A,g,k-1}(p)$ then gives a bijection between the points of intersection.
	
	To elaborate, let $\cX_i\subset Q$ be submanifold representatives of the homology classes $\alpha_i$, and let $\cB\subset \dmlog_{g,k-1}$ be a suborbifold representative of the homology class $\beta$.
	The canonical section $s_i$ is a smooth embedding of $\dmlog_{g,k-1}$ into $\dmlog_{g,k}$.  It follows that the homology class $s_{i*} \beta \in H_*(\dmlog_{g,k};\Q)$ is represented by the suborbifold  $s_i(\cB)\subset \dmlog_{g,k}$.
	We may perturb the suborbifold $\cB$ such that $\pi|_{\cS_{A,g,k-1}(p)} \pitchfork \cB$;
	then using the fact that $\cS^\text{uc}_{A,g,k}(ft_k^*(p))$ is the universal curve over $\cS_{A,g,k-1}(p)$ one can see that $\pi|_{\cS_{A,g,k-1}(p)} \pitchfork \cB$ implies that $\pi|_{\cS^\text{uc}_{A,g,k}(ft_k^*p)} \pitchfork s_i (\cB)$.
	We may also perturb the submanifold $\cX_1 \times \cdots \times \cX_{k-1}$ such that 
	$ev|_{\pi^{-1}(\cB)} \pitchfork (\cX_1 \times \cdots \times \cX_{k-1})$ and $ev|_{\pi^{-1}(s_i(\cB))} \pitchfork(\cX_1 \times \cdots \times \cX_{k-1})$.
	
	It therefore follows that we have transversality with the perturbed GW-moduli spaces,
	\begin{itemize}
		\item $(ev \times \pi) |_{\cS_{A,g,k-1}(p)} \pitchfork (\cX_1 \times \cdots \times \cX_{k-1} \times \cB)$,
		\item $(ev \times \pi) |_{\cS^\text{uc}_{A,g,k}(ft_k^*(p))} \pitchfork (\cX_1 \times \cdots \times \cX_{k-1}\times Q \times s_i(\cB))$.
	\end{itemize}
	The map $ft_k$ gives a bijection between the finite intersection points; after taking into account orientation and weights the intersection numbers will be the same, hence
		\[
		\GW_{A,g,k} (\alpha_1,\ldots,\alpha_{k-1},[Q]; s_{i*}\beta) = \GW_{A,g,k-1} (\alpha_1,\ldots,\alpha_{k-1};\beta).
		\]
\end{proof}

\begin{axiom*}{Divisor axiom}
	Suppose $(A,g,k)$ is not basic.
	If $\deg (\alpha_k) = 2n-2$ then
	\[
	\GW_{A,g,k} (\alpha_1,\ldots,\alpha_k; \PD (ft_k^* \PD (\beta))) = (A\cdot \alpha_k ) \ \GW_{A,g,k-1} (\alpha_1,\ldots,\alpha_{k-1};\beta),
	\]
	where $A\cdot \alpha_k$ is given by the homological intersection product.
\end{axiom*}
\begin{proof}
	Without loss of generality, we may assume that $\cX_i\subset Q$ are submanifold representatives of the homology classes $\alpha_i$ and $\cB\subset \dmlog_{g,k-1}$ is a suborbifold representative of the homology class $\beta$.
	The map $ft_k :\dmlog_{g,k} \to \dmlog_{g,k-1}$ is a submersion thus $ft_k^{-1} (\cB) \subset \dmlog_{g,k}$ is also a suborbifold and moreover $[ft_k^{-1} (\cB)]=\PD (ft_k^* \PD (\beta))$ (compare with the reasoning in the genus zero case discussed in \cite[Lem.~7.5.5]{MSbook}).
	
	Without loss of generality we may assume that 
		\[(ev\times \pi)|_{\cS_{A,g,k-1}(p)} \pitchfork (\cX_1 \times \cdots \times \cX_{k-1} \times \cB);\]
	the intersection consists of a finite set of perturbed solutions $[\Sigma,j,M,D,u]\in \cS_{A,g,k-1}(p)$.
	It follows that
		\[
		(ev_1\times \cdots \times ev_{k-1} \times \pi) |_{\cS^\text{uc}_{A,g,k}(ft_k^*p)} \pitchfork (\cX_1 \times \cdots \times \cX_{k-1} \times ft_k^{-1}(\cB))
		\]
	and the preimage $(ev_1\times \cdots \times ev_{k-1} \times \pi)^{-1}(\cX_1 \times \cdots \times \cX_{k-1} \times ft_k^{-1}(\cB))$ consists of the universal curves $\Sigma / \sim$ over the finite intersection points $[\Sigma,j,M,D,u]\in \cS_{A,g,k-1}(p)$.
	The evaluation map at the $k$th-marked point restricted to this preimage may be identified with the stable map considered on the universal curve, i.e.,
		\[u:\Sigma/\sim \ \to Q.\]
		
	Let $\cX_k\subset Q$ be a submanifold representative of the homology class $\alpha_k$; we may perturb $\cX_k$ so that it is transverse to the finitely many maps $u : \Sigma /\sim\  \to Q$.  By the assumption $\deg (\alpha_k) = 2n-2$, the intersection will be a finite set of points.  The intersection number $u|_{\Sigma / \sim} \cdot \cX_k$ is then equal to the homological intersection product $A \cdot \alpha_k$.
	
	We have $(ev\times \pi)|_{\cS^\text{uc}_{A,g,k}(ft_k^*p)} \pitchfork (\cX_1\times \cdots \times \cX_k \times ft_k^{-1}(\cB))$.  The intersection of $\cX_k$ with the image of the universal curve $\Sigma / \sim$ contributes the additional factor $A \cdot \alpha_k$ over an intersection point $[\Sigma,j,M,D,u]\in \cS_{A,g,k-1}(p)$.  Therefore,
	\[
	\GW_{A,g,k} (\alpha_1,\ldots,\alpha_k; \PD (ft_k^* \PD (\beta))) = (A\cdot \alpha_k ) \ \GW_{A,g,k-1} (\alpha_1,\ldots,\alpha_{k-1};\beta).
	\]
\end{proof}

\begin{figure}[ht]
	\centering
	\includegraphics[width=\textwidth]{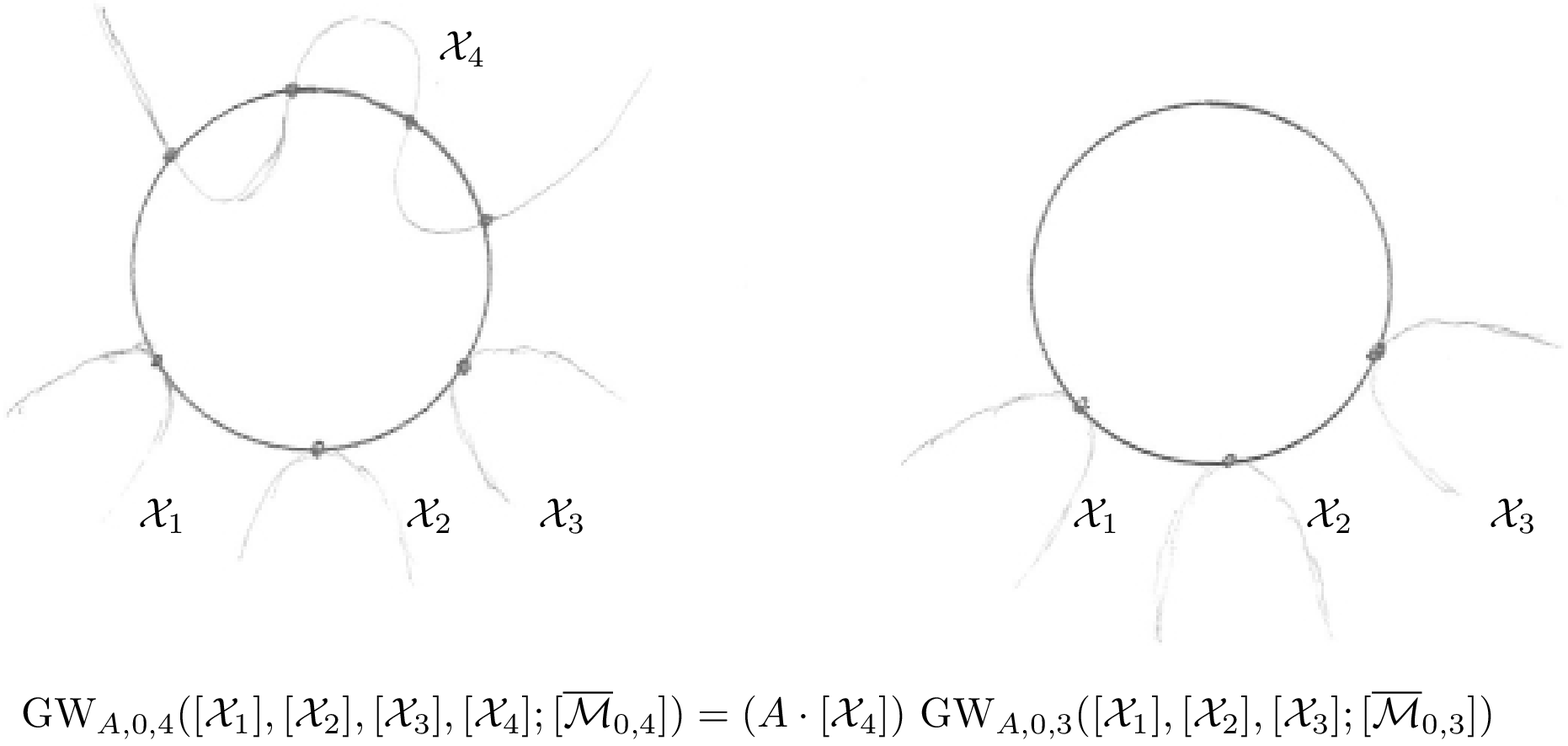}
	\caption{Divisor axiom}
\end{figure}

Recall from the introduction that we can write the Poincar\'e dual of the diagonal $\Delta \subset Q\times Q$ as 
$\PD([\Delta]) = \sum_\nu e_\nu \otimes e^\nu$
for a homogeneous basis $\{e_\nu \} \in H^*(Q;\Q)$ with dual basis with respect to Poincar\'e duality $\{e^\nu \} \in H^*(Q;\Q)$.

\begin{axiom*}{Splitting axiom}
	Fix a partition $S_0 \sqcup S_1 =\{1,\ldots, k\}$.
	Let $k_0 := \sharp S_0,$ $k_1 := \sharp S_1$ and let $g_0,$ $g_1 \geq 0$ such that $g= g_0 + g_1$,
	and $k_i + g_i \geq 2$ for $i=0,1$.
	Consider the natural map 
		\[\phi_S : \dmspace_{k_0+1 , g_0}\times \dmspace_{k_1+1 , g_1} \to \dmspace_{g,k}\]
	which identifies the last marked point of a stable noded Riemann surface in $\dmspace_{k_0+1 , g_0}$ with the first marked point of a stable noded Riemann surface in $\dmspace_{k_1+1, g_1}$, and which maps the first $k_0$ marked points of $\dmspace_{g_0,k_0+1}$ to marked points indexed by $S_0$ and likewise maps the last $k_1$ marked points of $\dmspace_{g_1,k_1+1}$ to marked points indexed by $S_1$.
	Then
	\begin{align*}
	&\GW_{A,g,k} (\alpha_1, \ldots, \alpha_k; \phi_{S*} (\beta_0\otimes \beta_1) ) = 
							(-1)^{N(S;\alpha)} \sum_{A_0+A_1 = A}	\sum_\nu \\
							&\qquad 
							\GW_{A_0,g_0,k_0+1} (\{\alpha_i\}_{i\in S_0}, \PD (e_\nu) ; \beta_0)
							\cdot
							\GW_{A_1,g_1,k_1+1} (\PD (e^\nu), \{\alpha_j\}_{j\in S_1} ; \beta_1)
	\end{align*}
	where $N(S;\alpha)=\sharp \{j<i \mid i\in S_0, j\in S_1, \deg(\alpha_i)\deg(\alpha_j)\in 2\Z +1 \}$.	
\end{axiom*}
\begin{proof}
	Without loss of generality, we may assume that $S_0$ consists of the first $k_0$ points of $\{1,\ldots, k\}$ ordered linearly, and $S_1$ consists of the last $k_1$ points of $\{1,\ldots, k\}$ ordered linearly. When this is the case, $N(S;\alpha_i)=0$.
	The general case reduces to this case through the symmetry axiom by considering the permutation 
		$\sigma: S_0 \sqcup S_1 \to \{1,\ldots, k\}$
	which sends $S_0$ to $\{1,\ldots,k_0\}$ and $S_1$ to $\{k_0+1,\ldots,k_0+k_1\}$ preserving the relative ordering. This moreover explains the presence of the sign correction.

	Observe that $[\Delta] = \sum_\nu \PD(e_\nu) \otimes \PD(e^\nu)$, and therefore interpreting the polyfold GW-invariants as intersection numbers we may write
		\begin{equation}
		\begin{split}
		\label{eq:split-1}
		&\sum_\nu 
		\GW_{A_0,g_0,k_0+1} (\{\alpha_i\}_{i\in S_0}, \PD (e_\nu) ; \beta_0)
		\cdot
		\GW_{A_1,g_1,k_1+1} (\PD (e^\nu), \{\alpha_j\}_{j\in S_1} ; \beta_1) \\
			&\phantom{\sum_\nu}
		= (ev \times ev'\times \pi \times \pi') |_{(\cS_{A_0,g_0,k_0+1} \times \cS_{A_1,g_1,k_1+1})(p)} \\
			&\phantom{\sum_\nu = }
		\cdot ( \{\cX_i\}_{i\in S_0}\times \Delta \times \{\cX_j\}_{j\in S_1} \times \cB_0 \times \cB_1).
		\end{split}
		\end{equation}
	We consider the following inclusion and marked point identifying maps as described in \S~\ref{subsec:inclusion-marked-point-identifying-maps} defined on the split GW-polyfolds of \S~\ref{subsec:other-gw-polyfolds}:
		\[
		\begin{tikzcd}
		\displaystyle{\bigsqcup_{A_0+A_1 = A}} \cZ_{A_0,g_0,k_0+1} \times \cZ_{A_1,g_1,k_1+1}	&	\\
		\displaystyle{\bigsqcup_{A_0+A_1 = A}} \cZ_{A_0+A_1,S}	\arrow[u, hook, "i"] \arrow[r, "\phi_S"] & \cZ_{A,g,k}
		\end{tikzcd}
		\]
	We may pullback a perturbation via the inclusion $i: \cZ_{A_0+A_1,S} \hookrightarrow \cZ_{A_0,g_0,k_0+1} \times \cZ_{A_1,g_1,k_1+1}$; since the evaluation map is a submersion we may also assume this perturbation is chosen such that
		\[ev_{k_0+1}\times ev'_{1} : (\cS_{A_0,g_0, k_0+1} \times \cS_{A_1,g_1,k_1+1}) (p) \to Q\times Q\]
	is transverse to the diagonal $\Delta\subset Q\times Q$.
	
	The map $i: \cS_{A_0+A_1,S} (i^*p) \hookrightarrow (\cS_{A_0,g_0, k_0+1} \times \cS_{A_1,g_1,k_1+1}) (p)$ gives an identification of $\cS_{A_0+A_1,S} (i^*p)$ with $(ev_{k_0+1}\times ev'_{1})^{-1} (\Delta)$.
	After perturbation of the remaining representing submanifolds/suborbifolds we assert that the following intersection numbers are equal:
		\begin{equation}\label{eq:split-2}
		\begin{split}
		& (ev \times ev'\times \pi \times \pi') |_{(\cS_{A_0,g_0,k_0+1} \times \cS_{A_1,g_1,k_1+1})(p)} \\
		& \qquad \cdot ( \{\cX_i\}_{i\in S_0}\times \Delta \times \{\cX_j\}_{j\in S_1} \times \cB_0 \times \cB_1) \\
		& \qquad = (ev\times \pi)|_{\cS_{A_0+A_1,S} (i^*p)}
				\cdot ( \{\cX_i\}_{i\in S_0}\times\{\cX_j\}_{j\in S_1} \times \cB_0 \times \cB_1)
		\end{split}
		\end{equation}
	
	On the other hand, we may also pullback a perturbation via the marked point identifying maps $\sqcup \phi_S: \bigsqcup_{A_0+A_1 = A} \cZ_{A_0+A_1,S} \rightarrow \cZ_{A,g,k}$.
	We therefore consider the following commutative diagram:
		\[
		\begin{tikzcd}
		&  & Q^k \\
		\displaystyle{\bigsqcup_{A_0+A_1 = A}} \cS_{A_0+A_1,S}(\phi_S^* q) \arrow[r, "\sqcup \phi_S"'] \arrow[d, "\pi"'] \arrow[rru, "ev"] & \cS_{A,g,k}(q) \arrow[d, "\pi"] \arrow[ru, "ev"'] &  \\
		\dmlog_{g_0,k_0+1}\times \dmlog_{g_1,k_1+1} \arrow[r, "\phi_S"'] & \dmlog_{g,k}. & 
		\end{tikzcd}
		\]
	We may now perturb the representing submanifolds $\cX_i$ and representing suborbifolds $\cB_0$, $\cB_1$ such that
		$(ev\times \pi) |_{\cS_{A_0+A_1,S}(\phi_S^* q)} \pitchfork ( \{\cX_i\}_{i\in S_0}\times\{\cX_j\}_{j\in S_1} \times \cB_0 \times \cB_1)$.
	The map $\sqcup \phi_S$ gives a restriction to the finite points of intersection
		\[\bigsqcup_{A_0+A_1 = A} (ev\times \pi)^{-1} ( \{\cX_i\}_{i\in S_0}\times\{\cX_j\}_{j\in S_1} \times \cB_0 \times \cB_1) \subset \bigsqcup_{A_0+A_1 = A}\cS_{A_0+A_1,S}(\phi_S^* q)\] and
		\[(ev\times \pi)^{-1} (\cX_1 \times \cdots \times \cX_k \times \phi_s(\cB_0\times \cB_1)) \subset \cS_{A,g,k}(q).\]
	We would like to claim that this map is a bijection. 
	The concern is that different decompositions $A_0+A_1 = A$, $A'_0+ A'_1 = A$ may share a common decomposition, resulting in a multiple points mapping to the same point.
	To prevent this, the representing submanifolds and representing suborbifolds should be perturbed to be transverse to all further decompositions of $A$; for dimension reasons, the intersection points with a further decomposition will then be empty.

	We must make note of an important subtlety. Due to the fact that $\pi$ is not a submersion, we can never achieve transversality of the map $ev\times \pi : \cS_{A,g,k}(q) \to Q^k \times \dmlog_{g,k}$ with the suborbifold $\cX_1 \times \cdots \times \cX_k \times \phi_S(\cB_0\times \cB_1)$ through choice of perturbation $q$. 
	(Of course, we can find a representative in the same homology class of the suborbifold $\phi_S(\cB_0\times \cB_1)$ such that we \textit{will} have transversality, but then we lose the relationship between $\cB_0\times \cB_1$ and $\phi_S(\cB_0\times \cB_1)$.)
	
	However, even though the intersection is not transverse, the intersection does consist of a finite set of points.
	Comparing the branched integrals of the Poincar\'e duals at these finite points of intersection we obtain the equality
		\begin{equation}
		\label{eq:split-3}
		\begin{split}
		&\sum_{A_0+A_1 = A} \int_{\cS_{A_0+A_1,S}(\phi_S^* q)} ev^* \{\PD[\cX_i]\}_{i\in S_0} \wedge ev^* \{\PD[\cX_j]\}_{j\in S_1} \wedge \pi^* \PD [\cB_0 \times \cB_1]\\
		&\qquad = \int_{\cS_{A,g,k}(q)} ev_1^* \PD[\cX_1] \wedge \cdots \wedge ev_k^* \PD[\cX_k] \wedge \pi^* \PD [\phi_S(\cB_0\times \cB_1)].
		\end{split}
		\end{equation}
	Combining equations \eqref{eq:split-1}, \eqref{eq:split-2}, and \eqref{eq:split-3} we obtain
	\begin{align*}
	&\sum_{A_0+A_1 = A}	\sum_\nu
	\GW_{A_0,g_0,k_0+1} (\{\alpha_i\}_{i\in S_0}, \PD (e_\nu) ; \beta_0)
	\cdot
	\GW_{A_1,g_1,k_1+1} (\PD (e^\nu), \{\alpha_j\}_{j\in S_1} ; \beta_1) \\
	&\qquad= \sum_{A_0+A_1 = A} 
			\left(
			\begin{array}{l}
			(ev \times ev'\times \pi \times \pi') |_{(\cS_{A_0,g_0,k_0+1} \times \cS_{A_1,g_1,k_1+1})(p)} \\
			\qquad\cdot ( \{\cX_i\}_{i\in S_0}\times \Delta \times \{\cX_j\}_{j\in S_1} \times \cB_0 \times \cB_1)
			\end{array}
			\right)
			\\
	&\qquad= \sum_{A_0+A_1 = A} (ev\times \pi)|_{\cS_{A_0+A_1,S} (i^*p)} \cdot ( \{\cX_i\}_{i\in S_0}\times\{\cX_j\}_{j\in S_1} \times \cB_0 \times \cB_1) \\
	&\qquad= \sum_{A_0+A_1 = A} \int_{\cS_{A_0+A_1,S}(\phi_S^* q)} 
		\left(
		\begin{array}{l}
		ev^* \{\PD[\cX_i]\}_{i\in S_0} \wedge ev^* \{\PD[\cX_j]\}_{j\in S_1} \\
		\qquad \wedge \pi^* \PD [\cB_0 \times \cB_1]		
		\end{array}
		\right)
		\\
	&\qquad= \int_{\cS_{A,g,k}(q)} ev_1^* \PD[\cX_1] \wedge \cdots \wedge ev_k^* \PD[\cX_k] \wedge \pi^* \PD [\phi_S(\cB_0\times \cB_1)]\\
	&\qquad= \GW_{A,g,k} (\alpha_1, \ldots, \alpha_k; \phi_{S*} (\beta_0\otimes \beta_1) )
	\end{align*}
	
\end{proof}

\begin{axiom*}{Genus reduction axiom}
	Consider the natural map 
		\[\psi: \dmspace_{g-1,k+2} \to \dmspace_{g,k}\]
	which identifies the last two marked points of a stable noded Riemann surface, increasing the arithmetic genus by one.
	Then
		\[
		2 \cdot \GW_{A,g,k} (\alpha_1, \ldots, \alpha_k; \psi_* \beta) = \sum_\nu \GW_{A,g-1,k+2} (\alpha_1,\ldots,\alpha_k, \PD (e_\nu) , \PD (e^\nu) ; \beta).
		\]
\end{axiom*}
\begin{proof}
	Again, $[\Delta]= \sum_\nu \PD(e_\nu) \otimes \PD(e^\nu)$, and therefore interpreting the polyfold GW-invariants as intersection numbers we may write
		\begin{equation}\label{eq:gr-1}
		\begin{split}
		&\sum_\nu \GW_{A,g-1,k+2} (\alpha_1,\ldots,\alpha_k, \PD (e_\nu) , \PD (e^\nu) ; \beta)\\
		&\qquad = (ev_1 \times \cdots \times ev_{k+2}\times \pi) |_{\cS_{A,g-1,k+2}(p)} \cdot (\cX_1\times \cdots \times \cX_k \times \Delta \times \cB).
		\end{split}
		\end{equation}
	We consider the following inclusion maps and marked point identifying maps as described in \S~\ref{subsec:inclusion-marked-point-identifying-maps} defined on the genus GW-polyfolds of \S~\ref{subsec:other-gw-polyfolds}:
	\[
	\begin{tikzcd}
	\cZ_{A,g-1,k+2} 		&	\\
	\cZ^\text{g}_{A,g-1,k+2}	\arrow[u, hook, "i"] \arrow[r, "\psi"] & \cZ_{A,g,k}
	\end{tikzcd}
	\]
	We may pullback a perturbation via the inclusion $i: \cZ^\text{g}_{A,g-1,k+2} \hookrightarrow \cZ_{A,g-1,k+2}$; we may also assume this perturbation is chosen such that
		\[ev_{k+1}\times ev_{k+2} : \cS_{A,g-1, k+2}(p) \to Q\times Q\]
	is transverse to the diagonal $\Delta\subset Q\times Q$.
	Consider the following commutative diagram:
		\[
		\begin{tikzcd}
		&  & Q^{k} \\
		\cS^\text{g}_{A,g-1,k+2}(i^* p) \arrow[r, hook, "i"'] \arrow[d, "\pi"'] \arrow[rru, "ev"] & \cS_{A,g-1,k+2}(p) \arrow[d, "\pi"] \arrow[ru, "ev"'] &  \\
		\dmlog_{g-1,k+2} \arrow[r, "\id"'] & \dmlog_{g-1,k+2}. & 
		\end{tikzcd}
		\]
	Observe that the map $i: \cS^\text{g}_{A,g-1,k+2}(i^* p) \hookrightarrow \cS_{A,g-1,k+2}(p)$ gives an identification of $\cS^\text{g}_{A,g-1,k+2}(i^*p)$ with $(ev_{k+1}\times ev_{k+2})^{-1} (\Delta)$.	
	The associated intersection numbers are equal:
		\begin{equation}
		\begin{split}
		\label{eq:gr-2}
		&(ev_1 \times \cdots \times ev_{k+2}\times \pi) |_{\cS_{A,g-1,k+2}(p)} \cdot (\cX_1\times \cdots \times \cX_k \times \Delta \times \cB)\\
		&\qquad =(ev_1 \times \cdots \times ev_k \times \pi)|_{\cS^\text{g}_{A,g-1,k+2}(i^*p)} \cdot (\cX_1\times \cdots \times \cX_k \times \cB).
		\end{split}
		\end{equation}
	
	On the other hand, we may also pullback a perturbation via the marked point identifying map $\psi: \cZ^\text{g}_{A,g-1,k+2} \to \cZ_{A,g,k}$.
	We therefore consider the following commutative diagram:
		\[
		\begin{tikzcd}
		&  & Q^{k} \\
		\cS^\text{g}_{A,g-1,k+2}(\psi^* q) \arrow[r, "\psi"'] \arrow[d, "\pi"'] \arrow[rru, "ev"] & \cS_{A,g,k}(q) \arrow[d, "\pi"] \arrow[ru, "ev"'] &  \\
		\dmlog_{g-1,k+2} \arrow[r, "\psi"'] & \dmlog_{g,k}. & 
		\end{tikzcd}
		\]
	We may now perturb the representing submanifolds $\cX_i$ and the representing suborbifold $\cB$ such that
	$(ev\times \pi) |_{\cS^\text{g}_{A,g-1,k+2}(\psi^* q)} \pitchfork (\cX_1 \times \cdots \times \cX_k \times \cB)$.
	The marked point identifying map $\psi$ gives a two-to-one map between the finite points of intersection
		\[(ev\times \pi)^{-1} (\cX_1 \times \cdots \times \cX_k \times \cB) \subset \cS^\text{g}_{A,g-1,k+2}(\psi^* q)\]
	and
		\[(ev\times \pi)^{-1} (\cX_1 \times \cdots \times \cX_k \times \psi(\cB)) \subset \cS_{A,g,k}(q).\]
	It is two-to-one since by exchanging the marked points $z_{k+1}$ and $z_{k+2}$ we obtain distinct stable curve solutions which map to the same stable curve after identifying these marked points, compare with the discussion in \S~\ref{subsec:inclusion-marked-point-identifying-maps}.
	
	We again must make note of an important subtlety. Due to the fact that $\pi$ is not a submersion, we can never achieve transversality of the map $ev\times \pi : \cS_{A,g,k}(q) \to Q^k \times \dmlog_{g,k}$ with the suborbifold $\cX_1 \times \cdots \times \cX_k \times \psi(\cB)$ through choice of perturbation $q$.
	(Again, we can find a representative in the same homology class of the suborbifold $\psi(\cB)$ such that we \textit{will} have transversality, but then we lose the relationship between $\cB$ and $\psi(\cB)$.)
	
	However, even though the intersection is not transverse, the intersection does consist of a finite set of points.
	Comparing the branched integrals of the Poincar\'e duals at these finite points of intersection we obtain the equality
		\begin{equation}
		\begin{split}
		\label{eq:gr-3}
		&\int_{\cS^\text{g}_{A,g-1,k+2}(\psi^* q)} ev_1^* \PD[\cX_1] \wedge \cdots \wedge ev_k^* \PD[\cX_k] \wedge \pi^* \PD [\cB]
		\\
		&\qquad= 2 \cdot \int_{\cS_{A,g,k}(q)} ev_1^* \PD[\cX_1] \wedge \cdots \wedge ev_k^* \PD[\cX_k] \wedge \pi^* \PD [\psi(\cB)].
		\end{split}
		\end{equation}
	Combining equations \eqref{eq:gr-1}, \eqref{eq:gr-2}, and \eqref{eq:gr-3} we obtain
		\begin{align*}
		&\sum_\nu \GW_{A,g-1,k+2} (\alpha_1,\ldots,\alpha_k, \PD (e_\nu) , \PD (e^\nu) ; \beta)\\
		&\qquad =(ev_1 \times \cdots \times ev_{k+2}\times \pi) |_{\cS_{A,g-1,k+2}(p)} \cdot (\cX_1\times \cdots \times \cX_k \times \Delta \times \cB)\\
		&\qquad =(ev_1 \times \cdots \times ev_k \times \pi)|_{\cS^\text{g}_{A,g-1,k+2}(i^*p)} \cdot (\cX_1\times \cdots \times \cX_k \times \cB)\\
		&\qquad =\int_{\cS^\text{g}_{A,g-1,k+2}(\psi^* q)} ev_1^* \PD[\cX_1] \wedge \cdots \wedge ev_k^* \PD[\cX_k] \wedge \pi^* \PD [\cB]\\
		&\qquad =2\cdot \int_{\cS_{A,g,k}(q)} ev_1^* \PD[\cX_1] \wedge \cdots \wedge ev_k^* \PD[\cX_k] \wedge \pi^* \PD [\psi(\cB)]\\
		&\qquad =2\cdot \GW_{A,g,k}(\alpha_1,\ldots,\alpha_k; \psi_*\beta).	
		\end{align*}
\end{proof}




\begin{bibdiv}
\begin{biblist}
	\bib{bott2013differential}{book}{
		AUTHOR = {Bott, Raoul},
		AUTHOR = {Tu, Loring W.},
		TITLE = {Differential forms in algebraic topology},
		SERIES = {Graduate Texts in Mathematics},
		VOLUME = {82},
		PUBLISHER = {Springer-Verlag, New York-Berlin},
		YEAR = {1982},
		PAGES = {xiv+331},
		ISBN = {0-387-90613-4},
	}

	\bib{castellano2016genus}{thesis}{
		AUTHOR = {Castellano, Robert},
		TITLE = {Kuranishi atlases and genus zero {G}romov--{W}itten invariants},
		NOTE = {Thesis (Ph.D.)--Columbia University},
		PUBLISHER = {ProQuest LLC, Ann Arbor, MI},
		YEAR = {2016},
		PAGES = {119},
		ISBN = {978-1339-60952-2},
		URL =
		{http://gateway.proquest.com/openurl?url_ver=Z39.88-2004&rft_val_fmt=info:ofi/fmt:kev:mtx:dissertation&res_dat=xri:pqm&rft_dat=xri:pqdiss:10096865},
	}

	\bib{cieliebak2007symplectic}{article}{
		AUTHOR = {Cieliebak, Kai},
		AUTHOR = {Mohnke, Klaus},
		TITLE = {Symplectic hypersurfaces and transversality in
			{G}romov--{W}itten theory},
		JOURNAL = {J. Symplectic Geom.},
		VOLUME = {5},
		YEAR = {2007},
		NUMBER = {3},
		PAGES = {281--356},
		ISSN = {1527-5256},
		URL = {http://projecteuclid.org/euclid.jsg/1210083200},
	}

	\bib{ffgw2016polyfoldsfirstandsecondlook}{article}{
		AUTHOR = {Fabert, Oliver},
		AUTHOR = {Fish, Joel W.},
		AUTHOR = {Golovko, Roman},
		AUTHOR = {Wehrheim, Katrin},
		TITLE = {Polyfolds: A first and second look},
		JOURNAL = {EMS Surv. Math. Sci.},
		VOLUME = {3},
		YEAR = {2016},
		NUMBER = {2},
		PAGES = {131--208},
		ISSN = {2308-2151},
		DOI = {10.4171/EMSS/16},
		URL = {https://doi.org/10.4171/EMSS/16},
	}

	\bib{filippenko2018constrained}{article}{
		author = {{Filippenko}, Benjamin},
		title = {Polyfold regularization of constrained moduli spaces},
		journal = {arXiv e-prints},
		year = {2018},
		pages = {89},
		eprint = {arXiv:1807.00386},
	}

	\bib{floer1988instanton}{article}{
		AUTHOR = {Floer, Andreas},
		TITLE = {An instanton-invariant for {$3$}-manifolds},
		JOURNAL = {Comm. Math. Phys.},
		VOLUME = {118},
		YEAR = {1988},
		NUMBER = {2},
		PAGES = {215--240},
		ISSN = {0010-3616},
		URL = {http://projecteuclid.org/euclid.cmp/1104161987},
	}

	\bib{fukaya2012technical}{article}{
		AUTHOR = {{Fukaya}, Kenji},
		AUTHOR = {{Oh}, Yong-Geun},
		AUTHOR = {{Ohta}, Hiroshi},
		AUTHOR = {{Ono}, Kaoru},
		title = {Technical details on Kuranishi structure and virtual fundamental chain},
		journal = {arXiv e-prints},
		year = {2012},
		pages = {257},
		eprint = {arXiv:1209.4410},
	}

	\bib{FO}{article}{
		AUTHOR = {Fukaya, Kenji},
		AUTHOR = {Ono, Kaoru},
		TITLE = {Arnold conjecture and {G}romov--{W}itten invariant for general
			symplectic manifolds},
		BOOKTITLE = {The {A}rnoldfest ({T}oronto, {ON}, 1997)},
		SERIES = {Fields Inst. Commun.},
		VOLUME = {24},
		PAGES = {173--190},
		PUBLISHER = {Amer. Math. Soc., Providence, RI},
		YEAR = {1999},
	}

	\bib{G}{article}{
		AUTHOR = {Gromov, M.},
		TITLE = {Pseudo holomorphic curves in symplectic manifolds},
		JOURNAL = {Invent. Math.},
		VOLUME = {82},
		YEAR = {1985},
		NUMBER = {2},
		PAGES = {307--347},
		ISSN = {0020-9910},
		DOI = {10.1007/BF01388806},
		URL = {https://doi.org/10.1007/BF01388806},
	}

	\bib{HWZ1}{article}{
		author={Hofer, H.},
		author={Wysocki, K.},
		author={Zehnder, E.},
		title={A general Fredholm theory. I. A splicing-based differential geometry},
		journal={J. Eur. Math. Soc. (JEMS)},
		volume={9},
		date={2007},
		number={4},
		pages={841--876},
		issn={1435-9855},
		review={\MR{2341834}},
		doi={10.4171/JEMS/99},
	}
	
	\bib{HWZ2}{article}{
		author={Hofer, H.},
		author={Wysocki, K.},
		author={Zehnder, E.},
		title={A general Fredholm theory. II. Implicit function theorems},
		journal={Geom. Funct. Anal.},
		volume={19},
		date={2009},
		number={1},
		pages={206--293},
		issn={1016-443X},
		review={\MR{2507223}},
		doi={10.1007/s00039-009-0715-x},
	}
	
	\bib{HWZ3}{article}{
		author={Hofer, H.},
		author={Wysocki, K.},
		author={Zehnder, E.},
		title={A general Fredholm theory. III. Fredholm functors and polyfolds},
		journal={Geom. Topol.},
		volume={13},
		date={2009},
		number={4},
		pages={2279--2387},
		issn={1465-3060},
		review={\MR{2515707}},
		doi={10.2140/gt.2009.13.2279},
	}
	
	\bib{HWZint}{article}{
		author={Hofer, H.},
		author={Wysocki, K.},
		author={Zehnder, E.},
		title={Integration theory on the zero sets of polyfold Fredholm sections},
		journal={Math. Ann.},
		volume={346},
		date={2010},
		number={1},
		pages={139--198},
		issn={0025-5831},
		review={\MR{2558891}},
		doi={10.1007/s00208-009-0393-x},
	}
	
	\bib{HWZsc}{article}{	
		author={Hofer, H.},
		author={Wysocki, K.},
		author={Zehnder, E.},
		title={sc-smoothness, retractions and new models for smooth spaces},
		journal={Discrete Contin. Dyn. Syst.},
		volume={28},
		date={2010},
		number={2},
		pages={665--788},
		issn={1078-0947},
		review={\MR{2644764}},
		doi={10.3934/dcds.2010.28.665},
	}
	
	\bib{HWZGW}{article}{
		author={Hofer, H.},
		author={Wysocki, K.},
		author={Zehnder, E.},
		title={Applications of polyfold theory I: The polyfolds of {G}romov--{W}itten theory},
		journal={Mem. Amer. Math. Soc.},
		volume={248},
		date={2017},
		number={1179},
		pages={v+218},
		issn={0065-9266},
		isbn={978-1-4704-2203-5},
		isbn={978-1-4704-4060-2},
		review={\MR{3683060}},
		doi={10.1090/memo/1179},
	}
	
	\bib{HWZbook}{article}{	
		author={Hofer, H.},
		author={Wysocki, K.},
		author={Zehnder, E.},
		title = {Polyfold and {F}redholm theory},
		journal = {arXiv e-prints},
		year = {2017},
		pages = {714},
		eprint = {arXiv:1707.08941},
	}

	\bib{HWZdm}{misc}{
		author={Hofer, H.},
		author={Wysocki, K.},
		author={Zehnder, E.},
		TITLE = {Deligne--{M}umford-type spaces with a view towards symplectic field theory},
		status = {Lecture notes in preparation},
	}
	
	\bib{ionel2013natural}{article}{
		author = {{Ionel}, Eleny-Nicoleta},
		author = {{Parker}, Thomas H.},
		title = {A natural {G}romov--{W}itten virtual fundamental class},
		journal = {arXiv e-prints},
		year = {2013},
		pages = {35},
		eprint = {arXiv:1302.3472},
	}

	
	\bib{KM}{article}{
		AUTHOR = {Kontsevich, M.},
		AUTHOR = {Manin, Yu.},
		TITLE = {{G}romov--{W}itten classes, quantum cohomology, and enumerative geometry},
		JOURNAL = {Comm. Math. Phys.},
		VOLUME = {164},
		YEAR = {1994},
		NUMBER = {3},
		PAGES = {525--562},
		ISSN = {0010-3616},
		URL = {http://projecteuclid.org/euclid.cmp/1104270948},
	}

	\bib{Kstable}{article}{
		AUTHOR = {Kontsevich, Maxim},
		TITLE = {Enumeration of rational curves via torus actions},
		BOOKTITLE = {The moduli space of curves ({T}exel {I}sland, 1994)},
		SERIES = {Progr. Math.},
		VOLUME = {129},
		PAGES = {335--368},
		PUBLISHER = {Birkh\"{a}user Boston, Boston, MA},
		YEAR = {1995},
	}

	\bib{LT}{article}{
		AUTHOR = {Li, Jun},
		AUTHOR = {Tian, Gang},
		TITLE = {Virtual moduli cycles and {G}romov--{W}itten invariants of
			general symplectic manifolds},
		BOOKTITLE = {Topics in symplectic {$4$}-manifolds ({I}rvine, {CA}, 1996)},
		SERIES = {First Int. Press Lect. Ser., I},
		PAGES = {47--83},
		PUBLISHER = {Int. Press, Cambridge, MA},
		YEAR = {1998},
	}

	\bib{mcduff1991symplectic}{article}{
		AUTHOR = {McDuff, Dusa},
		TITLE = {Symplectic manifolds with contact type boundaries},
		JOURNAL = {Invent. Math.},
		VOLUME = {103},
		YEAR = {1991},
		NUMBER = {3},
		PAGES = {651--671},
		ISSN = {0020-9910},
		DOI = {10.1007/BF01239530},
		URL = {https://doi.org/10.1007/BF01239530},
	}
	
	\bib{MSbook}{book}{
		AUTHOR = {McDuff, Dusa},
		AUTHOR = {Salamon, Dietmar},
		TITLE = {{$J$}-holomorphic curves and symplectic topology},
		SERIES = {American Mathematical Society Colloquium Publications},
		VOLUME = {52},
		EDITION = {Second},
		PUBLISHER = {American Mathematical Society, Providence, RI},
		YEAR = {2012},
		PAGES = {xiv+726},
		ISBN = {978-0-8218-8746-2},
	}

	\bib{mcduff2012smooth}{article}{
		author = {{McDuff}, Dusa},
		author = {{Wehrheim}, Katrin},
		title = {Kuranishi atlases with trivial isotropy - the 2013 state of affairs},
		journal = {arXiv e-prints},
		year = {2012},
		pages = {170},
		eprint = {arXiv:1208.1340},
	}

	\bib{MWtopology}{article}{
		AUTHOR = {McDuff, Dusa},
		AUTHOR = {Wehrheim, Katrin},
		TITLE = {The topology of {K}uranishi atlases},
		JOURNAL = {Proc. Lond. Math. Soc. (3)},
		VOLUME = {115},
		YEAR = {2017},
		NUMBER = {2},
		PAGES = {221--292},
		ISSN = {0024-6115},
		DOI = {10.1112/plms.12032},
		URL = {https://doi.org/10.1112/plms.12032},
	}

	\bib{mcduff2018fundamental}{article}{
		AUTHOR = {McDuff, Dusa},
		AUTHOR = {Wehrheim, Katrin},
		TITLE = {The fundamental class of smooth {K}uranishi atlases with
			trivial isotropy},
		JOURNAL = {J. Topol. Anal.},
		VOLUME = {10},
		YEAR = {2018},
		NUMBER = {1},
		PAGES = {71--243},
		ISSN = {1793-5253},
		DOI = {10.1142/S1793525318500048},
		URL = {https://doi.org/10.1142/S1793525318500048},
	}

	\bib{pardon2016algebraic}{article}{
		AUTHOR = {Pardon, John},
		TITLE = {An algebraic approach to virtual fundamental cycles on moduli
			spaces of pseudo-holomorphic curves},
		JOURNAL = {Geom. Topol.},
		VOLUME = {20},
		YEAR = {2016},
		NUMBER = {2},
		PAGES = {779--1034},
		ISSN = {1465-3060},
		DOI = {10.2140/gt.2016.20.779},
		URL = {https://doi.org/10.2140/gt.2016.20.779},
	}

	\bib{robbinsalamon2006delignemumford}{article}{
		AUTHOR = {Robbin, Joel W.},
		AUTHOR = {Salamon, Dietmar A.},
		TITLE = {A construction of the {D}eligne-{M}umford orbifold},
		JOURNAL = {J. Eur. Math. Soc. (JEMS)},
		VOLUME = {8},
		YEAR = {2006},
		NUMBER = {4},
		PAGES = {611--699},
		ISSN = {1435-9855},
		DOI = {10.4171/JEMS/101},
		URL = {https://doi.org/10.4171/JEMS/101},
	}

	\bib{ruan1994symplectic}{article}{
		AUTHOR = {Ruan, Yongbin},
		TITLE = {Symplectic topology on algebraic {$3$}-folds},
		JOURNAL = {J. Differential Geom.},
		VOLUME = {39},
		YEAR = {1994},
		NUMBER = {1},
		PAGES = {215--227},
		ISSN = {0022-040X},
		URL = {http://projecteuclid.org/euclid.jdg/1214454682},
	}

	
	\bib{ruan1996topological}{article}{
		AUTHOR = {Ruan, Yongbin},
		TITLE = {Topological sigma model and {D}onaldson-type invariants in
			{G}romov theory},
		JOURNAL = {Duke Math. J.},
		VOLUME = {83},
		YEAR = {1996},
		NUMBER = {2},
		PAGES = {461--500},
		ISSN = {0012-7094},
		DOI = {10.1215/S0012-7094-96-08316-7},
		URL = {https://doi.org/10.1215/S0012-7094-96-08316-7},
	}

	\bib{rt1995quatumcohomology}{article}{
		AUTHOR = {Ruan, Yongbin},
		AUTHOR = {Tian, Gang},
		TITLE = {A mathematical theory of quantum cohomology},
		JOURNAL = {J. Differential Geom.},
		VOLUME = {42},
		YEAR = {1995},
		NUMBER = {2},
		PAGES = {259--367},
		ISSN = {0022-040X},
		URL = {http://projecteuclid.org/euclid.jdg/1214457234},
	}

	\bib{schmaltz2019steenrod}{article}{
		author = {{Schmaltz}, Wolfgang},
		title = {The Steenrod problem for orbifolds and polyfold invariants as intersection numbers},
		journal = {arXiv e-prints},
		year = {2019},
		pages = {32},
		eprint = {arXiv:1904.02186},
	}

	\bib{schmaltz2019naturality}{article}{
		author	= {{Schmaltz}, Wolfgang},
		title	= {Naturality of polyfold invariants and pulling back abstract perturbations},
		year	= {2019},
		status	= {in preparation}
	}

	\bib{Si}{article}{
		author = {{Siebert}, Bernd},
		title = {{G}romov--{W}itten invariants of general symplectic manifolds},
		journal = {eprint arXiv:dg-ga/960800},
		year = {1996},
		pages = {dg-ga/9608005},
		eprint = {arXiv:dg-ga/9608005},
	}

	\bib{thom1954quelques}{article}{
		AUTHOR = {Thom, Ren\'{e}},
		TITLE = {Quelques propri\'{e}t\'{e}s globales des vari\'{e}t\'{e}s diff\'{e}rentiables},
		JOURNAL = {Comment. Math. Helv.},
		VOLUME = {28},
		YEAR = {1954},
		PAGES = {17--86},
		ISSN = {0010-2571},
		DOI = {10.1007/BF02566923},
		URL = {https://doi.org/10.1007/BF02566923},
	}

	\bib{witten1988topological}{article}{
		AUTHOR = {Witten, Edward},
		TITLE = {Topological sigma models},
		JOURNAL = {Comm. Math. Phys.},
		VOLUME = {118},
		YEAR = {1988},
		NUMBER = {3},
		PAGES = {411--449},
		ISSN = {0010-3616},
		URL = {http://projecteuclid.org/euclid.cmp/1104162092},
	}
	
	\bib{witten1990two}{article}{
		AUTHOR = {Witten, Edward},
		TITLE = {Two-dimensional gravity and intersection theory on moduli
			space},
		BOOKTITLE = {Surveys in differential geometry ({C}ambridge, {MA}, 1990)},
		PAGES = {243--310},
		PUBLISHER = {Lehigh Univ., Bethlehem, PA},
		YEAR = {1991},
	}
\end{biblist}
\end{bibdiv}
\end{document}